\documentclass[11pt]{amsart} % change to \documentclass[a4paper,11pt]{amsart} if necessary. The current one is suitable for legal size paper.
\usepackage[lmargin=1in,rmargin=1in,tmargin=1in,bmargin=1in]{geometry}
\usepackage[ps,all,arc,rotate]{xy}
\usepackage{graphicx, float, epstopdf}
\usepackage{hyperref, color}
\usepackage{centernot}
\usepackage{fancyhdr}
\usepackage{amsfonts,amssymb,amsmath,amsthm}
\usepackage{bm}
\usepackage[usenames,dvipsnames]{xcolor}
\usepackage{datetime}
\numberwithin{equation}{section}
\numberwithin{figure}{section}
\allowdisplaybreaks[4]                        %allows breaking multiline environments in amsmath commands
 
\newtheorem{lemma}{Lemma}[section]
\newtheorem{theorem}{Theorem}[section]

\theoremstyle{definition}

\newtheorem{remark}{Remark}[section]
\newcommand{\Z}{\mathbb{Z}}

\newcommand{\R}{\mathbb{R}}
\newcommand{\C}{\mathbb{C}}
\newcommand{\N}{\mathbb{N}}
\newcommand{\real}{\operatorname{Re}}

\newcommand{\SL}{\operatorname{SL}}

\begin{document}
\title[On mean values of mollifiers and $L$-functions]{On mean values of mollifiers and $L$-functions \\ associated to primitive cusp forms}
\author{Patrick K\"{u}hn}
\address{Institut f\"{u}r Mathematik, Universit\"{a}t Z\"{u}rich, Winterthurerstrasse 190, CH-8057 Z\"{u}rich, Switzerland}
\email{patrick.kuehn@math.uzh.ch} 
\author{Nicolas Robles}
\address{Department of Mathematics, University of Illinois, 1409 West Green Street, Urbana, IL 61801, United States}
\curraddr{\textsc{Department of Mathematics, Harvard University, 1 Oxford St, Cambridge, MA 02138, United States}}
\email{nirobles@illinois.edu} 
\author{Dirk Zeindler}
\address{Department of Mathematics and Statistics, Lancaster University, Fylde College, Bailrigg, Lancaster LA1 4YF, United Kingdom}
\email{d.zeindler@lancaster.ac.uk}
\maketitle
\begin{abstract}
We study the second moment of the $L$-function associated to a holomorphic primitive cusp form of even weight perturbed by a new family of mollifiers. This family is a natural extension of the mollifers considered by Conrey and by Bui, Conrey and Young. As an application, we improve the current lower bound on critical zeros of holomorphic primitive cusp forms.
\end{abstract}
% 
%\tableofcontents
%%%%%%%%%%%%%%%%%%%%%%%%%%%%%%%%%%%%%%%%%%%%%%%%%%%%%%%%%%%%%%%%%%%%%%%%%%%%%%%%%%%%%%%%%%%%%%%%%%%%%%%
\section{Introduction}

% \subsection{Review of $L$-functions}
\subsection{Cusp forms and associated $L$-function}
\label{sec:cusp_form}
Let $\mathbb{H} = \{ x + iy, x \in \R, \; y>0\}$. A modular form of weight $k$ for the congruence subgroup of a square-free integer $N$,
\[
\Gamma_0(N) = \left\{ \left( \begin{array}{*{20}{c}}
a&b \\ 
c&d 
\end{array} \right) \in \SL(2,\Z) \ \bigg| \ c \equiv 0 \; \bmod N \right\},
\]
is a complex valued function $f : \mathbb{H} \to \C$ such that:
\begin{itemize}
	\item $f$ is holomorphic;
	\item $(f|_k \gamma)(z):= (cz + d)^{-k} f(\gamma z) = f(z)$ for each $\gamma \in \Gamma_0(N)$;
	\item $f$ is holomorphic at all cusps of $\Gamma_0(N)$ (meaning that the Fourier series at those cusps is a Taylor series in $q := e^{2 \pi i z}$). The cusps are given by $\gamma(\infty) = \frac{a}{c}$ where
\[
	\gamma = \begin{pmatrix} a & b \\ c & d \end{pmatrix}
\] 
	is an element of $\Gamma_0(N) \setminus \SL(2,\Z)$.
\end{itemize}
Additionally, $f$ is a cusp form if it is a modular form and if it vanishes at all cusps of $\Gamma_0(N)$.\\ 

Let $f$ denote a primitive cusp form of even weight $k$. The Fourier expansion of $f$ at the cusp $\infty$ is given by
\begin{align} \label{fourierexpansion}
f(z) = \sum_{n \ge 1} \lambda_{f}(n)n^{(k-1)/2} e^{2 \pi i n z},
\end{align}
for every complex number $z$ in the upper half-plane $\mathbb{H}$. The arithmetic normalization is $\lambda_f(1)=1$. The Fourier coefficients $\lambda_f(n)$ satisfy the multiplicative relations
\begin{align} \label{recursive1}
\lambda_f(n) \lambda_f(m) = \sum_{\substack{ d|(m,n) \\ (d,N)=1}} \lambda_f \bigg( \frac{mn}{d^2} \bigg)
\quad \textnormal{and} \quad
\lambda_f(mn) = \sum_{\substack{ d|(m,n) \\ (d,N)=1}} \mu(d) \lambda_f \bigg( \frac{m}{d} \bigg)\lambda_f \bigg( \frac{n}{d} \bigg),
\end{align}
for all positive integers $m$ and $n$. Here $\mu(n)$ is the M\"{o}bius function. For $\sigma:=\real(s)>1$, we consider
\begin{align}
L(f,s) 
= 
\sum_{n \ge 1} \frac{\lambda_f(n)}{n^s} 
= 
\prod_{p} \bigg(1 - \frac{\lambda_f(p)}{p^s} + \chi_0(p) \frac{1}{p^{2s}} \bigg)^{-1} = \prod_{p} \bigg(1 - \frac{\alpha_f(p)}{p^s} \bigg)^{-1} \bigg(1 - \frac{\beta_f(p)}{p^s} \bigg)^{-1},
\label{eq:def_L_f}
\end{align}
which is an absolutely convergent and non-vanishing Dirichlet series. In the Euler product $\chi_0$ denotes the trivial character modulo $N$. Here $\alpha_f(p)$, $\beta_f(p)$ are the complex roots of the equation $X^2 - \lambda_f(p)X + \chi_0(p) = 0$ and they are called Satake parameters. The function
\[\Lambda (f,s) = {\bigg( {\frac{{\sqrt N }}{{2\pi }}} \bigg)^s}\Gamma \left( {s + \frac{{k - 1}}{2}} \right)L(f,s) = {L_\infty }(f,s)L(f,s)\]
is called the completed $L$-function of $L(f,s)$. It can be extended to a holomorphic function on $\C$ and it satisfies the functional equation
\[
\Lambda(f,s) = \varepsilon(f) \Lambda(f,1-s),
\]
where $\varepsilon(f) = \pm 1$ is the sign of the $L$-function. The sign is real because the $L$-function is self-dual. We also use in the following pages the following function
\[
\chi_f(s) := \varepsilon(f) \frac{L_{\infty}(f,1-s)}{L_{\infty}(f,s)},
\]
so that $\chi_f(s) L(f,1-s) = L(f,s)$. The duplication formula of $\Gamma(s)$ allows us to write 
\[{L_\infty }(f,s) = {\left( {\frac{{{2^k}}}{{8\pi }}} \right)^{1/2}}{\bigg( {\frac{{\sqrt N }}{{2\pi }}} \bigg)^s}\Gamma \left( {\frac{s}{2} + \frac{{k - 1}}{4}} \right)\Gamma \left( {\frac{s}{2} + \frac{{k + 1}}{4}} \right).\]
It is well-known, following analogies of the Riemann zeta-function, that the non-trivial zeros $\rho_f = \beta_f + i\gamma_f$ of $L(f,s)$ are located inside the critical strip $0 < \beta_f < 1$. 

\subsection{Rankin-Selberg convolution}

The Rankin-Selberg convolution of two $L$-functions coming from primitive cusp forms $f$ and $g$ is the $L$-function defined by 
\[L(f \otimes g,s) = L(\chi,2s)\sum\limits_{n = 1}^\infty  {\frac{{{\lambda _f}(n){\lambda _g}(n)}}{{{n^s}}}}  = \prod\limits_p \prod\limits_{i = 1}^2 \prod\limits_{j = 1}^2 {{{\left( {1 - \frac{{{\alpha _{f,i}}(p){\alpha _{g,j}}(p)}}{{{p^s}}}} \right)}^{ - 1}}} .  \]
This is an $L$-function of degree 4. For each prime $p$, $\alpha_{f,1}(p)$, $\alpha_{f,2}(p)$ and $\alpha_{g,1}(p)$, $\alpha_{g,2}(p)$ are the roots of the quadratic equations
\[
X^2 - \lambda_f(p) X + \chi_q(p) = 0, \quad \textnormal{and} \quad X^2 - \lambda_g(p) X + \chi_N(p) = 0.
\]  
We may also write
\begin{align*}
L(\chi_0,s) & = \prod_p \bigg(1 - \frac{\chi_0(p)}{p^{s}} \bigg)^{-1}  = \prod_{p \centernot| N} \bigg(1 - \frac{1}{p^{s}} \bigg)^{-1} = \prod_{p | N} \bigg(1 - \frac{1}{p^{s}} \bigg) \zeta(s) =: \zeta^{(N)}(s),
\end{align*} 
so that 
\begin{align} \label{rankinselbergsquare}
L(f \otimes f, s) = \zeta^{(N)}(2s) \sum_{n = 1}^{\infty} \frac{\lambda_f^2(n)}{n^s}
\end{align} 
for $\real(s)>1$. For an $L$-function of degree $2$ we have the unconditional bound 
\begin{align} \label{kimsarnak}
|\lambda_f(n)| \le   \tau(n) n^{\bm{\theta}}, \quad \textnormal{where} \quad \bm{\theta}=7/64,
\end{align}
where $\tau(n)$ is the divisor function, which satisfies $\tau(n) \ll n^{\varepsilon}$ for each $\varepsilon > 0$. 
This bound, which currently holds the record, is due to Kim and Sarnak \cite{kimsarnak}. However, we work primarily with primitive cusp forms and for those we have a much stronger bound.
Indeed, it was proven by Deligne \cite{deligne} that
\begin{align} \label{ramanujanhypothesis}
|\lambda_f(n)| \le  \tau(n).
\end{align}
The Ramanujan hypothesis states that \eqref{ramanujanhypothesis} is also true for all $L$-functions of degree 2, but it is proven only for a few cases.
\subsection{The zeros of the $L$-function}
If $N_f(T)$ denotes the number of critical (or non-trivial) zeros of $L(f,s)$ up to height $0 < \gamma_f <T$, then one can show by the argument principle that \cite[$\mathsection$5]{iwanieckowalski}
\[
N_f(T) = \frac{T}{\pi} \log \frac{N T^d}{(2 \pi e)^d} + O(\log \mathfrak{q}(f,iT)),
\]
for $T \ge 1$ and where $d$ denotes the degree of $L$. Here $\mathfrak{q}$ denotes the analytic conductor
\[
\mathfrak{q}(f,s) = N \mathfrak{q}_\infty(s) = N\prod_{j=1}^d (|s+\kappa_j|+3) ,
\]
where $N \ge 1$ is the conductor or level of $L(f,s)$, see \cite[pp. 93-95]{iwanieckowalski}.\\

Lastly, we will need to know a zero-free region \cite[Theorem 5.10]{iwanieckowalski}. Specifically, we know that provided the Rankin-Selberg convolutions $L(f \otimes f,s)$ and $L(f \otimes \bar{f}, s)$ exist with the latter having a simple pole at $s=1$ and the former being entire if $f \ne \bar{f}$, then there exists an absolute constant $c>0$ such that $L(f,s)$ has no zeros in the region
\begin{align} \label{zerofreeinequality}
\sigma \ge 1 - \frac{c}{d^4 \log(N(|t|+3))},
\end{align}
except possibly for one simple real zero $\beta_f < 1$, in which case $f$ is self-dual.
%%%%%%%%%%%%%%%%%%%%%%%%%%%%%%%%%%%%%%%%%%%%%%%%%%%%%%%%%%%%%%%%%%%%%%%%%%%%%%%%%%%%%%%%%%%%%%%%%%%%%%%
\subsection{Mollifiers}
Let $Q$ be a polynomial with complex coefficients satisfying $Q(0)=1$ and $Q(x)+Q(1-x) = \operatorname{constant}$. 
Set
\begin{align}
V(s) = Q \bigg(-\frac{1}{2 \log T}\frac{d}{ds} \bigg) L(f,s),
\end{align}
where, for large $T$, we set $L= \log T$. Moreover, let $P(x) = \sum_j a_j x^j$ be a polynomial satisfying $P(0)=0$ and $P(1)=1$, and let $M_1 = T^{\nu_1 -\epsilon}$ where 
\begin{align} 
\label{boundonu}
0 < \nu_1 < \frac{1-2\bm{\theta}}{4 + 2\bm{\theta}},
\end{align}
 with $\bm{\theta}$ as in \eqref{kimsarnak}.
For convenience we adopt the notation
\[
P[n] = P \bigg( \frac{\log M/n}{\log M} \bigg)
\]
for $1 \le n \le M$. By convention, we set $P[x]=0$ for $x \ge M$. A mollifier $\psi$ is a Dirichlet polynomial that approximates the function $(L(f,s))^{-1}$ on the critical line. One of the first mollifiers, introduced by Levinson \cite{levinson} and Conrey \cite{conrey83a, conrey89} is (in the context of $L$-functions) 
\begin{align} \label{firstmollifier}
\psi_1(s) = \sum_{h \le M_1} \frac{\mu_{f}(h)  h^{\sigma_0-1/2}}{h^s}P_1[h],
\end{align}
with $P_1(0)=0$ and $P_1(1)=1$ and here $\mu_f(h)$ is given by
\[
\sum_{h=1}^\infty \frac{\mu_f(h)}{h^s} = \frac{1}{L(f,s)}
\]
for $\real(s)>1$ and $\sigma_0 = 1/2 - R/L$. Here $R$ is a bounded positive number to be chosen later and $M_1$ is the length of the mollifier. It is well-known that the main idea behind the choice of $\psi_1(s)$ in \eqref{firstmollifier} is to replicate the behavior of $1/L(f,s)$ in the mean value integral 
\begin{align} \label{meanvalueI}
I=\int_1^T |V \psi(\sigma_0 + it)|^2 dt,
\end{align}
and to minimize the integral in this way. In \cite{bcy}, Bui, Conrey and Young attached a second piece to this mollifier, i.e. they worked with
\[
\psi(s) = \psi_1(s) + \psi_2(s),
\]
where $\psi_1(s)$ is the same as in \eqref{firstmollifier} and $\psi_2(s)$ has the shape
\[
\psi_2(s) = \chi_f(s+1/2-\sigma_0)\sum_{hk \le M_2} \frac{\mu_{f,2}(h)h^{\sigma_0-1/2}k^{1/2-\sigma_0}}{h^s k^{1-s}} P_2[hk],
\]
with $M_2 = T^{\nu_2}$ where $0< \nu_2 \le \nu_1$. In this case $P_2$ is some other polynomial such that $P_2(0)=P_2'(0)=P_2''(0)=0$. The terms $\mu_{f,2}(h)$ are given by the Dirichlet convolution $(\mu_f * \mu_f)(h)$. By convention $\mu_{f,1} = \mu_{f}$. The reasoning behind this choice comes from the formal calculation
\[
\chi _f(s)\sum\limits_{h,k = 1}^\infty  {\frac{{{\mu _{f,2}}(h)}}{{{h^s}{k^{1 - s}}}}}  = \frac{{{\chi _f}(s)L(f,1 - s)}}{{{L^2}(f,s)}} = \frac{1}{L(f,s)}.
\]
This indicates that, up some extent, the second piece $\psi_2(s)$ also replicates the behavior of $1/L(f,s)$. Set $\lambda_f^{*2}(k) = (\lambda_f * \lambda_f)(k)$ and $\mu_{f,3}(h) = (\mu_f * \mu_f * \mu_f)(h)$. With this in mind, we can also claim that 
\[
\psi_3(s) = \chi_f^2(s+1/2-\sigma_0)\sum_{hk \le M_3} \frac{\mu_{f,3}(h) \lambda_f^{*2} (k) h^{\sigma_0-1/2}k^{1/2-\sigma_0}}{h^s k^{1-s}} P_3[hk],
\]
for an appropriate $P_3$, is a suitable mollifier since (formally)
\[
{\chi_f^2}(s)\sum\limits_{h,k = 1}^\infty  \frac{\mu_{f,3}(h)\lambda_f^{*2}(k)}{{{h^s}{k^{1 - s}}}}  = \frac{{{\chi_f^2}(s)L^2(f,1 - s)}}{{{L^3}(f,s)}} = \frac{1}{{L(f,s)}}.
\]
Naturally, this welcomes a higher order generalization. Suppose that $\ell \in \N$. This idea may be extended by taking
\begin{align} \label{generalmollifier}
{\psi _\ell }(s) = \chi _f^{\ell-1} (s + \tfrac{1}{2} - {\sigma _0})\sum\limits_{hk \le {M_\ell }} {\frac{{{\mu _{f,\ell }}(h)\lambda _f^{ * \ell  - 1}(k){h^{{\sigma _0} - 1/2}}{k^{1/2 - {\sigma _0}}}}}{{{h^s}{k^{1 - s}}}}{P_\ell }[hk]},
\end{align}
where $\mu_{f,\ell}(n)$ is given by
\begin{align}
 \frac{1}{L^{\ell}(f,s)} = \sum_{n=1}^{\infty} \frac{\mu_{f,\ell}(n)}{n^{s}} \quad \textnormal{for} \quad \real(s)>1,
 \label{eq:def_mu_f}
\end{align}
and $\lambda_f ^{* k}$ stands for convolving $\lambda_f$ with itself exactly $k$ times; in other words $\lambda_f ^{* k}(n) = (\lambda_f * \cdots * \lambda_f)(n)$. The conditions on $P_\ell$ are
\begin{align} \label{conditiononP}
P_\ell(0) &= 0 \quad \textnormal{and} \quad P_{\ell}(1)=1, \quad \textnormal{when} \quad \ell=1, \nonumber \\
P_\ell(0) &= P_\ell'(0) = P_\ell''(0) = \cdots = P_\ell^{(\ell(\ell-1))}(0) = 0 , \quad \textnormal{when} \quad \ell > 1,
\end{align}
where $P^{(m)}$ denotes the $m$-th derivative of $P$. Moreover $M_\ell = T^{\nu_\ell}$, where $0 < \nu_\ell \le \nu_1$. Another formal calculation shows that indeed one has
\begin{align} \label{formalgeneral}
\chi _f^{\ell  - 1}(s)\sum\limits_{h,k = 1}^\infty  {\frac{{{\mu _{f,\ell }}(h)\lambda _f^{ * \ell  - 1}(k)}}{{{h^s}{k^{1 - s}}}}}  = \frac{{\chi _f^{\ell  - 1}(s){L^{\ell  - 1}}(f,1 - s)}}{{{L^\ell }(f,s)}} = \frac{1}{{L(f,s)}}.
\end{align}
Clearly, when $\ell = 1$ (by the use of \eqref{reductionell1} below) and $\ell=2$, the pieces of Conrey and Levinson and of Bui, Conrey and Young follow as special cases, respectively. Consequently, the mollifier we will be working with is given by
\begin{align} \label{generalsumL}
\psi (s) = \sum\limits_{\ell  = 1}^{\mathcal{L}} \psi _\ell(s) ,
\end{align}
where $\mathcal{L} \in \N$ is of our choice.
The reason behind this choice is due that one may think of $\psi_1(s)$ as the main term of the mollifier and of $\{\psi_\ell(s) \}_{\ell \ge 2}$ as the perturbations to the main piece.
\subsection{Proportions of zeros on the critical line}
In this paper we revise the techniques of \cite{bcy} for a mollifier consisting of several pieces. 
This approach is extremely general.
As an application, we modestly increase the current proportion of critical zeros of $L(f,s)$ and clarify the situation of simple critical zeros. 
Our technique is based on developments by Conrey and Snaith \cite{conreysnaithratios} on ratios conjectures, 
and by Conrey, Farmer and Zirnbauer \cite{conreyfarmerzirnbauer} on autocorrelation of ratios of $L$-functions.\\

Let us define $N_{f,0}(T)$ to be the number of non-trivial zeros of $L(f,s)$ up to height $T>0$ such that $\real(s) = \frac{1}{2}$ and $N_f(T)$ the number of zeros inside the rectangle $0 < \real(s)  <1$ also up to height $T$. We moreover set
\[
\kappa_f = \mathop {\lim \inf }\limits_{T \to \infty } \frac{N_{f,0}(T)}{N_f(T)}.
\]

In 2015, Bernard \cite{bernard} revisited Young's paper \cite{youngsimple} and adapted it to modular forms (see \cite[Proposition 5]{bernard} 
as well as \cite{farmer,hafner1,hafner2,rezvyakova}). In particular, if one applies Littlewood's lemma, and then the arithmetic and geometric mean inequalities, one arrives at
\begin{align} \label{kappaineq}
\kappa_f \ge \mathop {\lim \sup}\limits_{T \to \infty} \bigg(1-\frac{1}{2R}\log \bigg(\frac{1}{T}\int_1^T |V\psi(\sigma_0+it)|^2 dt \bigg) \bigg),
\end{align}
where $\sigma_0 = 1/2 - R/L$ with $R$ a bounded positive number of our choice.\\

For holomorphic primitive cusp forms of even weight, square-free level and trivial character, 
Bernard's results \cite[p. 203]{bernard} are that $\kappa_f \ge 6.93\%$. 
  For this result, Bernard requires the Ramanujan hypothesis ($ \bm{\theta} =0$ in \eqref{kimsarnak}), which is proven in this case.
  If one were to use instead the weaker bound proven by Kim and Sarnak ($ \bm{\theta} =7/64$), one would get only $\kappa_{f} \ge 2.97\%$. 
Unfortunately, as mentioned on \cite[p. 203]{bernard}, the size of the mollifier, even under the Ramanujan hypothesis, 
is too small to establish results for simple zeros on the critical line. Further details can be found in $\mathsection$5.\\

Because of \eqref{generalmollifier} and \eqref{formalgeneral}, 
it is clear that the same mechanism that makes $\psi_2$ be a useful mollifier will also make $\psi_3, \psi_4, \cdots$ useful. Moreover, from \cite[p. 38]{bcy} and \cite[p. 310]{rrz01} we know that
\[
|\psi_2(s)| \ll \sqrt{t} \bigg( \frac{M_2}{t} \bigg)^\sigma L^2, \quad |\psi_3(s)| \ll t \bigg( \frac{M_3}{t^2} \bigg)^\sigma L^4, 
\]
and so on. Therefore, $\log \psi(s)$, where $\psi(s)$ is given by \eqref{generalsumL}, 
is analytic and a valid mollifier that replicates the behavior of $1/L(f,s)$ in a certain region of the complex plane. See \cite[$\mathsection$10]{rrz01} for further details.\\

Unfortunately, the presence of the powers of $\chi$ in \eqref{generalmollifier} decreases the usefulness of the additional pieces as the exponential decay of the pre-factor $\chi$ overwhelms the Dirichlet polynomial. 
Communications with K. Sono \cite{sono}, who has computed the effect of the additional $\psi_\ell$ pieces for the Riemann zeta-function, seem to indicate that the contribution of $\psi_3$ will be smaller than $10^{-4}$.\\

Moreover, we have full control over the additional pieces $\psi_\ell$ via the coefficients of the polynomial inside the Dirichlet series. 
We may thus turn them off or finely calibrate them to suit our needs. Therefore, these additional pieces cannot be harmful.\\

Furthermore, adding just another perturbation (of a different nature) to $\psi_1$ and handling the errors produced by the off-diagonal terms carefully has produced an increment of $0.421\%$ for the case of the Riemann zeta-function, see \cite{pr01}.\\

Finally, it is worth mentioning a situation in which these perturbations are very helpful. Following \cite[p. 215]{farmer}, if we place ourselves in the context of the Riemann zeta-function and conjecturally take $\nu_1 \to 1$ (currently $\nu_1 < \frac{4}{7}$ is the best one can do, \cite{conrey89}), then one obtains that at least $58.65\%$ of non-trivial zeros of $\zeta(s)$ are on the critical line. If we were to add $\psi_2(s)$ and work with $\psi_1(s)+\psi(2)$, in other words with $\mathcal{L}=2$, and conjecturally take $\nu_1, \nu_2 \to 1$ ($\nu_2 < \frac{1}{2}$ is currently the best as proved in \cite{bcy}), then we show in $\mathsection$5 that this figure increases to $60.586\%$.

\subsection{Numerical evaluations}
We will improve Bernard's proportions a little bit by taking $\mathcal{L}=2$ in \eqref{generalsumL}. 
As a consequence of our results we can now establish the following.
\begin{theorem} \label{theoremofkappa}
One has
\begin{equation}
\kappa_f \ge 
\begin{cases}
2.97607\%, & \mbox{unconditionally},  \nonumber \\ 
6.93872\%, & \mbox{under the Ramanujan conjecture}. 
\end{cases}
\end{equation} 
\end{theorem}
The underlying polynomials and optimized value of $R$ can be found in $\mathsection$5.
\subsection{Proof techniques}
The argumentation used in this paper is based on the techniques introduced in \cite{bcy, youngsimple} for the Riemann zeta-function which in turn are borrowed from \cite{conreyfarmerzirnbauer, conreysnaithratios}. 
% \textcolor{red}{which was first time adapted to primitive cusp forms and Mass forms by Blomer \cite{blomer1}.}
For this, we split the occurring expressions into the diagonal and off-diagonal contributions. The diagonal contributions can handled by generalizing the results of \cite{bcy}. However, the estimation of the off-diagonal contributions is much more challenging and it cannot be done in the same way as for the zeta-function.
For this we use the pioneering work of Blomer on shifted convolution sums on average in \cite{blomer2}, and its extension by Bernard in \cite[pp. 208-217]{bernard}, see also \cite{blomer1}. The specific details are in $\mathsection$3.\\

An important difference with the Riemann zeta-function is that the lengths $M_\ell$ of the mollifiers are in our setting much shorter.
For Riemann zeta-function one can use $T^{\vartheta}$ with $\vartheta<4/7$ for $M_1$. 
This was an accomplishment of Conrey \cite{conrey89} who used the work of Deshouillers and Iwaniec \cite{deshouillersIwaniec1,deshouillersIwaniec2}, see also \cite{bchb}. For primitive cups forms, we can use only
\begin{align}
 M_1 = T^{\frac{1-2\bm{\theta}}{4+2\bm{\theta}}-\epsilon}
 \quad \textnormal{and} \quad  M_\ell = T^{\frac{1-2\bm{\theta}}{4+6\bm{\theta}}-\epsilon} 
 \quad \text{for} \quad \ell \ge 2
\end{align}
and $\varepsilon>0$ small. As the Ramanujan hypothesis is proven in our situation, we have $\bm{\theta} =0$ and we can use
\begin{align}
M_\ell = T^{\frac{1}{4}-\epsilon} \quad \textnormal{for} \quad \ell \ge 1.
\end{align}
We thus see that the lengths of the mollifiers are much shorter than for the zeta-function and this results in much smaller lower bounds for $\kappa_f$, see $\mathsection$4 for further details.\\

Lastly, as remarked by Farmer \cite[p. 216]{farmer}, our improvements above are consistent with his observations that it is substantially harder to work with $L$-functions of higher degrees. Indeed, the efficiency of the mollifier is severely limited by the range of its length as the degree increases.
\section{Results}
The method sketched in \cite{bcy, rrz01} to deal with multiple piece mollifiers in the mean value integral \eqref{meanvalueI} carries through and our main results are as follows.
\begin{theorem} \label{unsmoothed}
% Suppose that $\nu_\ell = \nu/2 - \varepsilon$ where $\nu = \frac{1-2\bm{\theta}}{4+2\bm{\theta}}$ for $\varepsilon > 0 $ small. 
Suppose that $\nu_1 = \frac{1-2\bm{\theta}}{4+2\bm{\theta}} - \varepsilon$ and $\nu_\ell = \frac{1-2\bm{\theta}}{4+6\bm{\theta}} - \varepsilon$ for $\ell\geq2$ and $\varepsilon > 0 $ small. 
Then
\[
\frac{1}{T} \int_1^T |V \psi(\sigma_0 + it)|^2 dt = c(\{P_\ell\}_{\ell=1}^{\mathcal{L}},R,\{\nu_\ell\}_{\ell=1}^{\mathcal{L}}) + o(1),
\]
where
\[
c(\{P_\ell\}_{\ell=1}^{\mathcal{L}},Q,R,\{\nu_\ell\}_{\ell=1}^{\mathcal{L}}) = \sum\limits_{{k_1} + {k_2} +  \cdots  + {k_\mathcal{L}} = 2} \binom{2}{k_1,k_2,\cdots ,k_\mathcal{L}} \prod\limits_{1 \le \ell  \le \mathcal{L}} (c_{\ell,\ell+k_\ell})^{k_\ell} 
\]
and the different $c_{i,j}$ are given by \eqref{cellellformula} and \eqref{cellell1formula}.
\end{theorem}
\begin{remark}
The following two points ought to be noted.
\begin{enumerate}
 \item[a)] We need in our computations of the cross-term  $I_{\ell,\ell+1}$ in Theorem~\ref{cellell1theorem} the condition $\nu_\ell + \nu_{\ell+1} <1$.
	   We also need in our computations of the cross-term  $I_{\ell,\ell+j}$ in Theorem~\ref{cellelljtheorem} with $j\geq 2$ the condition $\nu_{\ell} +  \nu_{\ell+j} < 2(j-1)$.
	   As $\bm{\theta}\geq0$, we get $\nu\leq 1/4$ and thus both conditions are automatically fulfilled.
	   However, the mollifiers in this paper can be adapted to the study of the Riemann zeta-function and other $L$-functions.
	   For those other functions, one has to check carefully if these conditions are fulfilled. 
 \item[b)] There is no need to explicitly compute $c_{\ell,\ell+j}$ where $j=2,3,\cdots$ since the contribution of the associated integral 
	   is $O(TL^{-1+\varepsilon})$, see Theorem \ref{cellelljtheorem} below.
\end{enumerate}
\end{remark}
%%%%%%%%%%%%%%%%%%%%%%%%%%%%%%%%%%%%%%%%%%%%%%%%%%%%%%%%
\subsection{The smoothing argument}
The idea of smoothing the mean value integrals was worked out in \cite{bernard,bcy,youngsimple} and it makes the following computations more convenient. Let $w(t)$ be a smooth function satisfying the following properties:
\begin{enumerate}
\item[(a)] $0 \le w(t) \le 1$ for all $t \in \R$,
\item[(b)] $w$ has compact support in $[T/4,2T]$,
\item[(c)] $w^{(j)}(t) \ll_j \Delta^{-j}$, for each $j=0,1,2,\cdots$ and where $\Delta = T/L$.
\end{enumerate}
Note that for the Fourier transform of $w$, we have $\widehat{w}(0)=T/2 +O(T/L)$. This allows us to re-write Theorem \ref{unsmoothed} as follows.
\begin{theorem} \label{smoothed}
Let $\nu_\ell$ for $\ell\geq 1$ be as in Theorem \textnormal{\ref{unsmoothed}}.
For any $w$ satisfying conditions \textnormal{(a)}, \textnormal{(b)} and \textnormal{(c)} and $\sigma_0 = 1/2-R/L$,
\[
\int_{-\infty}^{\infty} w(t)|V \psi(\sigma_0 + it)|^2 dt = c(\{P_\ell\}_{\ell=1}^{\mathcal{L}},Q,R,\{\nu_\ell\}_{\ell=1}^{\mathcal{L}})\widehat{w}(0) + O(T/L),
\]
uniformly for $R \ll 1$, where 
\[
c(\{P_\ell\}_{\ell=1}^{\mathcal{L}},Q,R,\{\nu_\ell\}_{\ell=1}^{\mathcal{L}}) = \sum\limits_{{k_1} + {k_2} +  \cdots  + {k_\mathcal{L}} = 2} \binom{2}{k_1,k_2,\cdots ,k_\mathcal{L}} \prod\limits_{1 \le \ell  \le \mathcal{L}} (c_{\ell,\ell+k_\ell})^{k_\ell}, 
\]
where the different $c_{i,j}$ given by \eqref{cellellformula} and \eqref{cellell1formula}.
\end{theorem}
The technique of a multi-piece mollifier was developed in \cite{bcy,feng}. In \cite{rrz01} a 4-piece mollifier was handled. The idea is to open the square in the integrand
\begin{align}
\int |V\psi|^2 &= \int |V\psi_1|^2 + \int |V|^2\psi_1 \overline{\psi_2} + \int |V|^2\overline{\psi_1} \psi_2 + \cdots + \int |V\psi_{\mathcal{L}}|^2 \nonumber \\
&=\sum_{1 \le \ell \le \mathcal{L}} \{ I_{\ell,\ell} + I_{\ell,\ell+1} + \overline{I_{\ell+1,\ell}} + I_{\ell,\ell+j} + \overline{I_{\ell+j,\ell}} \}, \nonumber
\end{align}
for $j=2,3,4,\cdots$. We will compute these integrals in the next sections. The integral $I_{\ell,\ell+1}$ is asymptotically real, thus $I_{\ell+1,\ell}$ follows from $I_{\ell,\ell+1}$, i.e. $I_{\ell,\ell+1} \sim \overline{I_{\ell+1,\ell}}$.
\subsection{The main terms}
The main terms coming from integrals $I_{\ell,\ell}$, $I_{\ell,\ell+1}$ and $I_{\ell,\ell+j}$ where $j=2,3,4,\cdots$ are now stated as theorems.
\begin{theorem} \label{cellelltheorem}
Let $\ell \in \N$. 
Let $\nu_\ell$ for $\ell\geq 1$ be as in Theorem \textnormal{\ref{unsmoothed}}. Then we have for $\ell\geq 1$
\[
\int_{-\infty}^{\infty} w(t)|V \psi_{\ell}(\sigma_0+it)|^2 dt \sim c_{\ell,\ell}(P_\ell,Q,R,\nu_\ell)\widehat{w}(0) + O(T L^{-1+\varepsilon})
\]
uniformly for $R \ll 1$, where
\begin{align} \label{cellellformula}
  c_{\ell ,\ell } &= \frac{1}{\Gamma^2(\ell-1)}\frac{{{2^{2\ell (\ell  - 1)}}}}{{({\ell ^2} + {{(\ell  - 1)}^2} - 1)!}}\frac{{{d^{2\ell }}}}{{d{x^\ell }d{y^\ell }}} \nonumber \\
   &\quad \times \bigg[\int_0^1 {\int_0^1 {\int_0^1 {\int_0^1 {\left( {\frac{2}{{{\nu _\ell }}} + (x + y - v(y + r) - u(x + r))} \right)} } } } {(1 - r)^{{\ell ^2} + {{(\ell  - 1)}^2} - 1}} \nonumber \\
   &\quad \times {e^{ - \frac{\nu_\ell}{2}R[x + y - v(y + r) - u(x + r)]}}{e^{2Rt[1 + \frac{\nu_\ell}{2}(x + y - v(y + r) - u(x + r))]}} \nonumber \\
   &\quad \times Q\bigg(\frac{\nu _\ell }{2}( - x + v(y + r)) + t\bigg(1 + \frac{\nu _\ell }{2}(x + y - v(y + r) - u(x + r))\bigg)\bigg) \nonumber \\
   &\quad \times Q\bigg(\frac{\nu _\ell }{2}( - y + u(x + r)) + t\bigg(1 + \frac{\nu _\ell }{2}(x + y - v(y + r) - u(x + r))\bigg)\bigg) \nonumber \\
   &\quad \times {(x + r)^{\ell  - 1}}{(y + r)^{\ell  - 1}}{u^{\ell  - 2}}{v^{\ell  - 2}} \nonumber \\
   &\quad \times P_\ell ^{(\ell (\ell  - 1))}((1 - u)(x + r))P_\ell ^{(\ell (\ell  - 1))}((1 - v)(y + r))dtdrdudv \bigg]_{x = y = 0}. 
\end{align}
\end{theorem}
\begin{remark}
The case $\ell=1$ has to be handled with a certain amount of care as it superficially seems divergent due to the presence of $u^{\ell-2}$ and $v^{\ell-2}$ in the integrands. This is taken care of by $\frac{1}{\Gamma^2(\ell-1)}$ in the denominator of the first line. By the use of \eqref{limitingconrey} below, the term $\frac{1}{\Gamma^2(\ell-1)}$ cancels out the integrals with respect to $u$ and with respect to $v$, leaving us with
\[
c_{1,1}(P,Q,R,\nu_1) = 1 + \frac{1}{\nu_1} \int_0^1 \int_0^1 e^{2Rv} \bigg[ \frac{d}{dx} (e^{R \nu x} Q(v+\nu x)P(x + u))|_{x=0} \bigg]^2 du dv,
\] 
which is precisely the term recovered by Conrey \cite{conrey89} and Bernard \cite{bernard}, see also \eqref{specialell1} and \cite[p. 544]{youngsimple}. 
\end{remark}
\begin{theorem} \label{cellell1theorem}
Let $\ell \in \N$ and $\nu_\ell$ be as in Theorem \textnormal{\ref{unsmoothed}}. Then
\[
\int_{-\infty}^{\infty} w(t)V \psi_{\ell} \overline{\psi_{\ell+1}}(\sigma_0+it) dt \sim c_{\ell,\ell+1}(P_\ell,P_{\ell+1},Q,R,\nu_{\ell},\nu_{\ell+1})\widehat{w}(0) + O(T L^{-1+\varepsilon})
\]
uniformly for $R \ll 1$, where
\begin{align} \label{cellell1formula}
  c_{\ell ,\ell  + 1} &= \frac{{{2^{2{\ell ^2}}}}}{{(2{\ell ^2} - 1)!}}{\left( {\frac{{{\nu _{\ell  + 1}}}}{{{\nu _\ell }}}} \right)^{\ell (\ell  + 1)}}{e^R}\frac{{{d^{2\ell }}}}{{d{x^\ell }d{y^\ell }}} \bigg[  \iint_{\substack{a + b \le 1 \\ a, b \ge 0}} {\int_0^1 {{u^{2\ell }}{{(1 - u)}^{2{\ell ^2} - 1}}{e^{R[\frac{\nu _\ell}{2}(y - x) + u\frac{\nu_{\ell  + 1}}{2}(a - b)]}}} } \nonumber \\
   &\quad \times Q\bigg( \frac{- x{\nu _\ell } + au{\nu _{\ell  + 1}}}{2}\bigg)Q\bigg(1+\frac{ y{\nu _\ell } - bu{\nu _{\ell  + 1}}}{2}\bigg) \nonumber \\
   &\quad \times P_\ell ^{(\ell (\ell  - 1))}\left( {x + y + 1 - (1 - u)\frac{{{\nu _{\ell  + 1}}}}{{{\nu _\ell }}}} \right)P_{\ell  + 1}^{(\ell (\ell  + 1))}((1 - a - b)u){(ab)^{\ell  - 1}}dudadb \bigg]_{x = y = 0}.  
\end{align}
\end{theorem}
\begin{theorem} \label{cellelljtheorem}
Let $\ell \in \N$, $j\geq2$ and $\nu_\ell$ and $\nu_{\ell+j}$ be as in Theorem \textnormal{\ref{unsmoothed}}. Then
\[
\int_{-\infty}^{\infty} w(t)V \psi_{\ell} \overline{\psi_{\ell+j}}(\sigma_0+it) dt \ll_\ell TL^{-1+\varepsilon}
\]
uniformly for $\alpha,\beta \ll L^{-1}$.
\end{theorem}
The smoothing argument is helpful because we can easily deduce Theorem \ref{unsmoothed} from Theorem \ref{smoothed} and so on. By having chosen $w(t)$ to satisfy conditions (a), (b) and (c) and in addition to being an upper bound for the characteristic function of the interval $[T/2,T]$, and with support $[T/2-\Delta,T+\Delta]$, we get
\[
\int_{T/2}^T |V \psi(\sigma_0+it)|^2 dt \le c(P_1,P_\ell,Q,2R,\nu_1/2,\nu_2/2)\widehat{w}(0) + O(T/L).
\]
Note that $\widehat{w}(0)=T/2 + O(T/L)$. We similarly get a lower bound. Summing over dyadic segments gives the full result.
\subsection{The shift parameters $\alpha$ and $\beta$}
Rather than working directly with $V(s)$, we shall instead consider the following three general shifted integrals
\begin{align}
  I_{\ell,\ell}(\alpha ,\beta ) &= \int_{ - \infty }^\infty  w(t)L (f,\tfrac{1}{2} + \alpha  + it)L (f,\tfrac{1}{2} + \beta  - it)\overline {{\psi _\ell}} {\psi _\ell}({\sigma _0} + it)dt , \nonumber \\
  I_{\ell,\ell+1}(\alpha ,\beta ) &= \int_{ - \infty }^\infty  w(t) L(f,\tfrac{1}{2} + \alpha  + it)L (f,\tfrac{1}{2} + \beta  - it)\overline {{\psi _\ell}} {\psi _{\ell+1}}({\sigma _0} + it)dt , \nonumber \\
	I_{\ell,\ell+j}(\alpha ,\beta ) &= \int_{ - \infty }^\infty  w(t) L(f,\tfrac{1}{2} + \alpha  + it)L (f,\tfrac{1}{2} + \beta  - it)\overline {{\psi _\ell}} {\psi _{\ell+j}}({\sigma _0} + it)dt , \nonumber 
\end{align}
for $j=2,3,4,\cdots$. The computation is now reduced to proving the following three lemmas.
\begin{lemma} \label{cellelllemma}
We have
\[
I_{\ell ,\ell } = c_{\ell ,\ell }(\alpha,\beta)\widehat{w}(0) + O(T/L),
\]
uniformly for $\alpha,\beta \ll L^{-1}$, where
\begin{align}
  c_{\ell ,\ell }(\alpha ,\beta ) &= \frac{1}{\Gamma^2(\ell-1)}\frac{{{2^{2\ell (\ell  - 1)}}}}{{({\ell ^2} + {{(\ell  - 1)}^2} - 1)!}}\frac{{{d^{2\ell }}}}{{d{x^\ell }d{y^\ell }}} \nonumber \\
   &\quad \times \bigg[\int_0^1 {\int_0^1 {\int_0^1 {\int_0^1 {{{(1 - r)}^{{\ell ^2} + {{(\ell  - 1)}^2} - 1}}} } } } {u^{\ell  - 2}}{v^{\ell  - 2}} \nonumber \\
   &\quad \times {T^{{\nu _\ell }(\beta (x - v(y + r)) + \alpha (y - u(x + r)))}}{({T^{2 + {\nu _\ell }(x + y - v(y + r) - u(x + r))}})^{ - t(\alpha  + \beta )}} \nonumber \\
   &\quad \times \left( {\frac{2}{{{\nu _\ell }}} + (x + y - v(y + r) - u(x + r))} \right){(x + r)^{\ell  - 1}}{(y + r)^{\ell  - 1}} \nonumber \\
   &\quad \times P_\ell ^{(\ell (\ell  - 1))}((1 - u)(x + r))P_\ell ^{(\ell (\ell  - 1))}((1 - v)(y + r))dtdrdudv \bigg]_{x = y = 0}. \nonumber
\end{align}
\end{lemma}
\begin{lemma} \label{cellell1lemma}
We have
\[
I_{\ell,\ell+1} = c_{\ell,\ell+1}(\alpha,\beta)\widehat{w}(0) + O(T/L),
\]
uniformly for $\alpha,\beta \ll L^{-1}$, where
\begin{align}
  c_{\ell ,\ell  + 1}(\alpha ,\beta ) &= \frac{{{2^{2{\ell ^2}}}}}{{(2{\ell ^2} - 1)!}}{\left( {\frac{{{\nu _{\ell  + 1}}}}{{{\nu _\ell }}}} \right)^{\ell (\ell  + 1)}}\frac{{{d^{2\ell }}}}{{d{x^\ell }d{y^\ell }}} \bigg[ \nonumber \\
   &\quad \times \iint_{\substack{a + b \le 1 \\ a, b \ge 0}} {\int_0^1 {{u^{2\ell }}{{(1 - u)}^{2{\ell ^2} - 1}}{{(M_\ell ^{ - x}M_{\ell  + 1}^{au})}^{ - \alpha }}{{(M_\ell ^yM_{\ell  + 1}^{ - bu}T^2)}^{ - \beta }}} } \nonumber \\
   &\quad \times P_\ell ^{(\ell (\ell  - 1))}\left( {x + y + 1 - (1 - u)\frac{{{\nu _{\ell  + 1}}}}{{{\nu _\ell }}}} \right)P_{\ell  + 1}^{(\ell (\ell  + 1))}((1 - a - b)u){(ab)^{\ell  - 1}}dudadb \bigg]_{x = y = 0}. \nonumber  
\end{align}
\end{lemma}
\begin{lemma} \label{cellelljlemma}
For $j=2,3,4,\cdots$, we have
\[
\int_{-\infty}^{\infty} w(t)V \psi_{\ell} \overline{\psi_{\ell+1}}(\sigma_0+it) dt \ll TL^{-1+\varepsilon}
\]
uniformly for $\alpha,\beta \ll L^{-1}$.
\end{lemma}
To get Theorems \ref{cellelltheorem} and \ref{cellell1theorem} we use the following technique. Let $I_{\star}$ denote either of the integrals in questions, and note that
\[
{I_ {\star} } = Q\left( { - \frac{1}{{2\log T}}\frac{d}{{d\alpha }}} \right)Q\left( { - \frac{1}{{2\log T}}\frac{d}{{d\beta }}} \right){I_{\star}}(\alpha ,\beta ){\bigg|_{\alpha  = \beta  = R/L}}.
\]
Since $I_{\star}(\alpha ,\beta )$ and $c_{\star}(\alpha ,\beta )$ are holomorphic with respect to $\alpha,\beta$ small, the derivatives appearing in the equation above can be obtained as integrals of radii $\asymp L^{-1}$ around the points $-R/L$, using Cauchy's integral formula. Since the error terms hold uniformly on these contours, the same error terms that hold for $I_{\star}(\alpha ,\beta )$ also hold for $I_{\star}$. That the above differential operator on $c_{\star}(\alpha ,\beta )$ does indeed give $c_{\star}$ follows from
\[
Q \bigg( \frac{-1}{2\log T} \frac{d}{d\alpha} \bigg)X^{-\alpha} = Q\bigg( \frac{\log X}{2\log T}\bigg) X^{-\alpha/2}.
\]
Note that from the above equation we get
\begin{align}
  Q\left( {\frac{{ - 1}}{{2\log T}}\frac{d}{{d\alpha }}} \right) & Q\left( {\frac{{ - 1}}{{2\log T}}\frac{d}{{d\beta }}} \right){(M_\ell ^{ - x}M_{\ell  + 1}^{au})^{ - \alpha }}{(M_\ell ^yM_{\ell  + 1}^{ - bu}T)^{ - \beta }} \nonumber \\
   &= Q\bigg( {\frac{{\log M_\ell ^{ - x}M_{\ell  + 1}^{au}}}{{2\log T}}} \bigg)Q\bigg( {\frac{{\log M_\ell ^yM_{\ell  + 1}^{ - bu}T}}{{2\log T}}} \bigg){(M_\ell ^{ - x}M_{\ell  + 1}^{au})^{ - \alpha }}{(M_\ell ^yM_{\ell  + 1}^{ - bu}T)^{ - \beta }} \nonumber \\
   &= Q( - x{\nu _\ell/2 } + au{\nu _{\ell  + 1}}/2)Q(1 + {\nu _\ell }y/2 - bu{\nu _{\ell  + 1}/2}){(M_\ell ^{ - x}M_{\ell  + 1}^{au})^{ - \alpha/2 }}{(M_\ell ^yM_{\ell  + 1}^{ - bu}T)^{ - \beta/2 }}, \nonumber 
\end{align}
as well as
\begin{align}
  Q\left( {\frac{{ - 1}}{{2\log T}}\frac{d}{{d\alpha }}} \right)&Q\left( {\frac{{ - 1}}{{2\log T}}\frac{d}{{d\beta }}} \right){T^{{\nu _\ell }(\beta (x - v(y + r)) + \alpha (y - u(x + r)))}}{({T^{2 + {\nu _\ell }(x + y - v(y + r) - u(x + r))}})^{ - t(\alpha  + \beta )}} \nonumber \\
   &= Q({\nu _\ell/2 }( - x + v(y + r)) + t(2 + {\nu _\ell/2 }(x + y - v(y + r) - u(x + r)))) \nonumber \\
   &\quad \times Q({\nu _\ell/2 }( - y + u(x + r)) + t(2 + {\nu _\ell/2 }(x + y - v(y + r) - u(x + r)))) \nonumber \\
	 &\quad \times {T^{{\nu _\ell/2 }(\beta (x - v(y + r)) + \alpha (y - u(x + r)))}}{({T^{2 + {\nu _\ell/2 }(x + y - v(y + r) - u(x + r))}})^{ - t(\alpha  + \beta )}} .\nonumber 
\end{align}
Hence using the differential operators $Q((-1/2\log T) d/d\alpha)$ and $Q((-1/2\log T) d/d\beta)$ on the last line of $c_{\ell,\ell+1}(\alpha,\beta)$ we get in the integrand
\begin{align}
  e^{2R}  e^{2R[{\nu _\ell }(y - x)/2 + u{\nu _{\ell  + 1}}(a - b)/2]} Q( - x{\nu _\ell/2 } + au{\nu _{\ell  + 1}/2})Q(1 + y{\nu _\ell/2 } - bu{\nu _{\ell  + 1}/2}), \nonumber 
\end{align}
by setting $\alpha = \beta = -R/L$ and using $T^{z/L}=T^{z/\log T}=e^{z}$. Hence Theorem \ref{cellell1theorem} follows. Similarly, when we use the differential operators $Q((-1/2\log T) d/d\alpha)$ and $Q((-1/2\log T) d/d\beta)$ on the integrand of $c_{\ell,\ell}(\alpha,\beta)$ it becomes
\begin{align}
  &e^{ - \frac{\nu _\ell}{2}R[x + y - v(y + r) - u(x + r)]}{e^{2Rt[1 + \frac{\nu _\ell}{2}(x + y - v(y + r) - u(x + r))]}} \nonumber \\
   &\quad \times Q({\nu _\ell }( - x + v(y + r))/2 + t(1 + {\nu _\ell }/2(x + y - v(y + r) - u(x + r)))) \nonumber \\
   &\quad \times Q({\nu _\ell }( - y + u(x + r))/2 + t(1 + {\nu _\ell }/2(x + y - v(y + r) - u(x + r)))), \nonumber  
\end{align}
by the same substitutions. This proves Theorem \ref{cellelltheorem}.
%%%%%%%%%%%%%%%%%%%%%%%%%%%%%%%%%%%%%%%%%%%%%%%%%%%%%%%%%%%%%%%%%%%%%%%%%%%%%%%%%%%%%%%%%%%%%%%%%%%%%%%
\section{Preliminary tools}
The following results are needed throughout the paper. The first lemma is used to compute the "square" terms $I_{\ell,\ell}$. 
We start by quoting a result from Bernard's paper \cite[Lemma 1]{bernard}, who in turn quoted it with small modifications from \cite[Theorem 5.3]{iwanieckowalski}.
\begin{lemma} \label{firstlemmaAFE}
Let $G$ be any entire function which decays exponentially fast in vertical strips, is even and normalised by $G(0)=1$. 
Then for any $\alpha$, $\beta\in\C$ such that $0 \le |\real(\alpha)| , |\real(\beta)| \le \tfrac{1}{2}$, one has
\begin{align}
  L(f,\tfrac{1}{2} + \alpha  + it)L(f,\tfrac{1}{2} + \beta  - it) &= \sum\limits_{m \ge 1} {\sum\limits_{n \ge 1} {\frac{{{\lambda _f}(m){\lambda _f}(n)}}{{{m^{1/2 + \alpha }}{n^{1/2 + \beta }}}}} {{\left( {\frac{m}{n}} \right)}^{ - it}}{V_{\alpha ,\beta }}(mn,t)}  \nonumber \\
   &\quad + {X_{\alpha ,\beta ,t}}(t)\sum\limits_{m \ge 1} {\sum\limits_{n \ge 1} {\frac{{{\lambda _f}(m){\lambda _f}(n)}}{{{m^{1/2 - \beta }}{n^{1/2 - \alpha }}}}} {{\left( {\frac{m}{n}} \right)}^{ - it}}{V_{ - \beta , - \alpha }}(mn,t)},  \nonumber  
\end{align}
where
\[{g_{\alpha ,\beta }}(s,t) = \frac{{{L_\infty }(f,\tfrac{1}{2} + \alpha  + s + it){L_\infty }(f,\tfrac{1}{2} + \beta  + s - it)}}{{{L_\infty }(f,\tfrac{1}{2} + \alpha  + it){L_\infty }(f,\tfrac{1}{2} + \beta  - it)}},\]
and
\[
{V_{\alpha ,\beta }}(x,t) = \frac{1}{{2\pi i}}\int_{(1)} {\frac{{G(s)}}{s}{g_{\alpha ,\beta }}(s,t){x^{ - s}}ds} ,
\]
as well as
\[{X_{\alpha ,\beta ,t}}(t) = \frac{{{L_\infty }(f,\tfrac{1}{2} - \alpha  - it){L_\infty }(f,\tfrac{1}{2} - \beta  + it)}}{{{L_\infty }(f,\tfrac{1}{2} + \alpha  + it){L_\infty }(f,\tfrac{1}{2} + \beta  - it)}}.\]
\end{lemma}
\begin{lemma} \label{firstlemmaAFE2}
Suppose that $w$ satisfies the three conditions \textnormal{(a)}, \textnormal{(b)}, \textnormal{(c)}, and suppose that $h$, $k$ are positive integers with $h,k \le T^{\nu}$. Then one has
\begin{align}
  \int_{ - \infty }^\infty  w(t){{\left( {\frac{h}{k}} \right)}^{ - it}}&L(f,\tfrac{1}{2} + \alpha  + it)L(f,\tfrac{1}{2} + \beta  - it)dt = \sum\limits_{hm = kn} {\frac{{{\lambda _f}(m){\lambda _f}(n)}}{{{m^{1/2 + \alpha }}{n^{1/2 + \beta }}}}\int_{ - \infty }^\infty  {w(t){V_{\alpha ,\beta }}(mn,t)dt} }  \nonumber \\
   &\quad + \sum\limits_{hm = kn} \frac{{{\lambda _f}(m){\lambda _f}(n)}}{{{m^{1/2 - \beta }}{n^{1/2 - \alpha }}}}\int_{ - \infty }^\infty w(t){V_{ - \beta , - \alpha }}(mn,t){X_{\alpha ,\beta ,t}}dt \nonumber \\
	 &\quad + O_\varepsilon((hk)^{(1+\bm\theta)/2} T^{1/2+\bm\theta +\varepsilon}) , \nonumber
\end{align}
for $\alpha, \beta \ll L^{-1}$.
\end{lemma}
\begin{proof}
This is proved by applying Lemma \ref{firstlemmaAFE} to the right-hand side above
\begin{align}
  \int_{ - \infty }^\infty w(t) & {{\left( {\frac{h}{k}} \right)}^{ - it}}L(f,\tfrac{1}{2} + \alpha  + it)L(f,\tfrac{1}{2} + \beta  - it)dt  \nonumber \\
	 &= \sum\limits_{m \ge 1} \sum\limits_{n \ge 1} {\frac{{{\lambda _f}(m){\lambda _f}(n)}}{{{m^{1/2 + \alpha }}{n^{1/2 + \beta }}}}} \int_{ - \infty }^\infty {{\left( {\frac{{hm}}{{kn}}} \right)}^{ - it}}w(t){V_{\alpha ,\beta }}(mn,t)dt   \nonumber \\
   &\quad + \sum\limits_{m \ge 1} \sum\limits_{n \ge 1} {\frac{{{\lambda _f}(m){\lambda _f}(n)}}{{{m^{1/2 - \beta }}{n^{1/2 - \alpha }}}}} \int_{ - \infty }^\infty {{\left( {\frac{{hm}}{{kn}}} \right)}^{ - it}}w(t){X_{\alpha ,\beta ,t}}(t){V_{ - \beta , - \alpha }}(mn,t)dt .  \nonumber 
\end{align}
Clearly the main terms appearing in the statement of the lemma are given by the diagonal case $hm=kn$. Let us now look at the off-diagonal terms. Following \cite{bernard} we set
\[I_{h,k}^{{{\operatorname{ND} }_1}}(\alpha ,\beta ) = \sum\limits_{hm \ne kn} {\frac{{{\lambda _f}(m){\lambda _f}(n)}}{{{m^{1/2 + \alpha }}{n^{1/2 + \beta }}}}} \int_{ - \infty }^\infty  {{{\left( {\frac{{hm}}{{kn}}} \right)}^{ - it}}w(t){V_{\alpha ,\beta }}(mn,t)dt} ,\]
and
\[I_{h,k}^{{{\operatorname{ND} }_2}}(\alpha ,\beta ) = \sum\limits_{hm \ne kn} {\frac{{{\lambda _f}(m){\lambda _f}(n)}}{{{m^{1/2 - \beta }}{n^{1/2 - \alpha }}}}} \int_{ - \infty }^\infty  {{{\left( {\frac{{hm}}{{kn}}} \right)}^{ - it}}w(t){X_{\alpha ,\beta ,t}}(t){V_{\alpha ,\beta }}(mn,t)dt} .\]
By \cite[Lemma 3]{bernard}, we have that for any $\varepsilon > 0$, $0 < \gamma < 1$ and for any real $A>0$
\[
I_{h,k}^{\operatorname{ND}_1}(\alpha ,\beta ) = \sum_{\substack{km \ne hn \\ mn \ll T^{2 + \varepsilon} \\ | \tfrac{hm}{kn} - 1| \ll T^{ - \gamma }}} \frac{{{\lambda _f}(m){\lambda _f}(n)}}{{{m^{1/2 + \alpha }}{n^{1/2 + \beta }}}} \int_{ - \infty }^\infty w(t){{\left( {\frac{{hm}}{{kn}}} \right)}^{ - it}}{V_{\alpha ,\beta }}(mn,t)dt + O(T^{-A}), 
\]
provided that $h,k \le T^{\nu}$ and $\alpha,\beta \ll L^{-1}$. Now fix an arbitrary smooth function $\rho: ]0,\infty[ \mapsto \R$, compactly supported in $[1,2]$ and with
\[
\sum_{l=-\infty}^{\infty} \rho(2^{-l/2}x) =1 .
\]
For each integer $l$, we define
\[
\rho_l (x) = \rho(x/A_l) \quad \textnormal{with} \quad A_l = 2^{l/2}T^{\gamma}.
\]
By \cite[Lemma 4]{bernard} one has
\[
I_{h,k}^{\operatorname{ND}_1}(\alpha ,\beta ) = \sum_{\substack{{A_{{l_1}}}{A_{{l_2}}} \ll hk{T^{2 + \varepsilon }} \\ {A_{{l_1}}} \asymp {A_{{l_2}}} \\ {A_{{l_1}}},{A_{{l_2}}} \gg {T^\gamma }}} \sum_{0 < |h| \ll {T^{ - \gamma }}\sqrt {{A_{{l_1}}}{A_{{l_2}}}}} \sum_{km - hn = q} {\lambda _f}(m){\lambda _f}(n){F_{q;{l_1},{l_2}}}(km,hn) + O({T^{ - A}}),
\]
where $h,k \le T^{\nu}$ are positive integers and $\gamma$ is as above. Here
\[
{F_{q;{l_1},{l_2}}}(x,y) = \frac{{{k^{1/2 + \alpha }}{h^{1/2 + \beta }}}}{{{x^{1/2 + \alpha }}{y^{1/2 + \beta }}}}{\rho _{{l_1}}}(x){\rho _{{l_2}}}(y)\int_{ - \infty }^\infty  {w(t){{\left( {1 + \frac{q}{h}} \right)}^{ - it}}{V_{\alpha ,\beta }}\left( {\frac{{xy}}{{hk}},t} \right)dt} .
\]
As mentioned in the introduction, the key aspect of the proof of this lemma relies on a strong result about shifted convolution sums on average due to Blomer, \cite[Theorem 2]{blomer2}. Fortunately, the tool needed from \cite{blomer2} can be quoted almost verbatim (a straightforward adaption is needed and it is supplied by Bernard in \cite[Theorem 3]{bernard}). The statement is as follows.
\begin{lemma}[Bernard, 2015] \label{bernardblomerlemma}
Let $l_1$, $l_2$, $H$ and $h_1$ be positive integers. Let $M_1, M_2, P_1$ and $P_2$ be real numbers greater than $1$. Let $\{g_n\}$ be a family of smooth functions supported in $[M_1,2M_1] \times [M_2,2M_2]$ with $||g_h^{(ij)}||_\infty \ll_{i,j} (P_1/M_1)^i(P_2/M_2)^j$ for all $i,j \ge 0$. Let $\{a(h)\}$ be a sequence of complex numbers such that
\[
a(h) \ne 0 \quad  \Rightarrow \quad h \le H,\quad h_1|h \quad \textnormal{and} \quad (h_1,h/h_1) = 1.
\]
If $l_1 M_1 \asymp l_2 M_2 \asymp A$ and if there exists $\varepsilon > 0$ such that
\[
H \ll \frac{A}{\max\{P_1,P_2\}} \frac{1}{(l_1l_2 M_1 M_2 P_1 P_2)^\varepsilon},
\]
then one has
\begin{align}
  &\sum_{h = 1}^H {a(h)} \sum\limits_{{m_1},{m_2} \ge 1} {{\lambda _f}({m_1})\overline {{\lambda _f}({m_2})} {g_h}({m_1},{m_2})}  \nonumber \\
   &\ll {A^{1/2}}h_1^\theta ||a|{|_2}{({P_1} + {P_2})^{3/2}}\bigg[ {\sqrt {{P_1} + {P_2}}  + {{\left( {\frac{A}{{\max \{ {P_1},{P_2}\} }}} \right)}^\theta }\bigg( {1 + \sqrt {\frac{{({h_1},{l_1}{l_2})H}}{{{h_1}{l_1}{l_2}}}} } \bigg)} \bigg] \nonumber \\
   &\quad \times ({l_1}{l_2}{M_1}{M_2}{P_1}{P_2}H)^\varepsilon,  \nonumber  
\end{align}
for all $\varepsilon > 0$.
\end{lemma}
The required bounds for the test function were established in \cite[Lemma 5]{bernard}. Specifically, let $\alpha,\beta \ll L^{-1}$ be complex numbers and let $\sigma$ be any positive real number. For all non-negative integers $i$ and $j$, we have
\begin{align} \label{testbound}
{x^i}{y^j}\frac{{{\partial ^{i + j}}}}{{\partial {x^i}\partial {y^j}}}{F_{q;{l_1},{l_2}}}(x,y) \ll _{i,j} {\left( {\frac{a}{{{A_{{l_1}}}}}} \right)^{1/2 + \real(\alpha)  + \sigma }}{\left( {\frac{b}{{{A_{{l_2}}}}}} \right)^{1/2 + \real(\beta)  + \sigma }}{T^{1 + 2\sigma }}{\log ^j}T,
\end{align}
where the implicit constant does not depend on $q$. The trivial bound for shifted convolution sums would have yielded
\[
\sum_{\ell_1 m_1 - \ell_2 m_2 = \Delta} \lambda_f(m_1)\lambda_f(m_2)g_h(m_1,m_2) \ll_\varepsilon \min \{M_1,M_2\} (M_1 M_2)^\varepsilon,
\]
so that when we combine this with \eqref{testbound} we get
\[
I_{h,k}^{\operatorname{ND}_1}(\alpha ,\beta) \ll_\varepsilon \min \{h,k\} T^{1+\varepsilon},
\]
which is clearly not useful. As explained by Bernard, using \cite[Theorem 6.3]{ricotta} would give
\[
I_{h,k}^{\operatorname{ND}_1}(\alpha ,\beta) \ll_\varepsilon \min (hk)^{3/4+\bf{\theta}/2}T^{3/2+\varepsilon+\varepsilon}.
\]
If instead of the trivial bound we now use \cite[Theorem 1.3]{blomer1} along with \eqref{testbound}, then
\[
I_{h,k}^{\operatorname{ND}_1}(\alpha ,\beta) \ll_\varepsilon \min (hk)^{3/4+\bf{\theta}/2}T^{1/2+\varepsilon+\varepsilon},
\]
see \cite[pp. 215-217]{bernard} for further details. It is only by using Lemma \ref{bernardblomerlemma} with $H = T^{-\gamma}\sqrt{A_{l_1}A_{l_2}}$, $h_1=1$ and
\begin{equation}
a(h) = \begin{cases}
1, & \mbox{ if } h \le H, \nonumber \\
0, & \mbox{ otherwise}, \nonumber
\end{cases}
\end{equation}
as well as the previous results on $F_{q;l_1,l_2}$ that we get
\begin{align} \label{bernardblomer}
I_{h,k}^{{{\operatorname{ND} }_1}}(\alpha ,\beta ) \ll_\varepsilon {(hk)^{(1 + \bm{\theta} )/2}}{T^{1/2 + \bm{\theta}  + \varepsilon }}.
\end{align}
Similarly, one has
\[
I_{h,k}^{{{\operatorname{ND} }_2}}(\alpha ,\beta ) \ll_\varepsilon {(hk)^{(1 + \bm{\theta} )/2}}{T^{1/2 + \bm{\theta}  + \varepsilon }},
\]
and this can be shown by a similar argument.
\end{proof}
\begin{remark}
Specifically we shall use the pole annihilator
\[
G(s) = e^{s^2}p(s) \quad \textnormal{and} \quad p(s) = \frac{{{{(\alpha  + \beta )}^2} - {{(2s)}^2}}}{{{{(\alpha  + \beta )}^2}}}.
\]
The function $G(s)$ can be chosen from a wide class of functions. This choice is taken from \cite{bcy}. All that is needed is that $G$ should have rapid decay and that it vanishes at $s = \pm \frac{\alpha+\beta}{2}$.
\end{remark}
The following lemma, which is an adaption of the approximate functional equation, is needed for the computation of the term "crossterms" $I_{\ell,\ell+1}$.
\begin{lemma} \label{lemmasigma}
Let $\sigma_{\alpha,-\beta}(f,l) = \sum_{ab=l}a^{-\alpha}b^{\beta}\lambda_f(a)\lambda_f(b)$. For $L^2 \le |t| \le 2T$ we have
\[
L(f, \tfrac{1}{2} + \alpha + it) L(f, \tfrac{1}{2} - \beta + it) = \sum_{\ell = 1}^{\infty} \frac{\sigma_{\alpha, - \beta}(f, \ell)}{\ell^{1/2+it}} e^{- \ell/T^{6}} + O(T^{-1 + \epsilon})
\]
uniformly for $\alpha,\beta < L^{-1}$.
\end{lemma}
\begin{proof}
See Lemma 4.1 of \cite{bcy} and \cite{iwanieckowalski} for the appropriate bounds.
\end{proof}
\begin{lemma} \label{eulermaclaurinlemma}
Let $\{f(m)\}_{m\in\N}$ be a sequence of complex numbers and suppose that
 \begin{align}
  \sum_{m\leq M} f(m) = cM +O(M^{3/5}) 
 \end{align}
as $M \to \infty$ for some $c\in\C$. We then have for $k\in\N$ 
  \begin{align}
  \label{eq:Lemma8_ugly1}
  \sum_{m\leq M} (f^{*k})(m) = c^k M \frac{\log^{k-1}M}{k!} +O(M\log^{k-2} M )
 \end{align}
as $M \to \infty$. Furthermore, we have
  \begin{align}
  \label{eq:Lemma8_ugly21}
  \sum_{m\leq M} \frac{(f^{*k})(m) }{m} =  c^k \frac{\log^{k}M}{k!k} +O(\log^{k-1}M)
 \end{align}
as $M \to \infty$. Let $n \ge 1$ and $M' \ge 1$ with $\log M' \asymp \log M$ be given.  We then have uniformly in $\gamma$ for all $|\gamma|\le 1/2$ 
 \begin{align}
  \label{eq:Lemma8_ugly23}
  \sum_{m\leq M} (f^{*k})(m) \frac{(\log\frac{M'}{m})^n}{m} \left(\frac{M'}{m}\right)^{-\gamma}
  =&\, 
  \frac{c^k M^{-\gamma} \log^k M}{k!} \int_{0}^{1} M^{\gamma r} r^{k-1}(\log(M'/M^r))^{n}\, dr \nonumber \\
   &+ O ((\log M)^{k+n-1} ), 
%    &\asymp
%    \log^{k+n}(M)
 \end{align}
as $M \to \infty$, and the expression in  \eqref{eq:Lemma8_ugly23} has order of magnitude $(\log M)^{k+n}$ if $c\neq 0$. 
\end{lemma}
\begin{proof}
We first prove \eqref{eq:Lemma8_ugly1}. We do this by induction over $k$. 
For $k=1$, this is trivial. We thus assume \eqref{eq:Lemma8_ugly1} is true for $k-1$. 
To simplify the notation, we write $g(m):=(f^{*k-1})(m)$ and get
 \begin{align}
  \sum_{m\leq M} (f^{*k})(m)
  &=
  \sum_{m\leq M} (g*f)(m)
  = 
  \sum_{m\leq M} \sum_{d|n} g(d)f(m/d) = \sum_{ab\leq M}  g(a)f(b) \nonumber\\
  &= \sum_{a \le M}  g(a)  \sum_{b \le M/a}f(b)  =
  \sum_{a\leq M}  g(a)\bigg( \frac{cM}{a} +O\bigg(\bigg(\frac{M}{a}\bigg)^{3/5}\bigg) \bigg).
  \label{eq:Lemma8_ugly1_proof1}
 \end{align}
 We first consider the main term. We use partial summation and get
  \begin{align*}
  \sum_{a\leq M}  g(a)\frac{cM}{a}
  &=
  cM  \sum_{a\leq M}   \frac{1}{a}g(a)
  =
  cM \bigg(1/M \sum_{a\leq M} g(a) -  \sum_{m\leq M-1} \bigg(\sum_{a\leq m} g(a)\bigg)\left(\frac{1}{m+1}-\frac{1}{m}\right)	\bigg)\\
  &=
  O(M\log^{k-2} M)+ cM\sum_{m\leq M-1} \bigg(c^{k-1} m\frac{\log^{k-2} m}{(k-1)!} +O(m\log^{k-3}m ) \bigg)\frac{1}{m(m+1)}\\
  &=
  \frac{c^kM}{(k-1)!}\sum_{m\leq M-1} \frac{\log^{k-2} m}{m+1}  + O\bigg(M\log^{k-2} M+  M\sum_{m\le M-1} \frac{\log^{k-3}m}{m} \bigg)\\
    &=
  c^k M\frac{\log^{k-1}M}{k!} +O(M\log^{k-2} M ).
 \end{align*}
It remains to show that the error term in \eqref{eq:Lemma8_ugly1_proof1} is of lower order.
We have
   \begin{align*}
  \sum_{a\leq M}   g(a)O\left(\frac{M}{a}\right)^{3/5}
  &=
  O\bigg(M^{3/5}  \sum_{a\le M}   \frac{1}{a^{3/5}}g(a)\bigg)\\
  &=
  M^{3/5} O \bigg( \frac{1}{M^{3/5}} \sum_{a\leq M} g(a) -  \sum_{m\leq M-1} \bigg(\sum_{a\leq m} g(a)\bigg)\bigg(\frac{1}{(m+1)^{3/5}}-\frac{1}{m^{3/5}}\bigg)\bigg)\\
  &=
  O(M\log^{k-2} M )+ O \bigg( \sum_{m\leq M-1}m\log^{k-2} m \cdot \frac{1}{m^{6/5}}\bigg) 
  =
  O(M\log^{k-2} M ).
 \end{align*}
This completes the proof of \eqref{eq:Lemma8_ugly1}. 
The proof of \eqref{eq:Lemma8_ugly21} is almost the same as \eqref{eq:Lemma8_ugly23} and we thus prove only \eqref{eq:Lemma8_ugly23}. 
We first consider the case $\gamma=0$ and get
 \begin{align}
  \label{eq:Lemma8_ugly3}
  \sum_{m\leq M} \frac{(\log\frac{M'}{m})^n}{m} (f^{*k})(m) 
  &=
  - 
  \sum_{m\leq M-1} \bigg(\sum_{a\leq m} (f^{*k})(a) \bigg)    \left(\left(\frac{\log^{n}(M'/(m+1))}{m+1}\right)-\left(\frac{\log^{n}(M'/m))}{m}\right)\right).
 \end{align}
Note that
\begin{align*}
 \frac{\log^{n}(M'/(m+1))}{m+1}
 &=
 \frac{(\log(M') -\log(m+1))^n}{m} \frac{1}{1+1/m}\\
 &=
 \frac{(\log(M') - \log(m)-\log(1+1/m))^n}{m} \left(1-\frac{1}{m}+ O\left(\frac{1}{m^2}\right) \right)\\
 &=
  \frac{(\log(M'/m) -1/m+ O(1/m^2))^n}{m} \left(1-\frac{1}{m}+ O\left(\frac{1}{m^2}\right) \right)\\
  &=
 \frac{(\log(M'/m))^n}{m} - \frac{1}{m^2} (\log(M'/m))^{n}- \frac{n}{m^2} (\log(M'/m))^{n-1} +O\left(\frac{\log^{n-1}M'}{m^3}\right).
\end{align*}
Inserting this and \eqref{eq:Lemma8_ugly1} in \eqref{eq:Lemma8_ugly3} gives
 \begin{align*}
  \sum_{m\leq M-1} &\bigg(c^k m\frac{\log^{k-1}m}{k!} +O(m\log^{k-2} m ) \bigg)  \nonumber \\  
  &\quad \times\left(\frac{1}{m^2} \bigl(\log(M'/m)\bigr)^{n}+\frac{n}{m^2} \bigl(\log(M'/m)\bigr)^{n-1} +O\left(\frac{\log^{n-1}M'}{m^3}\right)\right).
 \end{align*}
Applying Euler-Maclaurin summation to the leading term yields
\begin{align*}
 \frac{c^k}{k!} \sum_{m\le M-1} (\log m)^{k-1} (\log(M'/m))^{n}
 &=
  \frac{c^k}{k!} \int_{1}^{M}\frac{1}{y} (\log y)^{k-1} (\log(M'/y))^{n} dy + O(\log^{k+n-1} M )\\
 &=
 \frac{c^k}{k!} (\log M)^k \int_{0}^{1} (\log y)^{k-1} (\log(M'/M^r))^{n} dr + O(\log^{k+n-1}M).
 \end{align*}
Using the variable substitution $y=M^r$ supplies the main term in \eqref{eq:Lemma8_ugly23}.
Applying Euler-Maclaurin summation to the remaining terms, one sees immediately that they are $O(\log^{k+n-1}M)$.  
This completes the proof of \eqref{eq:Lemma8_ugly23}. The argumentation for $\gamma\neq 0$ is almost identical and requires that
 \begin{align*}
 \frac{1}{(m+1)^{1-\gamma}} 
 = 
 \frac{1}{m^{1-\gamma}}  \frac{1}{(1+1/m)^{1-\gamma}} 
 =
 \frac{1}{m^{1-\gamma}} \bigg(1-(1-\gamma)\frac{1}{m} +O(1/m^2)\bigg),
 \end{align*}
with $O(1/m^2)$ uniform in $\gamma$ for  $|\gamma|\leq1/2$. This completes the proof of the lemma.
\end{proof}
This lemma can be upgraded to read like Lemma 3.3 of \cite{rrz01} and Lemma 4.4 of \cite{bcy} by incorporating smooth functions $F$ and $H$. Our choice of $c$ will be 
\begin{align} \label{constanrankin}
c = \frac{\operatorname{Res}_{s=1} L(f \otimes f,s)}{\zeta^{(N)}(2)},
\end{align}
as per the asymptotic behavior 
\begin{align}
 \label{eq:sum_la_2_f}
\sum_{n \le x} \lambda_f^2(n) = x \frac{\operatorname{Res}_{s=1} L(f \otimes f,s)}{\zeta^{(N)}(2)} + O(x^{3/5}),
\end{align}
as $x \to \infty$, found by Rankin \cite{rankin}.
%%%%%%%%%%%%%%%%%%%%%%%%%%%%%%%%%%%%%%%%%%%%%%%%%%%%%%%%%%%%%%%%%%%%%%%%%%%%%%%%%%%%%%%%%%%%%%%%%%%%%%%
\section{Evaluation of the shifted mean value integrals $I_{\star}(\alpha,\beta)$}
\subsection{The mean value integral $I_{\ell,\ell}(\alpha,\beta)$}
\noindent The strategy is to insert the definition of $\psi_\ell$ into the mean value integral $I_{\ell,\ell}$ and then compute this integral by using the tools we have developed. 
One key aspect will be the evaluation of a certain arithmetical sum into a ratio of $L$-functions. 
This has the effect of transforming the problem from an arithmetical one to an analytic counterpart. Using \eqref{generalmollifier} on $I_{\ell,\ell}$, we obtain
\begin{align}
  {I_{\ell,\ell}}(\alpha ,\beta ) &= \int_{ - \infty }^\infty  {w(t)L(f,\tfrac{1}{2} + \alpha  + it)L(f,\tfrac{1}{2} + \beta  - it)\overline {{\psi _\ell}} {\psi _{\ell }}({\sigma _0} + it)dt}  \nonumber \\
   &= \int_{ - \infty }^\infty  {w(t)\chi _f^{\ell-1}(\tfrac{1}{2} - it)\chi _f^{\ell-1}(\tfrac{1}{2} + it)L(f,\tfrac{1}{2} + \alpha  + it)L(f,\tfrac{1}{2} + \beta  - it)}  \nonumber \\
   &\quad \times \sum\limits_{{h_1}{k_1} \le {M_{\ell}}} {\frac{{{\mu_{f,\ell}}({h_1}){\lambda_f^{* \ell - 1}}({k_1})}}{{h_1^{1/2 - it}k_1^{1/2 + it}}}} {P_\ell}[{h_1}{k_1}]\sum\limits_{{h_2}{k_2} \le {M_{\ell }}} {\frac{{{\mu _{f,\ell}}({h_2})\lambda_f^{* \ell - 1}({k_2})}}{{h_2^{1/2 + it}k_2^{1/2 - it}}}} {P_{\ell }}[{h_2}{k_2}]dt \nonumber \\
   &= 
   \sum\limits_{{h_1}{k_1} \le {M_{\ell}}} {\sum\limits_{{h_2}{k_2} \le {M_{\ell }}} {\frac{{{\mu_{f,\ell}}({h_1}){\mu _{f,\ell}}({h_2}){\lambda_f^{* \ell - 1}}({k_1})\lambda_f^{* \ell - 1}({k_2})}}{{{{({h_1}{h_2}{k_1}{k_2})}^{1/2}}}}} } {P_\ell}[{h_1}{k_1}]{P_{\ell }}[{h_2}{k_2}]{J_{2,f}}, \nonumber  
\end{align}
where
\[
{J_{2,f}} = \int_{ - \infty }^\infty  {w(t){{\left( {\frac{{{h_2}{k_1}}}{{{h_1}{k_2}}}} \right)}^{ - it}}L(f,\tfrac{1}{2} + \alpha  + it)L(f,\tfrac{1}{2} + \beta  - it)dt} ,
\]
since $\chi_f(\tfrac{1}{2}+it)\chi_f(\tfrac{1}{2}-it)=1$ for all values of $t$.
%%%%%%%%%%%%%%%%%%%%%%%%%%%%%%%%%%%%%%%%%%%%%%%%%%%%%%%%%%%%%%%%%%%%%%%%%%%%%%%%%%%%%%%%%%
We now use Lemma \ref{firstlemmaAFE2} and rewrite ${J_{2,f}}$ as
 \begin{align}
 \label{eq:spliting_I_ahorrible1}
 {J_{2,f}} 
 =&
  \sum\limits_{h_2k_1m = h_1k_2n} {\frac{{{\lambda _f}(m){\lambda _f}(n)}}{{{m^{1/2 + \alpha }}{n^{1/2 + \beta }}}}\int_{ - \infty }^\infty  {w(t){V_{\alpha ,\beta }}(mn,t)dt} } \\
   &\quad + \sum\limits_{h_2k_1m = h_1k_2n} {\frac{{{\lambda _f}(m){\lambda _f}(n)}}{{{m^{1/2 - \beta }}{n^{1/2 - \alpha }}}}
   \int_{ - \infty }^\infty  {w(t){V_{ - \beta , - \alpha }}(mn,t){X_{\alpha ,\beta ,t}}dt} }  \nonumber\\
   &\quad + {O_{\varepsilon }}(  {(h_2k_1 h_1k_2)^{(1 + \bm{\theta} )/2}}{T^{1/2 + \bm{\theta}  + \varepsilon }}) \nonumber, 
\end{align}
with $V_{\alpha,\beta}$, $V_{ - \beta , - \alpha }$ and $X_{\alpha ,\beta ,t}$ as in Lemma \ref{firstlemmaAFE2}. 
This means that we can write $I_{\ell,\ell}(\alpha,\beta)=I'_{\ell,\ell}(\alpha,\beta)+I''_{\ell,\ell}(\alpha,\beta) + E(h_1,h_2,k_1,k_2)$, where $E$ is the error term above. 
Note that $I''_{\ell,\ell}(\alpha,\beta)$ can be obtained from $I'_{\ell,\ell}(\alpha,\beta)$ by switching $\alpha$ by $-\beta$ and multiplying by
\begin{align}
X_{\alpha,\beta,t} = \bigg( \frac{t\sqrt{N}}{2\pi}\bigg)^{-2(\alpha+\beta)}\bigg(1 + \frac{i(\alpha^2-\beta^2)}{t} + O(t^{-2}) \bigg), \quad \bigg(\frac{t\sqrt{N}}{2\pi}\bigg)^{-2(\alpha+\beta)} = T^{-2(\alpha+\beta)} + O(1/L), 
\end{align}
which implies that if $w_N(t)=w(t)(\frac{t\sqrt{N}}{2\pi})^{-2(\alpha+\beta)}$, then $\widehat{w}_N(0)=T^{-2(\alpha+\beta)} \widehat{w}(0) + O(T/L)$, 
see \cite[Lemma 2 and p. 229]{bernard}. 
We are left with $I''_{\ell,\ell}(\alpha,\beta) = T^{-2(\alpha + \beta)} I'_{\ell,\ell}(-\beta, -\alpha) + O(T/L)$.\\

We now estimate the error terms. We begin with the case $\ell =1$. With $\mu_{f}(h) \ll 1$ we get
\begin{align}
  \sum_{h_1k_1 \le M_1} \sum_{h_2 k_2 \le M_1} &{\frac{{{\mu _{f }}({h_1})\lambda _f^{ * 0}({k_1}){\mu _{f }}({h_2})\lambda _f^{ * 0}({k_2})}}{{{{({h_1}{k_1}{h_2}{k_2})}^{1/2}}}}}E(h_1,h_2,k_1,k_2) \nonumber \\
  &=
   \sum_{h_1 \le M_1} \sum_{h_2 \le M_1} {\frac{{{\mu _{f,\ell }}({h_1}){\mu _{f,\ell }}({h_2})}}{{{{({h_1}{h_2})}^{1/2}}}}}E(h_1,h_2,1,1) \nonumber \\
  &\ll
   \sum_{h_1 \le M_1} \sum_{h_2 \le M_1} \left|{\frac{{{\mu _{f,\ell }}({h_1}){\mu _{f,\ell }}({h_2})}}{{{{({h_1}{h_2})}^{1/2}}}}}\right| {T^{1/2 + \bm{\theta}  + \varepsilon }}{({h_1}{h_2})^{(1 + \bm{\theta} )/2}} \nonumber \\
   &\ll
   {T^{1/2 + \bm{\theta}  + \varepsilon }}\sum_{h_1 \le M_1} \sum_{h_2 \le M_1} {({h_1}{h_2})^{\bm{\theta} /2}}
   \ll
   T^{1/2 + \bm{\theta}  + \varepsilon } \bigl(M_1^{\bm{\theta} /2+1} \bigr)^2 \nonumber \\
    &\ll
    T^{1/2 + \bm{\theta}  + \varepsilon + \nu_1(\bm{\theta}+2)}
  \end{align}  
where we have used $M_1 = T^{\nu_{1}}$. This now has to be $O(T^{1-\varepsilon})$ and thus we require 
\begin{align} \label{boundonnu1}
\nu_1 < \frac{1/2-\bm{\theta}}{2+\bm{\theta}} = \frac{1-2\bm{\theta}}{4+2\bm{\theta}}.
\end{align}
Next, we estimate the error terms for $\ell\geq 2$. 
%For this, we need bounds for $\mu_{f,\ell }(h)$ and $\lambda _f^{ * (\ell-1) }(k)$. 
We know that
\[
\mu_{f}(h) \ll 1 \quad \textnormal{and} \quad  |\lambda _f(k)| \leq \tau(k) k^{\bm{\theta}}\ll k^{\bm{\theta}+\varepsilon}
\]
with $\bm{\theta} =0$ if we use \eqref{ramanujanhypothesis} and $\bm{\theta} =7/64$ if we use \eqref{kimsarnak}. Induction over $\ell$ then gives for all $\ell\geq 2$
  \begin{align} 
  \label{criticalbounds}
\mu_{f,\ell }(h) \ll_\ell h^\varepsilon, \quad \textnormal{and} \quad \lambda _f^{ * (\ell-1) }(k) \ll_\ell k^{\bm{\theta}+\varepsilon}.
\end{align} 
Using these bounds yields
\begin{align}
  &\sum_{h_1k_1 \le M_\ell} \sum_{h_2 k_2 \le M_\ell} {\frac{{{\mu _{f,\ell }}({h_1})\lambda _f^{ * \ell  - 1}({k_1}){\mu _{f,\ell }}({h_2})\lambda _f^{ * \ell  - 1}({k_2})}}{{{{({h_1}{k_1}{h_2}{k_2})}^{1/2}}}}}E(h_1,h_2,k_1,k_2) \nonumber \\
  &\ll \sum\limits_{{h_1}{k_1} \le {M_\ell }} \sum\limits_{{h_2}{k_2} \le {M_\ell }} \left|{\frac{{{\mu _{f,\ell }}({h_1})\lambda _f^{ * \ell  - 1}({k_1}){\mu _{f,\ell }}({h_2})\lambda _f^{ * \ell  - 1}({k_2})}}{{{{({h_1}{k_1}{h_2}{k_2})}^{1/2}}}}} \right| {T^{1/2 + \bm{\theta}  + \varepsilon }}{({h_1}{k_1}{h_2}{k_2})^{(1 + \bm{\theta} )/2}} \nonumber \\
  & = {T^{1/2 + \bm{\theta}  + \varepsilon }}\sum\limits_{{h_1}{k_1} \le {M_\ell }} \sum\limits_{{h_2}{k_2} \le {M_\ell }} \left|{{\mu _{f,\ell }}({h_1})\lambda _f^{ * \ell  - 1}({k_1}){\mu _{f,\ell }}({h_2})\lambda _f^{ * \ell  - 1}({k_2})} \right|	 {({h_1}{k_1}{h_2}{k_2})^{\bm{\theta} /2}} \nonumber \\
  & = T^{1/2 + \bm{\theta}  + \varepsilon } \bigg( \sum\limits_{hk \le M_\ell } \left| \mu _{f,\ell }(h) \lambda_f^{ * \ell  - 1}(k) (hk)^{\bm{\theta} /2}\right| \bigg)^2 \nonumber \\
  & \ll_\ell T^{1/2 + \bm{\theta}  + \varepsilon }\bigg( {\sum\limits_{hk \le {M_\ell }} {{h^{\bm{\theta} /2+\varepsilon}}{k^{3\bm{\theta} /2 + 2\varepsilon }}} } \bigg)^2 
    \ll_\ell T^{1/2 + \bm{\theta}  + \varepsilon }\bigg( {\sum\limits_{hk \le {M_\ell }} {{h^{\bm{3\theta} /2+2\varepsilon}}{k^{3\bm{\theta} /2 + 2\varepsilon }}} } \bigg)^2 
  \nonumber \\
  & = {T^{1/2 + \bm{\theta}  + \varepsilon }}{\bigg( \tau(M_\ell) M_\ell^{3\bm{\theta} /2 + 1+2\varepsilon} \bigg)^2} 
   \ll {T^{1/2 + \bm{\theta}  + \varepsilon }}{(M_\ell ^{3\bm{\theta} /2 + 1 + 3\varepsilon })^2} = {T^{1/2 + \bm{\theta}  + \varepsilon  + }}{T^{2{\nu _\ell }(3\bm{\theta} /2 + 1 + 3\varepsilon )}}, \nonumber 
\end{align} 
%where we have used $M_\ell = T^{\nu_{\ell}}$. 
In order that this error be $O(T^{1-\varepsilon})$, we need
\begin{align} \label{boundonnul}
\nu_\ell < \frac{1/2-\bm{\theta}}{2+3\bm{\theta}} = \frac{1-2\bm{\theta}}{4+6\bm{\theta}}.
\end{align}
where we have used $M_\ell = T^{\nu_{\ell}}$.\\

%%%%%%%%%%%%%%%%%%%%%%%%%%%%%%%%%%
By employing the Mellin representation of the polynomial $P_{\ell}$, i.e.
\begin{align} \label{mellinpoly}
{P_\ell }[{h_1}{k_1}] = \sum\limits_{i = 0}^{\deg {P_\ell }} {\frac{{{a_{\ell ,i}}}}{{{{\log }^i}{M_\ell }}}{{(\log {M_\ell }/{h_1}{k_1})}^i}}  = \sum\limits_i {\frac{{{a_{\ell ,i}}i!}}{{{{\log }^i}{M_\ell }}}} \frac{1}{{2\pi i}}\int_{(1)} {{{\left( {\frac{{{M_\ell }}}{{{h_1}{k_1}}}} \right)}^s}\frac{{ds}}{{{s^{i + 1}}}}}, 
\end{align}
we see that 
\begin{align}
  I{'_{\ell ,\ell }}(\alpha ,\beta ) &= \int_{ - \infty }^\infty  {w(t)\sum\limits_i {\sum\limits_j {\frac{{{a_{\ell ,i}}i!{a_{\ell ,j}}j!}}{{{{\log }^{i + j}}{M_\ell }}}} } } \sum\limits_{{h_2}{k_1}n = {h_1}{k_2}m} {\frac{{{\mu _{f,\ell}}({h_1}){\mu _{f,\ell}}({h_2}){\lambda_f^{* \ell - 1} }({k_1}){\lambda_f^{* \ell - 1} }({k_2})}}{{h_1^{1/2 }h_2^{1/2 }k_1^{1/2 }k_2^{1/2 }{m^{1/2 + \alpha }}{n^{1/2 + \beta }}}}}  \nonumber \\
   &\quad \times {\left( {\frac{1}{{2\pi i}}} \right)^3}\int_{(1)} {\int_{(1)} {\int_{(1)} {{{\left( {\frac{{{M_\ell }}}{{{h_1}{k_1}}}} \right)}^s}{{\left( {\frac{{{M_\ell }}}{{{h_2}{k_2}}}} \right)}^u}{{\left( {\frac{t}{{2\pi mn}}} \right)}^z}\frac{{G(z)}}{z}dz} } \frac{{ds}}{{{s^{i + 1}}}}\frac{{du}}{{{u^{j + 1}}}}}dt.  \nonumber  
\end{align}
%%%%%%%%%%%%%%%%%%%%%%%%%%%%%%%%%%%%%%%%%%%%%%%%%%%%%%%%%%%%%%%%%%%%%%%%%%%%%%%%%%%%%%%%%%%%%%%%%%%%%%%%
Now comes the part where we evaluate the arithmetic sum $\sum_{h_2 k_1 n = h_1 k_2 m}$.
\begin{lemma}
Let $\Omega_{\alpha,\beta}$ be the set of vectors $u,v,s \in \mathbb{C}^3$ satisfying
\begin{align}
\real(s) + \real(u) &> -1/2, \nonumber \\
\real(z) &> -1/4 - \real(\alpha)/2 - \real(\beta)/2, \nonumber \\
\real(s) + \real(z) &> -1/2 - \real(\alpha), \nonumber \\
\real(u) + \real(z) &> -1/2 - \real(\beta). \nonumber
\end{align}
Then one has
\begin{align}
  \sum\limits_{{h_2}{k_1}n = {h_1}{k_2}m} &\frac{{{\mu _{f,\ell}}({h_1}){\mu _{f,\ell}}({h_2}){\lambda_f^{* \ell - 1} }({k_1}){\lambda_f^{* \ell - 1} }({k_2})}\lambda_f(m)\lambda_f(n)}{{h_1^{1/2 + s}h_2^{1/2 + u}k_1^{1/2 + s}k_2^{1/2 + u}{m^{1/2 + \alpha  + z}}{n^{1/2 + \beta  + z}}}}  \nonumber \\
  &= \frac{{{L^{{{\ell }^2} + {(\ell-1) ^2}}}(f \otimes f,1 + s + u)L(f \otimes f,1 + \alpha  + \beta  + 2z)}}{{{L^{\ell (\ell-1) }}(f \otimes f,1 + 2s){L^{\ell (\ell-1) }}(f \otimes f,1 + 2u)}} \nonumber \\
  &\times \frac{{{L^{\ell-1} }(f \otimes f,1 + \alpha  + s + z){L^{\ell-1} }(f \otimes f,1 + \beta  + u + z)}}{L{^{\ell  }(f \otimes f,1 + \beta  + s + z){L^{\ell  }}(f \otimes f,1 + \alpha  + u + z)}}{A_{\alpha ,\beta }}(s,u,z), \nonumber 
\end{align}
where $A_{\alpha,\beta}(s,u,z)$ is given by an absolutely convergent Euler product on $\Omega_{\alpha,\beta}$.
\end{lemma}
\begin{proof}
Let us set
\[\mathcal{S}_{\ell,\ell} = \sum\limits_{{h_2}{k_1}n = {h_1}{k_2}m} {\frac{{{\mu _{f,\ell}}({h_1}){\mu _{f,\ell}}({h_2}){\lambda_f^{* \ell - 1} }({k_1}){\lambda_f^{* \ell - 1} }({k_2})}\lambda_f(m)\lambda_f(n)}{{h_1^{1/2 + s}h_2^{1/2 + u}k_1^{1/2 + s}k_2^{1/2 + u}{m^{1/2 + \alpha  + z}}{n^{1/2 + \beta  + z}}}}} .\]
We now write this as an Euler product over primes so that
\[
\mathcal{S}_{\ell,\ell} = \prod\limits_p \sum\limits_{{\ell _2} + {\ell _3} + {\ell _6} = {\ell _1} + {\ell _4} + {\ell _5}} \frac{{{\mu _{f,\ell}}({p^{{\ell _1}}}){\mu _{f,\ell}}({p^{{\ell _2}}}){\lambda_f^{* \ell - 1} }({p^{{\ell _3}}}){\lambda_f^{* \ell - 1} }({p^{{\ell _4}}})}\lambda_f(p^{\ell_5})\lambda_f(p^{\ell_6})}{{{{({p^{{\ell _1}}})}^{1/2 + s}}{{({p^{{\ell _2}}})}^{1/2 + u}}{{({p^{{\ell _3}}})}^{1/2 + s}}{{({p^{{\ell _4}}})}^{1/2 + u}}{{({p^{{\ell _5}}})}^{1/2 + \alpha  + z}}{{({p^{{\ell _6}}})}^{1/2 + \beta  + z}}}} , 
\]
where we have employed the substitutions ${h_1} = {p^{{\ell _1}}},{h_2} = {p^{{\ell _2}}},{k_1} = {p^{{\ell _3}}},{k_2} = {p^{{\ell _4}}}$ and $m = {p^{{\ell _5}}},n = {p^{{\ell _6}}}$. Using the fact $\mu_{\ell+1}(p) = -(\ell+1)\lambda_f(p)$ and $\lambda_f^{* \ell - 1}(p) = (\ell-1)\lambda_f(p)$ we have
\begin{align}
  \mathcal{S}_{\ell,\ell} &= %\prod\limits_p \bigg(1 + \frac{{{{\ell }^2}}}{{{p^{1 + s + u}}}} - \frac{{\ell (\ell-1) }}{{{p^{1 + 2s}}}} - \frac{{\ell }}{{{p^{1 + \beta  + s + z}}}} - \frac{{(\ell-1) \ell}}{{{p^{1 + 2u}}}} + \frac{{{(\ell-1) ^2}}}{{{p^{1 + s + u}}}}\bigg)  \nonumber \\
   %&\quad \quad \quad \quad \quad \quad + \frac{\ell-1 }{{{p^{1 + \beta  + u + z}}}}  - \frac{{\ell }}{{{p^{1 + \alpha  + u + z}}}} + \frac{\ell -1}{{{p^{1 + \alpha  + s + z}}}} + \frac{1}{{{p^{1 + \alpha  + \beta  + 2z}}}} + O({p^{ - 2 + \delta(s,u,z,\alpha,\beta) }})) \nonumber \\
    \prod\limits_p \bigg(1 + \frac{\lambda_f(p)^2({{{\ell }^2} + {(\ell-1) ^2}})}{{{p^{1 + s + u}}}} - \frac{\lambda_f(p)^2{\ell (\ell-1) }}{{{p^{1 + 2s}}}} - \frac{{\lambda_f(p)^2 \ell  }}{{{p^{1 + \beta  + s + z}}}} - \frac{{\lambda_f(p)^2(\ell-1) \ell}}{{{p^{1 + 2u}}}}  \nonumber \\
   &\quad \quad \quad \quad \quad \quad + \frac{\lambda_f(p)^2(\ell-1) }{{{p^{1 + \beta  + u + z}}}} - \frac{{\lambda_f(p)^2\ell  }}{{{p^{1 + \alpha  + u + z}}}} + \frac{\lambda_f(p)^2(\ell-1) }{{{p^{1 + \alpha  + s + z}}}} + \frac{\lambda_f(p)^2}{{{p^{1 + \alpha  + \beta  + 2z}}}} + O({p^{ - 2 + \delta(s,u,z,\alpha,\beta) }})\bigg) \nonumber \\
   &= \frac{{{L^{{{\ell }^2} + {(\ell-1) ^2}}}(f \otimes f,1 + s + u)L(f \otimes f,1 + \alpha  + \beta  + 2z)}}{{{L^{\ell (\ell-1) }}(f \otimes f,1 + 2s){L^{\ell (\ell-1) }}(f \otimes f,1 + 2u)}} \nonumber \\
   &\quad \times \frac{{{L^{\ell-1} }(f \otimes f,1 + \alpha  + s + z){L^{\ell-1} }(f \otimes f,1 + \beta  + u + z)}}{{{L^{\ell }}(f \otimes f,1 + \beta  + s + z){L^{\ell }}(f \otimes f,1 + \alpha  + u + z)}}{A_{\alpha ,\beta }}(s,u,z), \nonumber  
\end{align}
where $\delta(s,u,z,\alpha,\beta) \in \Omega_{\alpha,\beta}$ and
\[
A_{\alpha,\beta}(s,u,z) = \prod_p \bigg(1 + \sum_{r,l} \frac{a_{p,l,\ell}(p)}{p^{r+X_{r,l,\ell}(s,u,z,\alpha,\beta)}} \bigg).
\]
Here $|a_{p,l,\ell}| \ll \ell^2$ and $X_{r,l,\ell}(s,u,z,\alpha,\beta)$ are linear forms in $s,u,z,\alpha,\beta$ and the sum over $r,l$ is absolutely convergent in $\Omega_{\alpha,\beta}$.
\end{proof}
\noindent Note that when $\ell=1$, the above reduces to
\begin{align} \label{specialell1}
\sum_{h_2 n = h_1 m} &\frac{\mu_f(h_1) \mu_f(h_2) \lambda_f(m)\lambda_f(n)}{h_1^{1/2+s}h_2^{1/2+u}m^{1/2+\alpha+z}n^{1/2+\beta+z}}  \nonumber \\
&=\frac{L(f \otimes f,1+s+u)L(f \otimes f,1 + \alpha + \beta +2s)}{L(f \otimes f,1 + \beta +s +z)L(f \otimes f,1 + \alpha +u + z)} A_{\alpha,\beta}(s,u,z),   
\end{align}
since
\begin{equation} \label{reductionell1}
\lim_{\substack{\ell \to 1 \\ \ell \in \N}} \lambda_f^{*\ell-1}(k) = 
\begin{cases}
1, & \mbox{ if $k=1$},  \\
0, & \mbox{ otherwise}.
\end{cases}
\end{equation}
Consequently, we arrive at
\begin{align}
  I{'_{\ell ,\ell }}(\alpha ,\beta ) &= \int_{ - \infty }^\infty  {w(t)\sum\limits_i {\sum\limits_j {\frac{{{a_{\ell ,i}}i!{a_{\ell ,j}}j!}}{{{{\log }^{i + j}}{M_\ell }}}} } }  {\left( {\frac{1}{{2\pi i}}} \right)^3}\int_{(1)} {\int_{(1)} {\int_{(1)} {M_\ell ^{s + u}{{\left( {\frac{t}{{2\pi }}} \right)}^z}\frac{{G(z)}}{z}} } }  \nonumber \\
   &\quad \times \frac{{{L^{{{\ell }^2} + {(\ell-1) ^2}}}(f \otimes f,1 + s + u)L(f \otimes f,1 + \alpha  + \beta  + 2z)}}{{{L^{\ell (\ell-1) }}(f \otimes f,1 + 2s){L^{\ell (\ell-1) }}(f \otimes f,1 + 2u)}} \nonumber \\
   &\quad \times \frac{{{L^{\ell-1} }(f \otimes f,1 + \alpha  + s + z){L^{\ell-1} }(f \otimes f,1 + \beta  + u + z)}}{{{L^{\ell }}(f \otimes f,1 + \beta  + s + z){L^{\ell }}(f \otimes f,1 + \alpha  + u + z)}}{A_{\alpha ,\beta }}(s,u,z)dz\frac{{ds}}{{{s^{i + 1}}}}\frac{{du}}{{{u^{j + 1}}}}dt. \nonumber 
\end{align}
Now that we have transformed the arithmetic part of the problem into its analytic counterpart, we can proceed to compute these integrals. To do so, we move the $s$-, $u$- and $z$-contours of integration to $\delta>0$ small. This is then followed by deforming the $z$-contour to $-\delta + \varepsilon$, thereby crossing the simple pole of $1/z$ at $z=0$. Recall that $G(z)$ vanishes at the pole of $\zeta(1+\alpha+\beta+2z)$. The new contour of integration yields a contribution of size
\begin{align}
  &\int_{ - \infty }^\infty  {w(t)\sum\limits_i {\sum\limits_j {\frac{{{a_{\ell ,i}}i!{a_{\ell ,j}}j!}}{{{{\log }^{i + j}}{M_\ell }}}} } }   {\left( {\frac{1}{{2\pi i}}} \right)^3}\int_{\operatorname{Re} (s) = \delta } {\int_{\operatorname{Re} (u) = \delta } {\int_{\operatorname{Re} (z) =  - \delta  + \varepsilon } {M_\ell ^{s + u}{{\left( {\frac{t}{{2\pi }}} \right)}^z}\frac{{G(z)}}{z}} } }  \nonumber \\
  &\quad \times \frac{{{L^{{{\ell }^2} + {(\ell-1) ^2}}}(f \otimes f,1 + s + u)L(f \otimes f,1 + \alpha  + \beta  + 2z)}}{{{L^{\ell (\ell-1) }}(f \otimes f,1 + 2s){L^{\ell (\ell-1) }}(f \otimes f,1 + 2u)}} \nonumber \\
  &\quad \times \frac{{{L^{\ell-1} }(f \otimes f,1 + \alpha  + s + z){L^{\ell-1} }(f \otimes f,1 + \beta  + u + z)}}{{{L^{\ell }}(f \otimes f,1 + \beta  + s + z){L^{\ell }}(f \otimes f,1 + \alpha  + u + z)}}{A_{\alpha ,\beta }}(s,u,z)dz\frac{{ds}}{{{s^{i + 1}}}}\frac{{du}}{{{u^{j + 1}}}}dt \nonumber \\
  & \ll \int_{ - \infty }^\infty  {|w(t)|dt} {T^{2( - \delta  + \varepsilon )}}M_\ell ^{2\delta } \ll {T^{1 - (2 - 2{\nu _\ell })\delta  + \varepsilon }} \ll {T^{1 - \varepsilon }} \nonumber  
\end{align}
for sufficiently small $\varepsilon$. Let us now write $I'_{\ell,\ell}(\alpha,\beta)$ as $I'_{\ell,\ell}(\alpha,\beta)=I'_{\ell,\ell,0}(\alpha,\beta)+O(T^{1-\varepsilon})$, where $I'_{\ell,\ell,0}(\alpha,\beta)$ corresponds to the residue at $z=0$, i.e.
\begin{align}
  I{'_{\ell ,\ell ,0}}(\alpha ,\beta ) &= \int_{ - \infty }^\infty  {w(t)\sum\limits_i {\sum\limits_j {\frac{{{a_{\ell ,i}}i!{a_{\ell ,j}}j!}}{{{{\log }^{i + j}}{M_\ell }}}} } } {\left( {\frac{1}{{2\pi i}}} \right)^2}\int_{(\delta)} {\int_{(\delta)} {\mathop {\operatorname{Res} }\limits_{z = 0} M_\ell ^{s + u}{{\left( {\frac{t}{{2\pi }}} \right)}^z}\frac{{G(z)}}{z}} }  \nonumber \\
   &\quad \times \frac{{{L^{{{\ell }^2} + {(\ell-1) ^2}}}(f \otimes f,1 + s + u)L(f \otimes f,1 + \alpha  + \beta  + 2z)}}{{{L^{\ell (\ell-1) }}(f \otimes f,1 + 2s){L^{\ell (\ell-1) }}(f \otimes f,1 + 2u)}} \nonumber \\
   &\quad \times \frac{{{L^{\ell-1} }(f \otimes f,1 + \alpha  + s + z){L^{\ell-1} }(f \otimes f,1 + \beta  + u + z)}}{{{L^{\ell }}(f \otimes f,1 + \beta  + s + z){L^{\ell }}(f \otimes f,1 + \alpha  + u + z)}}{A_{\alpha ,\beta }}(s,u,z)\frac{{ds}}{{{s^{i + 1}}}}\frac{{du}}{{{u^{j + 1}}}}dt \nonumber \\
   &= \widehat w(0)L(f \otimes f,1 + \alpha  + \beta )\sum\limits_i {\sum\limits_j {\frac{{{a_{\ell ,i}}i!{a_{\ell ,j}}j!}}{{{{\log }^{i + j}}{M_\ell }}}} } K_{\ell,\ell}, \nonumber  
\end{align}
where
\begin{align}
  K_{\ell,\ell} = {\left( {\frac{1}{{2\pi i}}} \right)^2}\int_{(\delta)} \int_{(\delta)} &{M_\ell ^{s + u}}  \frac{{{L^{{{\ell }^2} + {(\ell-1) ^2}}}(f \otimes f,1 + s + u)}}{{{L^{\ell (\ell-1) }}(f \otimes f,1 + 2s){L^{\ell (\ell-1) }}(f \otimes f,1 + 2u)}} \nonumber \\
   &\quad \times \frac{{{L^{\ell-1} }(f \otimes f,1 + \alpha  + s){L^{\ell-1} }(f \otimes f,1 + \beta  + u)}}{{{L^{\ell }}(f \otimes f,1 + \beta  + s){L^{\ell }}(f \otimes f,1 + \alpha  + u)}}{A_{\alpha ,\beta }}(s,u,0)\frac{{ds}}{{{s^{i + 1}}}}\frac{{du}}{{{u^{j + 1}}}}. \nonumber  
\end{align}
Before we compute $K_{\ell,\ell}$, we need to sort out the situation with $A_{\alpha,\beta}$. One can see that
\begin{align*}
 A_{0,0}(s,s,s) 
&= 
\sum\limits_{{h_2}{k_1}n = {h_1}{k_2}m} 
{\frac{{{\mu _{f,\ell}}({h_1}){\mu _{f,\ell}}({h_2}){\lambda_f^{* \ell - 1} }({k_1}){\lambda_f^{* \ell - 1}({k_2}) \lambda_f(m)\lambda_f(n)}}}{{{{({h_1}{h_2}{k_1}{k_2}mn)}^{1/2 + s}}}}} \\
&=
\sum_{j=1}^\infty j^{-1-s}\bigg( \sum\limits_{h_2k_1n = j}  \mu_{f,\ell}(h_2)\lambda_f^{* \ell - 1} (k_1) \lambda_f(n)\bigg)
\bigg( \sum\limits_{h_1k_2m = j} \mu_{f,\ell}(h_1)\lambda_f^{* \ell - 1} (k_2) \lambda_f(m) \bigg)\\
&=
\sum_{j=1}^\infty j^{-1-s}\bigg( (\mu_{f,\ell}*\lambda_f^{* \ell - 1}*\lambda_f)(j)\bigg)^2.
\end{align*}
It now follows by the definition of $\lambda_f$ and $\mu_{f,\ell}$, see \eqref{eq:def_L_f} and \eqref{eq:def_mu_f}, that 
\[A_{0,0}(s,s,s) 
% = \sum\limits_{{l_1}m = {l_2}n} {\frac{{\mu ({l_1})\mu ({l_2})}}{{{{({l_1}{l_2}mn)}^{1/2 + s}}}}}  
= 1,\]
for all values of $s$. We next use the Rankin-Selberg convolution $L$-function given by \eqref{rankinselbergsquare} and reverse the order of summation
\begin{align}
  K_{\ell ,\ell } = \sum\limits_{n \le {M_\ell }} &{\frac{{{{(\lambda _f^2(n))}^{^{ * {{\ell }^2} + {(\ell-1) ^2}}}}}}{m}} {\left( {\frac{1}{{2\pi i}}} \right)^2}\int_{(\delta )} {\int_{(\delta )} {\frac{{{{\{ {\zeta ^{(N)}}(2(1 + s + u))\} }^{{{\ell }^2} + {(\ell-1) ^2}}}}}{{{L^{\ell (\ell-1) }}(f \otimes f,1 + 2s){L^{\ell (\ell-1) }}(f \otimes f,1 + 2u)}}} }  \nonumber \\
   &\quad \times {\left( {\frac{{{M_\ell }}}{m}} \right)^{s + u}}\frac{{{L^{\ell-1} }(f \otimes f,1 + \alpha  + s){L^{\ell-1} }(f \otimes f,1 + \beta  + u)}}{{{L^{\ell }}(f \otimes f,1 + \beta  + s){L^{\ell }}(f \otimes f,1 + \alpha  + u)}}{A_{\alpha ,\beta }}(s,u,0)\frac{{ds}}{{{s^{i + 1}}}}\frac{{du}}{{{u^{j + 1}}}}. \nonumber  
\end{align}
To simplify the calculations that will follow shortly, we will set the integrand to be
\begin{align}
  {r_{\ell ,\ell }}(\alpha ,\beta ,i,j,s,u) &= \frac{{{{({M_\ell }/m)}^{s + u}}}}{{{s^{i + 1}}{u^{j + 1}}}}\frac{{{{\{ {\zeta ^{(N)}}(2(1 + s + u))\} }^{{{\ell }^2} + {(\ell-1) ^2}}}}}{{{L^{\ell (\ell-1) }}(f \otimes f,1 + 2s){L^{\ell (\ell-1) }}(f \otimes f,1 + 2u)}} \nonumber \\
   &\quad \times \frac{{{L^{\ell-1} }(f \otimes f,1 + \alpha  + s){L^{\ell-1} }(f \otimes f,1 + \beta  + u)}}{{{L^{\ell }}(f \otimes f,1 + \beta  + s){L^{\ell }}(f \otimes f,1 + \alpha  + u)}}{A_{\alpha ,\beta }}(s,u,0). \nonumber  
\end{align}
We are going to follow a reasoning analogous to \cite{bcy} and \cite{krz01} by using the zero-free region of $L(f \otimes f,s)$, see \cite[Theorem 5.10]{iwanieckowalski}. More precisely, by taking \eqref{zerofreeinequality} into account, we consider the contour $\gamma = \gamma_1 \cup \gamma_2 \cup \gamma_3$ given by
\begin{align}
\gamma_1 &= \{ i\tau : |\tau| \ge Y\}, \nonumber \\
\gamma_2 &= \{ \sigma \pm iY : -c/\log Y \le \sigma \le 0\}, \nonumber \\
\gamma_3 &= \{ -c / \log Y + i\tau : |\tau| \le Y\}, \nonumber
\end{align}
with $c>0$ and $Y \ge 1$ large, where $c$ is chosen so that there are no zeros between the curve $\gamma$ and $\real =\delta$. Since $L(f \otimes f,s)$ does not vanish, we replace the double integrals of $\real(u) = \real(v) = \delta$ by the contour of integration $\gamma$ so that by the Cauchy residue theorem we have
\begin{align}
  &{\left( {\frac{1}{{2\pi i}}} \right)^2}\int_{(\delta )} {\int_{(\delta )} {{r_{\ell ,\ell}}(\alpha ,\beta ,i,j,s,u)dsdu} }  \nonumber \\
   &= \mathop {\operatorname{Res} }\limits_{s = 0} \frac{1}{{2\pi i}}\int_{\operatorname{Re} (u) = \delta } {{r_{\ell ,\ell}}(\alpha ,\beta ,i,j,s,u)du}  + {\left( {\frac{1}{{2\pi i}}} \right)^2}\int_{s \in \gamma } {\int_{\operatorname{Re} (u) = \delta } {{r_{\ell ,\ell}}(\alpha ,\beta ,i,j,s,u)dsdu} }  \nonumber \\
   &= \mathop {\operatorname{Res} }\limits_{s = u = 0} {r_{\ell ,\ell}}(\alpha ,\beta ,i,j,s,u) + \mathop {\operatorname{Res} }\limits_{s = 0} \frac{1}{{2\pi i}}\int_{u \in \gamma } {{r_{\ell ,\ell}}(\alpha ,\beta ,i,j,s,u)du}  \nonumber \\
   &\quad + \mathop {\operatorname{Res} }\limits_{u = 0} \frac{1}{{2\pi i}}\int_{s \in \gamma } {{r_{\ell ,\ell}}(\alpha ,\beta ,i,j,s,u)ds}  + {\left( {\frac{1}{{2\pi i}}} \right)^2}\int_{s \in \gamma } \int_{u \in \gamma } {r_{\ell ,\ell}}(\alpha ,\beta ,i,j,s,u)dsdu .  %\nonumber  
   \label{eq:curve_gamma}
\end{align}
The first estimation will be that of ${\operatorname{Res} _{s = 0}}\frac{1}{{2\pi i}}\int_{s \in \gamma } {{r_{\ell ,\ell}}(\alpha ,\beta ,i,j,s,u)ds}$. To estimate this, we will first write the residue as a contour integral over a small circle of radius $1/L$ centered at $0$, i.e.
\begin{align}
  \mathop {\operatorname{Res} }\limits_{s = 0} \frac{1}{{2\pi i}}\int_{u \in \gamma } {{r_{\ell ,\ell }}(\alpha ,\beta ,i,j,s,u)du} &= {\left( {\frac{1}{{2\pi i}}} \right)^2}\int_{u \in \gamma } {\frac{{{{({M_\ell }/m)}^u}{L^{\ell-1} }(f \otimes f,1 + \beta  + u)}}{{{L^{\ell (\ell-1) }}(f \otimes f,1 + 2u){L^{\ell }}(f \otimes f,1 + \alpha  + u)}}}  \nonumber \\
  &\quad \times \oint_{D(0,{L^{ - 1}})} {{{\left( {\frac{{{M_\ell }}}{m}} \right)}^s}{{\{ {\zeta ^{(N)}}(2(1 + s + u))\} }^{{{\ell }^2} + {(\ell-1) ^2}}}{A_{\alpha ,\beta }}(s,u,0)}  \nonumber \\
  &\quad \times \frac{{{L^{\ell-1} }(f \otimes f,1 + \alpha  + s)}}{{{L^{\ell (\ell-1) }}(f \otimes f,1 + 2s){L^{\ell }}(f \otimes f,1 + \beta  + s)}}\frac{{ds}}{{{s^{i + 1}}}}\frac{{du}}{{{u^{j + 1}}}}. \nonumber  
\end{align}
We also have the bound \cite{bernard,iwanieckowalski}
\begin{align} \label{boundsLfunctions1}
\frac{1}{{L(f \otimes f,\sigma  + i\tau )}} \ll \log |\tau |.
\end{align}
Next we use the fact that 
\[
\zeta^{(N)}(2(1+s+u))A_{\alpha,\beta}(s,u,0) \ll 1
\] 
in this contour of integration, as well as the bound
\[
\frac{1}{s^{i+1}}\frac{{{L^{\ell-1} }(f \otimes f,1 + \alpha  + s)}}{{{L^{\ell (\ell-1) }}(f \otimes f,1 + 2s){L^{\ell }}(f \otimes f,1 + \beta  + s)}} 
\ll 
{(2s)^{\ell (\ell-1) -i-1 }}\frac{{{{(\beta  + s)}^{\ell}}}}{{{{(\alpha  + s)}^{\ell-1} }}} 
\ll {L^{i- \ell (\ell-1) }},
\]
since $s \asymp 1/L$.  Using the fact that the arclength of the curve is $\asymp 1/L$, we obtain
\begin{align}
  &\mathop {\operatorname{Res} }\limits_{s = 0} \frac{1}{{2\pi i}}\int_{u \in \gamma } {{r_{\ell ,\ell }}(\alpha ,\beta ,i,j,s,u)du}  \nonumber \\
   &\ll 
   {L^{i - 1 - \ell (\ell-1) }}\int_{u \in \gamma } {\frac{{{{({M_\ell }/m)}^{\operatorname{Re} (u)}}{L^{\ell-1} }(f \otimes f,1 + \beta  + u)}}{{{L^{\ell (\ell-1) }}(f \otimes f,1 + 2u){L^{\ell }}(f \otimes f,1 + \alpha  + u)}}\frac{{du}}{{|u{|^{j + 1}}}}}  \nonumber \\
   &\ll 
   L^{i - 1 - \ell (\ell-1) }\int_{|\tau | \ge Y} {\frac{{{{(\log \tau )}^{\ell (\ell-1)  + \ell }}}}{{|\tau {|^{j + 1}}}}d\tau }  
   + L^{i - 1 - \ell (\ell-1) }{(\log Y)^{\ell (\ell-1)  + \ell }}\int_{ - c/\log Y}^0 {\frac{{d\sigma }}{{|\sigma  + iY{|^{j + 1}}}}}  \nonumber \\
   &\quad + L^{i - 1 - \ell (\ell-1) }{\left( {\frac{{{M_\ell }}}{m}} \right)^{ - c/\log Y}}{(\log Y)^{\ell (\ell-1)  + \ell }}\int_{|\tau | \le Y} {\frac{{d\tau }}{{|\tau  - ic/\log Y{|^{j + 1}}}}}  \nonumber \\
   &\ll 
  L^{i - 1 - \ell (\ell-1) } {(\log Y)}^{\ell (\ell-1)  + \ell }\bigg(\frac{1}{Y^j} + {\left( {\frac{{{M_\ell }}}{m}} \right)^{ - c/\log Y}}(\log Y)^{j}\bigg). \nonumber 
\end{align}
Consequently, we get
\[\mathop {\operatorname{Res} }\limits_{s = 0} \frac{1}{{2\pi i}}\int_{u \in \gamma } {{r_{\ell ,\ell }}(\alpha ,\beta ,i,j,s,u)du}  \ll {L^{i - 1 - \ell (\ell-1) }}{(\log Y)^{\ell (\ell-1)  + \ell }} \bigg( {\frac{1}{Y^j} + {{(\log Y)}^{j}}{{\left( {\frac{{{M_\ell }}}{m}} \right)}^{ - c/\log Y}}} \bigg).\]
For reasons of symmetry, i.e. $r(\alpha,\beta,i,j,s,u)=r(\beta,\alpha,j,i,u,s)$, we also get
\[\mathop {\operatorname{Res} }\limits_{u = 0} \frac{1}{{2\pi i}}\int_{s \in \gamma } {{r_{\ell ,\ell }}(\alpha ,\beta ,i,j,s,u)du}  \ll {L^{j - 1 - \ell (\ell-1) }}{(\log Y)^{\ell (\ell-1)  + \ell }} \bigg( {\frac{1}{{{Y^i}}} + (\log Y)^{i}{{\left( {\frac{{{M_\ell }}}{m}} \right)}^{ - c/\log Y}}} \bigg).\]
Keeping this bound in mind, we can bound the double integrals over $\gamma$ as
\begin{align}
  &{\left( {\frac{1}{{2\pi i}}} \right)^2}\int_{s \in \gamma } {\int_{u \in \gamma } {{r_{\ell ,\ell }}(\alpha ,\beta ,i,j,s,u)dsdu} }  \nonumber \\
   &\ll \int_{s \in \gamma } {\frac{{{{({M_\ell }/m)}^{\operatorname{Re} (s)}}{L^{\ell-1} }(f \otimes f,1 + \alpha  + s)}}{{{L^{\ell (\ell-1) }}(f \otimes f,1 + 2s){L^{\ell }}(f \otimes f,1 + \beta  + s)}}\frac{{ds}}{{|s{|^{i + 1}}}}}  \nonumber \\
   &\quad \times \int_{u \in \gamma } {\frac{{{{({M_\ell }/m)}^u}{L^{\ell-1} }(f \otimes f,1 + \beta  + u)}}{{{L^{\ell (\ell-1) }}(f \otimes f,1 + 2u){L^{\ell }}(f \otimes f,1 + \alpha  + u)}}\frac{{du}}{{|u{|^{j + 1}}}}}  \nonumber \\
   &\ll {(\log Y)^{2(\ell (\ell-1)  + \ell )}} \bigg( {\frac{1}{{{Y^i}}} + {{(\log Y)}^i}{{\left( {\frac{{{M_\ell }}}{m}} \right)}^{ - c/\log Y}}} \bigg) \bigg( {\frac{1}{{{Y^j}}} + {{(\log Y)}^j}{{\left( {\frac{{{M_\ell }}}{m}} \right)}^{ - c/\log Y}}} \bigg) \nonumber \\
	 &\ll {(\log Y)^{2(\ell (\ell-1)  + \ell )}} \bigg( {\frac{1}{{{Y^{i + j}}}} + {{(\log Y)}^{i + j}}{{\left( {\frac{{{M_\ell }}}{m}}\right)}^{ - c/\log Y}}} \bigg). \nonumber
\end{align}
Let us now set
\[\Omega (\ell ,q) := \sum\limits_{n \le {M_\ell }} {\frac{{{{(\lambda _f^2(n))}^{ * {{\ell }^2} + {(\ell-1) ^2}}}}}{n}} \bigg( {\frac{1}{{{Y^q}}} + {{(\log Y)}^q}{{\left( {\frac{{{M_\ell }}}{m}} \right)}^{ - c/\log Y}}} \bigg).\] 
Using \eqref{eq:sum_la_2_f} and Lemma~\ref{eulermaclaurinlemma}, we can bound $\Omega (\ell ,q)$ by
\begin{align} \label{eq-terribleequation1}
  \Omega (\ell ,q) &\ll \frac{1}{{{Y^q}}}{(\log {M_\ell })^{{{\ell }^2} + {(\ell-1) ^2} }} + M_\ell ^{ - c/\log Y}{(\log Y)^q}M_\ell ^{c/\log Y}{(\log {M_\ell })^{{{\ell }^2} + {(\ell-1) ^2} }} \nonumber \\
   &\ll \frac{{{{(\log T)}^{{{\ell }^2} + {(\ell-1) ^2} }}}}{{{Y^q}}} + {(\log Y)^q}{(\log T)^{{{\ell }^2} + {(\ell-1) ^2}}} ,
\end{align}
since $\log T \asymp \log M_{\ell}$. 
Choosing $Y = \log T$, we obtain $\Omega (\ell ,q) \ll_q (\log T)^{{{\ell }^2} + {(\ell-1)^2}+\epsilon}$.
When we sum over $m$, we see that
\begin{align}
  K_{\ell ,\ell } &= \sum\limits_{m \le {M_\ell }} {\frac{{{{(\lambda _f^2(m))}^{ * {{\ell }^2} + {(\ell-1) ^2}}}}}{m}\mathop {\operatorname{Res} }\limits_{s = u = 0} {r_{\ell ,\ell }}(\alpha ,\beta ,i,j,s,u)}  \nonumber \\
   &\quad + O({L^{i - 1 - \ell (\ell-1) }}\Omega (\ell ,j){(\log Y)^{\ell (\ell-1)  + \ell }} + {L^{j - 1 - \ell (\ell-1) }}\Omega (\ell ,i){(\log Y)^{\ell (\ell-1)  + \ell }} \nonumber \\
   &\quad + \Omega (\ell ,i + j){(\log Y)^{2(\ell (\ell-1)  + \ell )}}) \nonumber \\
%    &= \sum\limits_{m \le {M_\ell }} {\frac{{{{(\lambda _f^2(m))}^{ * {{\ell }^2} + {(\ell-1) ^2}}}}}{m}\mathop {\operatorname{Res} }\limits_{s = u = 0} {r_{\ell ,\ell }}(\alpha ,\beta ,i,j,s,u)}  \nonumber \\
%    &\quad+ O\bigg({L^{i - 1 - \ell (\ell-1) }}{(\log Y)^{\ell (\ell-1)  + \ell }} \bigg( {\frac{{{{(\log T)}^{{{\ell }^2} + {(\ell-1) ^2} }}}}{{{Y^j}}} + {{(\log Y)}^j}{{(\log T)}^{{{\ell }^2} + {(\ell-1) ^2} }}} \bigg) \nonumber \\
%    &\quad+ {L^{j - 1 - \ell (\ell-1) }}{(\log Y)^{\ell (\ell-1)  + \ell }} \bigg( {\frac{{{{(\log T)}^{{{\ell }^2} + {(\ell-1) ^2} }}}}{{{Y^i}}} + {{(\log Y)}^i}{{(\log T)}^{{{\ell }^2} + {(\ell-1) ^2} }}} \bigg) \nonumber \\
%    &\quad+ {(\log Y)^{2(\ell (\ell-1)  + \ell )}} \bigg( {\frac{{{{(\log T)}^{{{\ell }^2} + {(\ell-1) ^2} }}}}{{{Y^{i + j}}}} + {{(\log Y)}^{i + j}}{{(\log T)}^{{{\ell }^2} + {(\ell-1) ^2} }}} \bigg) \bigg), \nonumber  
  &= \sum\limits_{m \le {M_\ell }} {\frac{{{{(\lambda _f^2(m))}^{ * {{\ell }^2} + {(\ell-1) ^2}}}}}{m}\mathop {\operatorname{Res} }\limits_{s = u = 0} {r_{\ell ,\ell }}(\alpha ,\beta ,i,j,s,u)}  \nonumber \\
   &\quad+  O((\log T)^{{{\ell }^2} + {(\ell-1)^2}+\epsilon} (L^{i - 1 - \ell (\ell-1)} + L^{j - 1 - \ell (\ell-1) } +1)), \nonumber  
\end{align}
% by applications of the bound \eqref{eq-terribleequation1}. 
Recall that we have $i,j\ge \ell^2-\ell+1$. 
% Choosing $Y = \log T$, yields the error term $O\left(\log T^{i+j-2+\varepsilon}\right)$. 
Therefore,
\begin{align}
  K_{\ell ,\ell } = \sum\limits_{m \le {M_\ell }} {\frac{{{{(\lambda _f^2(m))}^{ * {{\ell }^2} + {(\ell-1) ^2}}}}}{m}\mathop {\operatorname{Res} }\limits_{s = u = 0} {r_{\ell ,\ell }}(\alpha ,\beta ,i,j,s,u)} + O (\log T^{i+j-1+\varepsilon} ). \nonumber  
\end{align}
Let us now move on to the main term. We first notice that
\begin{align}
  &\frac{{{{\{ {\zeta ^{(N)}}(2(1 + s + u))\} }^{{{\ell }^2} + {(\ell-1) ^2}}}{A_{\alpha ,\beta }}(s,u,0)}}{{{L^{\ell (\ell-1) }}(f \otimes f,1 + 2s){L^{\ell (\ell-1) }}(f \otimes f,1 + 2u)}}\frac{{{L^{\ell-1} }(f \otimes f,1 + \alpha  + s){L^{\ell-1} }(f \otimes f,1 + \beta  + u)}}{{{L^{\ell }}(f \otimes f,1 + \beta  + s){L^{\ell }}(f \otimes f,1 + \alpha  + u)}} \nonumber \\
  &\quad = \frac{{{{\{ {\zeta ^{(N)}}(2)\} }^{{{\ell }^2} + {(\ell-1) ^2}}}{{(2s)}^{\ell (\ell-1) }}{{(2u)}^{\ell (\ell-1) }}}}{{{{({{\operatorname{Res} }_{s = 1}}L(f \otimes f,s))}^{{{\ell }^2} + (\ell-1) ^2+1}}}}\frac{{{{(\alpha  + u)}^{\ell}}{{(\beta  + s)}^{\ell }}}}{{{{(\alpha  + s)}^{\ell -1} }{{(\beta  + u)}^{\ell - 1} }}} + O(1/L^{2\ell(\ell-1)+3}), \nonumber  
\end{align}
since $A_{0,0}(0,0,0)=1$. We now get the product of two neatly separated integrals
\begin{align}
  \mathop {\operatorname{Res} }\limits_{s = u = 0} {r_{\ell ,\ell }}(\alpha ,\beta ,i,j,s,u) &= {2^{2\ell (\ell-1) }}\bigg( {\frac{{{\zeta ^{(N)}}(2)}}{{{{\operatorname{Res} }_{s = 1}}L(f \otimes f,s)}}} \bigg)^{{{\ell }^2} + {(\ell-1) ^2}} \frac{1}{\operatorname{Res}_{s = 1}L(f \otimes f,s)} \nonumber \\
   &\quad \times \frac{1}{{2\pi i}}\oint_{D(0,{L^{ - 1}})} {{{\left( {\frac{{{M_\ell }}}{m}} \right)}^s}\frac{{{{(\beta  + s)}^{\ell }}}}{{{{(\alpha  + s)}^{\ell -1} }}}\frac{{ds}}{{{s^{i + 1 - \ell (\ell-1) }}}}}  \nonumber \\
   &\quad \times \frac{1}{{2\pi i}}\oint_{D(0,{L^{ - 1}})} {{{\left( {\frac{{{M_\ell }}}{m}} \right)}^u}\frac{{{{(\alpha  + u)}^{\ell}}}}{{{{(\beta  + u)}^{\ell - 1} }}}\frac{{du}}{{{u^{j + 1 - \ell (\ell-1) }}}}}.  \nonumber  
\end{align}
Let us remark that the second integral is the same as the first integral except that $i$ has to be replaced by $j$ and $\alpha$ has to be replaced by $\beta$. Consequently, it is enough to compute any of these two integrals. The first integral is computed below.
\begin{lemma}
One has that
\begin{align}
  \frac{1}{{2\pi i}}\oint_{D(0,{L^{ - 1}})} {{{\left( {\frac{{{M_\ell }}}{m}} \right)}^s}\frac{{{{(\beta  + s)}^\ell }}}{{{{(\alpha  + s)}^{\ell  - 1}}}}\frac{{ds}}{{{s^{i + 1 - \ell (\ell  - 1)}}}}} &= \frac{1}{{(\ell  - 2)!}}\frac{1}{{(i - \ell (\ell  - 1))!}}\frac{d^\ell}{dx^\ell}{\left( {x + \log \frac{{{M_\ell }}}{m}} \right)^{\ell  - 1 + i - \ell (\ell  - 1)}} \nonumber \\
   &\quad \times \int_0^1 {{u^{\ell  - 2}}{{(1 - u)}^{i - \ell (\ell  - 1)}}{e^{x(\beta  - \alpha u)}}{{\left( {\frac{{{M_\ell }}}{m}} \right)}^{ - \alpha u}}du} \bigg|_{x = 0}, \nonumber  
\end{align}
for $i \ge \ell(\ell-1)+1$.
\end{lemma}
\begin{proof}
The first observation is that
\[{(\beta  + s)^{\ell }} = \frac{{{d^{\ell  }}}}{{d{x^{\ell  }}}}{e^{(\beta  + s)x}}\bigg|_{x = 0}\]
for all integer values of $\ell$. Next, set
\begin{align}
  \Upsilon_1(\alpha ,\beta ,\ell ) = \frac{1}{{2\pi i}}\oint_{D(0,{L^{ - 1}})} {{{\left( {\frac{{{M_\ell }}}{m}} \right)}^s}\frac{{{{(\beta  + s)}^{\ell }}}}{{{{(\alpha  + s)}^{\ell -1} }}}\frac{{ds}}{{{s^{i + 1 - \ell (\ell-1) }}}}}  \nonumber
\end{align}
so that
\[{\Upsilon _1}(\alpha ,\beta ,\ell ) = \frac{{{d^\ell }}}{{d{x^\ell }}}{e^{\beta x}}{\Upsilon _{11}}(x){|_{x = 0}} \quad \textnormal{where} \quad {\Upsilon _{11}}(x) = \frac{1}{{2\pi i}}\oint_{D(0,{L^{ - 1}})} {{{\left( {{e^x}\frac{{{M_\ell }}}{m}} \right)}^s}\frac{1}{{{{(\alpha  + s)}^{\ell  - 1}}}}\frac{{ds}}{{{s^{i + 1 - \ell (\ell  - 1)}}}}}. \]
Now taking a power series of the exponential inside $\Upsilon_{11}$ yields
\[{\Upsilon _{11}}(x) = \sum\limits_{r \ge 0} {\frac{1}{{r!}}{{\left( {x + \log \frac{{{M_\ell }}}{m}} \right)}^r}} \frac{1}{{2\pi i}}\oint_{D(0,{L^{ - 1}})} {\frac{{{s^{r - i - 1 + \ell (\ell  - 1)}}}}{{{{(\alpha  + s)}^{\ell  - 1}}}}ds}. \]
The poles of the integrand are $s = -\alpha$ and when $r - i - 1 + \ell (\ell  - 1) \leq  - 1$, thus the easiest approach is the one put forward in \cite{bcy}, namely that of computing the residue at infinity. By making the change of variables $s \mapsto 1/s$ we get
\[{\Upsilon _{11}}(x) = \sum\limits_{r \ge 0} {\frac{1}{{r!}}{{\left( {x + \log \frac{{{M_\ell }}}{m}} \right)}^r}} \frac{1}{{2\pi i}}\oint_{D(0,{L^{ - 1}})} \frac{{{s^{i - r - {\ell ^2} + 2\ell  - 2}}}}{{{{(1 + \alpha s)}^{\ell  - 1}}}}ds .\]
We take a power series of $(1+\alpha s)^{1-\ell}$ by the use of the binomial theorem with fractional powers 
\[\frac{1}{{{{(1 + \alpha s)}^{\ell  - 1}}}} = {(1 + \alpha s)^{1 - \ell }} = \sum\limits_{k \ge 0} {\binom{1-\ell}{k}{{(\alpha s)}^k}} .\]
Here
\[\binom{1-\ell}{k} = \frac{{(1 - \ell )(1 - \ell  - 1)(1 - \ell  - 2) \cdots (1 - \ell  - k + 1)}}{{k!}}.\]
When we insert this into $\Upsilon_{11}$ we have
\[{\Upsilon _{11}}(x) = \sum\limits_{r \ge 0} {\frac{1}{{r!}}{{\left( {x + \log \frac{{{M_\ell }}}{m}} \right)}^r}} \sum\limits_{k \ge 0} {\binom{1-\ell}{k}{\alpha ^k}} \frac{1}{{2\pi i}}\oint_{D(0,{L^{ - 1}})} {{s^{k + i - r - {\ell ^2} + 2\ell  - 2}}ds} .\]
This integral picks out $r = k + i - {\ell ^2} + 2\ell  - 1$, thus
\[{\Upsilon _{11}}(x) = {\left( {x + \log \frac{{{M_\ell }}}{m}} \right)^{i - {\ell ^2} + 2\ell  - 1}}\sum\limits_{k \ge 0} \binom{1 - \ell }{k} \frac{{{\alpha ^k}}}{{(k + i - {\ell ^2} + 2\ell  - 1)!}}{{\left( {x + \log \frac{{{M_\ell }}}{m}} \right)}^k} .\]
To end this calculation we invoke the confluent hypergeometric function of the first kind $_1F_1$, see e.g. \cite{abramowitz}. This allows us to write
\begin{align}
  \Upsilon _{11}(x) &= {\left( {x + \log \frac{{{M_\ell }}}{m}} \right)^{i - {\ell ^2} + 2\ell  - 1}}\frac{1}{{(i - {\ell ^2} + 2\ell  - 1)!}} \nonumber \\
   &\quad \times {_1F_1}\left( {\ell  - 1,i + 2\ell  - {\ell ^2}, - a\left( {x + \log \frac{{{M_\ell }}}{m}} \right)} \right) \nonumber \\
   &= {\left( {x + \log \frac{{{M_\ell }}}{m}} \right)^{i - {\ell ^2} + 2\ell  - 1}}\frac{1}{{(i - {\ell ^2} + 2\ell  - 1)!}}\frac{{\Gamma (i + 2\ell  - {\ell ^2})}}{{\Gamma (i + \ell  - {\ell ^2} + 1)\Gamma (\ell  - 1)}} \nonumber \\
   &\quad \times \int_0^1 {{e^{ - a(x + \log \tfrac{{{M_\ell }}}{m})u}}{u^{\ell  - 2}}{{(1 - u)}^{i + 2\ell  - {\ell ^2} - \ell }}du}  \nonumber \\
   &= {\left( {x + \log \frac{{{M_\ell }}}{m}} \right)^{i - {\ell ^2} + 2\ell  - 1}}\frac{1}{{(i + \ell  - {\ell ^2})!\Gamma(\ell-1)}}  \int_0^1 {{e^{ - aux}}{{\left( {\frac{{{M_\ell }}}{m}} \right)}^{ - au}}{u^{\ell  - 2}}{{(1 - u)}^{i + \ell  - {\ell ^2}  }}du},  \nonumber  
\end{align}
provided $\ell > 1$. Moreover, we remark that
\begin{align} \label{limitingconrey}
\mathop {\lim }\limits_{\ell  \to 1} \frac{1}{\Gamma(\ell-1)}\int_0^1 {{e^{ - au}}{u^{\ell  - 2}}{{(1 - u)}^{i + \ell  - {\ell ^2}}}du}  = 1
\end{align}
provided $i > -1$. Putting these results together we see that
\begin{align}
  {\Upsilon _1}(\alpha ,\beta ,\ell ) = \frac{{{d^\ell }}}{{d{x^\ell }}}{e^{\beta x}}{\left( {x + \log \frac{{{M_\ell }}}{m}} \right)^{i - {\ell ^2} + 2\ell  - 1}}\frac{1}{{(i + \ell  - {\ell ^2})!(\ell  - 2)!}} \nonumber \\
   \times \int_0^1 {{e^{ - aux}}{{\left( {\frac{{{M_\ell }}}{m}} \right)}^{ - au}}{u^{\ell  - 2}}{{(1 - u)}^{i + \ell  - {\ell ^2}}}du}  \bigg|_{x = 0}, \nonumber  
\end{align} 
as it was to be shown. This ends the proof.
\end{proof}
\noindent We can now insert this result in the residue at $s=u=0$ to obtain
\begin{align}
  K_{\ell ,\ell } &= {2^{2\ell (\ell  - 1)}}{\bigg( {\frac{{{\zeta ^{(N)}}(2)}}{{{{\operatorname{Res} }_{s = 1}}L(f \otimes f,s)}}} \bigg)^{{\ell ^2} + {{(\ell  - 1)}^2}}}\frac{1}{{{{\operatorname{Res} }_{s = 1}}L(f \otimes f,s)}} \nonumber \\
   &\quad \times \frac{1}{\Gamma^2(\ell-1)}\frac{1}{{(i - \ell (\ell  - 1))!}}\frac{1}{{(j - \ell (\ell  - 1))!}} \nonumber \\
   &\quad \times \frac{{{d^{2\ell }}}}{{d{x^\ell }d{y^\ell }}}\sum\limits_{m \le {M_\ell }} {\frac{{{{(\lambda _f^2(n))}^{ * {\ell ^2} + {{(\ell  - 1)}^2}}}}}{m}} {\left( {x + \log \frac{{{M_\ell }}}{m}} \right)^{\ell  - 1 + i - \ell (\ell  - 1)}}{\left( {y + \log \frac{{{M_\ell }}}{m}} \right)^{\ell  - 1 + j - \ell (\ell  - 1)}} \nonumber \\
   &\quad \times \int_0^1 \int_0^1 {{u^{\ell  - 2}}{v^{\ell  - 2}}{{(1 - u)}^{i - \ell (\ell  - 1)}}{{(1 - v)}^{j - \ell (\ell  - 1)}}{e^{x(\beta  - \alpha u)}}{e^{y(\alpha  - \beta v)}}{{\left( {\frac{{{M_\ell }}}{m}} \right)}^{ - \alpha u - \beta v}}dudv}  \bigg|_{x = y = 0} \nonumber \\
	 &\quad +O(L^{i+j-2+\varepsilon}). \nonumber 
\end{align}
Let us perform the sums over $i$ and $j$ in the expression for $I'_{\ell,\ell,0}$. For the first sum we have
\begin{align}
  &\sum\limits_i {\frac{{{a_{\ell ,i}}i!}}{{{{\log }^i}{M_\ell }}}\frac{1}{{(i - \ell (\ell  - 1))!}}{{\left( {x + \log \frac{{{M_\ell }}}{m}} \right)}^{\ell  - 1 + i - \ell (\ell  - 1)}}{{(1 - u)}^{i - \ell (\ell  - 1)}}}  \nonumber \\
   &= (\log M_\ell)^{ - \ell (\ell  - 1)}{\left( {x + \log \frac{{{M_\ell }}}{m}} \right)^{\ell  - 1}} \nonumber \\
   &\quad \times \sum\limits_i {{a_{\ell ,i}}i(i - 1)(i - 2) \cdots (i - \ell (\ell  - 1) + 1){{\bigg( {(1 - u)\frac{{(x + \log \tfrac{{{M_\ell }}}{m})}}{{\log {M_\ell }}}} \bigg)}^{i - \ell (\ell  - 1)}}}  \nonumber \\
   &= (\log{M_\ell })^{ - \ell (\ell  - 1)} {\left( {x + \log \frac{{{M_\ell }}}{m}} \right)^{\ell  - 1}}P_\ell ^{(\ell (\ell  - 1))} \bigg( {(1 - u)\frac{{(x + \log \tfrac{{{M_\ell }}}{m})}}{{\log {M_\ell }}}} \bigg). \nonumber  
\end{align}
Similarly, for the second sum we get
\begin{align}
  &\sum\limits_j {\frac{{{a_{\ell ,j}}j!}}{{{{\log }^j}{M_\ell }}}\frac{1}{{(j - \ell (\ell  - 1))!}}{{\left( {y + \log \frac{{{M_\ell }}}{m}} \right)}^{\ell  - 1 + j - \ell (\ell  - 1)}}{{(1 - u)}^{j - \ell (\ell  - 1)}}}  \nonumber \\
  & = (\log M_\ell)^{ - \ell (\ell  - 1)} {\left( {y + \log \frac{{{M_\ell }}}{m}} \right)^{\ell  - 1}}P_\ell ^{(\ell (\ell  - 1))} \bigg( {(1 - u)\frac{{(y + \log \tfrac{{{M_\ell }}}{m})}}{{\log {M_\ell }}}} \bigg). \nonumber  
\end{align}
Therefore, the expression for $I'_{\ell,\ell,0}$ becomes
\begin{align}
  I'_{\ell ,\ell ,0}(\alpha ,\beta ) &= \frac{{{2^{2\ell (\ell  - 1)}}\widehat w(0)}}{{\alpha  + \beta }} \bigg( {\frac{{{\zeta ^{(N)}}(2)}}{{{{\operatorname{Res} }_{s = 1}}L(f \otimes f,s)}}} \bigg)^{{\ell ^2} + {{(\ell  - 1)}^2}} \nonumber \\
   &\quad \times \frac{1}{{{{((\ell  - 2)!)}^2}}}\frac{{{d^{2\ell }}}}{{d{x^\ell }d{y^\ell }}} \bigg[{e^{\alpha y + \beta x}}\int_0^1 {\int_0^1 {{u^{\ell  - 2}}{v^{\ell  - 2}}{e^{ - \alpha ux - \beta vy}}} }  \nonumber \\
   &\quad \times \sum\limits_{m \le {M_\ell }} {\frac{{{{(\lambda _f^2(m))}^{ * {\ell ^2} + {{(\ell  - 1)}^2}}}}}{m}\frac{{{{(x + \log \tfrac{{{M_\ell }}}{m})}^{\ell  - 1}}{{(y + \log \tfrac{{{M_\ell }}}{m})}^{\ell  - 1}}}}{{{{\log }^{2\ell (\ell  - 1)}}{M_\ell }}}} {\left( {\frac{{{M_\ell }}}{m}} \right)^{ - \alpha u - \beta v}} \nonumber \\
   &\quad \times P_\ell ^{(\ell (\ell  - 1))}\bigg( {(1 - u)\frac{{(x + \log \tfrac{{{M_\ell }}}{m})}}{{\log {M_\ell }}}} \bigg)P_\ell ^{(\ell (\ell  - 1))}\bigg( {(1 - v)\frac{{(y + \log \tfrac{{{M_\ell }}}{m})}}{{\log {M_\ell }}}} \bigg)dudv \bigg]_{x = y = 0}, \nonumber \\
	 &\quad + O(TL^{-1+\varepsilon}),
\end{align}
where we have used the Laurent expansion
\[
L(f \otimes g, 1+ \alpha + \beta) = \frac{\operatorname{Res}_{s = 1} L(f \otimes g,s)}{\alpha + \beta} + O(1).
\]
We shall write the main in a more convenient way as
\begin{align}
  I'_{\ell ,\ell ,0}(\alpha ,\beta ) &= \frac{{{2^{2\ell (\ell  - 1)}}\widehat w(0)}}{{\alpha  + \beta }}{\bigg( {\frac{{{\zeta ^{(N)}}(2)}}{{{{\operatorname{Res} }_{s = 1}}L(f \otimes f,s)}}} \bigg)^{{\ell ^2} + {{(\ell  - 1)}^2}}} \nonumber \\
   &\quad \times \frac{1}{{{{((\ell  - 2)!)}^2}}}\frac{{{d^{2\ell }}}}{{d{x^\ell }d{y^\ell }}} \bigg[\int_0^1 {\int_0^1 {{u^{\ell  - 2}}{v^{\ell  - 2}}M_\ell ^{x(\beta  - \alpha u) + y(\alpha  - \beta v)}} }  \nonumber \\
   &\quad \times \sum\limits_{m \le {M_\ell }} {\frac{{{{(\lambda _f^2(m))}^{ * {\ell ^2} + {{(\ell  - 1)}^2}}}}}{m}\frac{{{{(x + \tfrac{{\log ({M_\ell }/m)}}{{\log {M_\ell }}})}^{\ell  - 1}}{{(y + \tfrac{{\log ({M_\ell }/m)}}{{\log {M_\ell }}})}^{\ell  - 1}}}}{{{{\log }^{2\ell (\ell  - 1)}}{M_\ell }{{\log }^{2}}{M_\ell }}}} {\left( {\frac{{{M_\ell }}}{m}} \right)^{ - \alpha u - \beta v}} \nonumber \\
   &\quad \times P_\ell ^{(\ell (\ell  - 1))}\bigg( {(1 - u)\bigg( {x + \frac{{\log \tfrac{{{M_\ell }}}{m}}}{{\log {M_\ell }}}} \bigg)} \bigg)P_\ell ^{(\ell (\ell  - 1))}\bigg( {(1 - v)\bigg( {y + \frac{{\log \tfrac{{{M_\ell }}}{m}}}{{\log {M_\ell }}}} \bigg)} \bigg)dudv \bigg]_{x = y = 0} \nonumber \\ 
	&\quad +O(T^{1-\varepsilon}). \nonumber
\end{align}
By the Euler-Maclaurin result of Lemma \ref{eulermaclaurinlemma}, with $k = {\ell ^2} + {(\ell  - 1)^2}$, $s =  - \alpha u - \beta v$, $x = z = M_\ell$, $F(r) = (x + r)^{\ell  - 1}P_\ell ^{(\ell (\ell  - 1))}((1 - u)(x + r))$ as well as $H(r) = (y + r)^{\ell  - 1}P_\ell ^{(\ell (\ell  - 1))}((1 - v)(y + r))$, we obtain
\begin{align}
  &\sum\limits_{m \le {M_\ell }} {\frac{{{{(\lambda _f^2(m))}^{ * {\ell ^2} + {{(\ell  - 1)}^2}}}}}{{{m^{1 - \alpha u - \beta v}}}}{{\left( {x + \frac{{\log {M_\ell }/n}}{{\log {M_\ell }}}} \right)}^{\ell  - 1}}{{\left( {y + \frac{{\log {M_\ell }/n}}{{\log {M_\ell }}}} \right)}^{\ell  - 1}}}  \nonumber \\
  &\quad \times P_\ell ^{(\ell (\ell  - 1))}\left( {(1 - u)\left( {x + \frac{{\log {M_\ell }/n}}{{\log {M_\ell }}}} \right)} \right)P_\ell ^{(\ell (\ell  - 1))}\left( {(1 - v)\left( {y + \frac{{\log {M_\ell }/n}}{{\log {M_\ell }}}} \right)} \right) \nonumber \\
  &= {\left( {\frac{{{{\operatorname{Res} }_{s = 1}}L(f \otimes f,s)}}{{{\zeta ^{(N)}}(2)}}} \right)^{{\ell ^2} + {{(\ell  - 1)}^2}}}\frac{{{{(\log {M_\ell })}^{{\ell ^2} + {{(\ell  - 1)}^2}}}}}{{({\ell ^2} + {{(\ell  - 1)}^2} - 1)!M_\ell ^{ - \alpha u - \beta v}}} \nonumber \\
  &\quad \times \int_0^1 {{{(1 - r)}^{{\ell ^2} + {{(\ell  - 1)}^2} - 1}}{{(x + r)}^{\ell  - 1}}{{(y + r)}^{\ell  - 1}}}  \nonumber \\
  &\quad \times P_\ell ^{(\ell (\ell  - 1))}((1 - u)(x + r))P_\ell ^{(\ell (\ell  - 1))}((1 - v)(y + r))M_\ell ^{r( - \alpha u - \beta v)}dr + O(L^{\ell^2 + (2\ell-1)^2-1}). \nonumber  
\end{align}
Consequently, we are left with
\begin{align}
  I'_{\ell ,\ell}(\alpha ,\beta ) &= \frac{{{2^{2\ell (\ell  - 1)}}\widehat w(0)}}{{(\alpha  + \beta )\log {M_\ell }}}\frac{1}{\Gamma^2(\ell-1)}\frac{1}{{({\ell ^2} + {{(\ell  - 1)}^2} - 1)!}} \nonumber \\
   &\quad \times \frac{{{d^{2\ell }}}}{{d{x^\ell }d{y^\ell }}} \bigg[\int_0^1 {\int_0^1 {\int_0^1 {M_\ell ^{\beta (x - v(y + r)) + \alpha (y - u(x + r))}} } }  \nonumber \\
   &\quad \times {u^{\ell  - 2}}{v^{\ell  - 2}}{(1 - r)^{{\ell ^2} + {{(\ell  - 1)}^2} - 1}}{(x + r)^{\ell  - 1}}{(y + r)^{\ell  - 1}} \nonumber \\
   &\quad \times P_\ell ^{(\ell (\ell  - 1))}((1 - u)(x + r))P_\ell ^{(\ell (\ell  - 1))}((1 - v)(y + r))drdudv \bigg]_{x = y = 0} +O(TL^{-1+\varepsilon}).\nonumber 
\end{align}
As we discussed earlier, to form the full $I_{\ell,\ell}(\alpha,\beta)$ we need to add $I'_{\ell,\ell}(\alpha,\beta)$ and $I''_{\ell,\ell}(\alpha,\beta)$, where $I''_{\ell,\ell}(\alpha,\beta)$ is formed by taking $I'_{\ell,\ell}(\alpha,\beta)$, then we switch $\alpha$ and $-\beta$, and finally we multiply by $T^{-2(\alpha + \beta)}$. To accomplish this, we first let
\[U(\alpha ,\beta ) = \frac{{M_\ell ^{\beta (x - v(y + r)) + \alpha (y - u(x + r))} - {T^{ -2( \alpha  + \beta) }}M_\ell ^{ - \alpha (x - v(y + r)) - \beta (y - u(x + r))}}}{{\alpha  + \beta }}.\]
This implies that
\begin{align}
  I_{\ell ,\ell }(\alpha ,\beta ) &= \frac{{{2^{2\ell (\ell  - 1)}}\hat w(0)}}{{\log {M_\ell }}}\frac{1}{{{\Gamma ^2}(\ell  - 1)}}\frac{1}{{({\ell ^2} + {{(\ell  - 1)}^2} - 1)!}} \nonumber \\
   &\quad \times \frac{{{d^{2\ell }}}}{{d{x^\ell }d{y^\ell }}} \bigg[\int_0^1 {} \int_0^1 {} \int_0^1 {} U(\alpha ,\beta ){u^{\ell  - 2}}{v^{\ell  - 2}}{(1 - r)^{{\ell ^2} + {{(\ell  - 1)}^2} - 1}}{(x + r)^{\ell  - 1}}{(y + r)^{\ell  - 1}} \nonumber \\
   &\quad \times P_\ell ^{(\ell (\ell  - 1))}((1 - u)(x + r))P_\ell ^{(\ell (\ell  - 1))}((1 - v)(y + r))drdudv \bigg]_{x = y = 0} + O(T{L^{ - 1 + \varepsilon }}). \nonumber  
\end{align}
However, we can also write
\[U(\alpha ,\beta ) = M_\ell ^{\beta (x - v(y + r)) + \alpha (y - u(x + r))}\frac{{1 - {{(T^2 M_\ell ^{x + y - v(y + r) - u(x + r)})}^{ - \alpha  - \beta }}}}{{\alpha  + \beta }}.\]
Finally, the identity
\[\frac{{1 - {z^{ - \alpha  - \beta }}}}{{\alpha  + \beta }} = \log z\int_0^1 {{z^{ - t(\alpha  + \beta )}}dt} ,\]
combined with the fact that $M_{\ell} = T^{\nu_{\ell}}$ yields
\begin{align}
  {c_{\ell ,\ell }}(\alpha ,\beta ) &= \frac{1}{{{\Gamma ^2}(\ell  - 1)}}\frac{{{2^{2\ell (\ell  - 1)}}}}{{({\ell ^2} + {{(\ell  - 1)}^2} - 1)!}}\frac{{{d^{2\ell }}}}{{d{x^\ell }d{y^\ell }}} \nonumber \\
   &\quad \times \bigg[\int_0^1 {} \int_0^1 {} \int_0^1 {} \int_0^1 {} {(1 - r)^{{\ell ^2} + {{(\ell  - 1)}^2} - 1}}{u^{\ell  - 2}}{v^{\ell  - 2}} \nonumber \\
   &\quad \times {T^{{\nu _\ell }(\beta (x - v(y + r)) + \alpha (y - u(x + r)))}}{({T^{2 + {\nu _\ell }(x + y - v(y + r) - u(x + r))}})^{ - t(\alpha  + \beta )}} \nonumber \\
   &\quad \times \left( {\frac{2}{{{\nu _\ell }}} + x + y - v(y + r) - u(x + r)} \right){(x + r)^{\ell  - 1}}{(y + r)^{\ell  - 1}} \nonumber \\
   &\quad \times P_\ell ^{(\ell (\ell  - 1))}((1 - u)(x + r))P_\ell ^{(\ell (\ell  - 1))}((1 - v)(y + r))dtdrdudvc\bigg]_{x = y = 0}. \nonumber  
\end{align}
This proves Lemma \ref{cellelllemma}. Theorem \ref{cellelltheorem} follows by using
\begin{align}
  c_{\ell ,\ell } &= Q\left( {\frac{{ - 1}}{{2\log T}}\frac{d}{{d\alpha }}} \right)Q\left( {\frac{{ - 1}}{{2\log T}}\frac{d}{{d\beta }}} \right){c_{\ell ,\ell }}(\alpha ,\beta ) \bigg|_{\alpha  = \beta  =  - R/L} \nonumber \\
   &= \frac{1}{\Gamma^2(\ell-1)}\frac{{{2^{2\ell (\ell  - 1)}}}}{{({\ell ^2} + {{(\ell  - 1)}^2} - 1)!}}\frac{{{d^{2\ell }}}}{{d{x^\ell }d{y^\ell }}} \nonumber \\
   &\quad \times \bigg[\int_0^1 {\int_0^1 {\int_0^1 {\int_0^1 {\left( {\frac{2}{{{\nu _\ell }}} + (x + y - v(y + r) - u(x + r))} \right)} } } } {(1 - r)^{{\ell ^2} + {{(\ell  - 1)}^2} - 1}} \nonumber \\
   &\quad \times {e^{ - \frac{\nu_\ell}{2}R[x + y - v(y + r) - u(x + r)]}}{e^{2Rt[1 + \frac{\nu_\ell}{2}(x + y - v(y + r) - u(x + r))]}} \nonumber \\
   &\quad \times Q\bigg(\frac{\nu _\ell }{2}( - x + v(y + r)) + t\bigg(1 + \frac{\nu _\ell }{2}(x + y - v(y + r) - u(x + r))\bigg)\bigg) \nonumber \\
   &\quad \times Q\bigg(\frac{\nu _\ell }{2}( - y + u(x + r)) + t\bigg(1 + \frac{\nu _\ell }{2}(x + y - v(y + r) - u(x + r))\bigg)\bigg) \nonumber \\
   &\quad \times {(x + r)^{\ell  - 1}}{(y + r)^{\ell  - 1}}{u^{\ell  - 2}}{v^{\ell  - 2}} \nonumber \\
   &\quad \times P_\ell ^{(\ell (\ell  - 1))}((1 - u)(x + r))P_\ell ^{(\ell (\ell  - 1))}((1 - v)(y + r))dtdrdudv \bigg]_{x = y = 0}. \nonumber 
\end{align}
This ends the computation of the $I_{\ell,\ell}$ term.
\subsection{The mean value integral $I_{\ell,\ell+1}(\alpha,\beta)$}
\noindent We shall follow a similar strategy to that of the case $I_{\ell,\ell+1}$, except that now we will have the factor $\chi_f(1/2+it)$ inside the integral $J_{2,f}$ below. This fact will account for the presence of the arithmetic term $\sigma_{\alpha,-\beta}(f,l)$ in the $p$-adic sum. We start by plugging in the definitions of $\psi_\ell$ and $\psi_{\ell+1}$ into the mean value integral $I_{\ell,\ell+1}$ so that
\begin{align}
  {I_{\ell,\ell + 1}}(\alpha ,\beta ) &= \int_{ - \infty }^\infty  {w(t)L(f,\tfrac{1}{2} + \alpha  + it)L(f,\tfrac{1}{2} + \beta  - it)\overline {{\psi _\ell}} {\psi _{\ell + 1}}({\sigma _0} + it)dt}  \nonumber \\
   &= \int_{ - \infty }^\infty  {w(t)\chi _f^{\ell-1}(\tfrac{1}{2} - it)\chi _f^{\ell }(\tfrac{1}{2} + it)L(f,\tfrac{1}{2} + \alpha  + it)L(f,\tfrac{1}{2} + \beta  - it)}  \nonumber \\
   &\quad \times \sum\limits_{{h_1}{k_1} \le {M_{\ell}}} {\frac{{{\mu_{f,\ell}}({h_1}){\lambda_f^{* \ell - 1}}({k_1})}}{{h_1^{1/2 - it}k_1^{1/2 + it}}}} {P_\ell}[{h_1}{k_1}]\sum\limits_{{h_2}{k_2} \le {M_{\ell + 1}}} {\frac{{{\mu_{f,\ell + 1}}({h_2})\lambda_f^{* \ell}({k_2})}}{{h_2^{1/2 + it}k_2^{1/2 - it}}}} {P_{\ell + 1}}[{h_2}{k_2}]dt \nonumber \\
   &= \sum\limits_{{h_1}{k_1} \le {M_{\ell}}} {\sum\limits_{{h_2}{k_2} \le {M_{\ell + 1}}} {\frac{{{\mu_{f,\ell}}({h_1}){\mu_{f,\ell + 1}}({h_2}){\lambda_f^{* \ell - 1}}({k_1})\lambda_f^{* \ell}({k_2})}}{{{{({h_1}{h_2}{k_1}{k_2})}^{1/2}}}}} } {P_\ell}[{h_1}{k_1}]{P_{\ell + 1}}[{h_2}{k_2}]{J_{1,f}}, \nonumber  
\end{align}
where
\[{J_{1,f}} = \int_{ - \infty }^\infty  {w(t){{\left( {\frac{{{h_2}{k_1}}}{{{h_1}{k_2}}}} \right)}^{ - it}}{\chi _f}(\tfrac{1}{2} + it)L(f,\tfrac{1}{2} + \alpha  + it)L(f,\tfrac{1}{2} + \beta  - it)dt} ,\]
since $\chi_f(\tfrac{1}{2}+it)=\chi_f(\tfrac{1}{2}-it)^{-1}$ for all values of $t$. At this point we employ the functional equation of $L(f,\tfrac{1}{2} + \beta -it)$ as well as the Stirling approximation \cite[Lemma 2]{bernard}
\[
{\chi _f}(\tfrac{1}{2} + \beta  - it){\chi _f}(\tfrac{1}{2} +  it) = \bigg( {\frac{t\sqrt{N}}{{2\pi }}} \bigg)^{ - 2\beta }(1 + O({t^{ - 1}})),
\]
to write
\[{J_{1,f}} = \int_{ - \infty }^\infty  {w(t){{\left( {\frac{{{h_2}{k_1}}}{{{h_1}{k_2}}}} \right)}^{ - it}}{{\bigg( {\frac{t\sqrt{N}}{{2\pi }}} \bigg)}^{ - 2\beta }}L(f,\tfrac{1}{2} + \alpha  + it)L(f,\tfrac{1}{2} - \beta  + it)dt}  + O({T^\varepsilon }).\]
Now that we have opposite signs in front of $\alpha$ and $\beta$ we apply Lemma \ref{lemmasigma} to get
\[{J_{1,f}} = \sum\limits_{l = 1}^\infty  {\frac{{{\sigma _{\alpha , - \beta }}(f,l)}}{{{l^{1/2}}}}{e^{ - l/{T^6}}}} \int_{ - \infty }^\infty  {w(t){{\left( {\frac{{{h_2}{k_1}}l}{{{h_1}{k_2}}}} \right)}^{ - it}}{{\bigg( {\frac{t\sqrt{N}}{{2\pi }}} \bigg)}^{ - 2\beta }}dt}  + O({T^\varepsilon }).\]
When we plug this back into $I_{\ell,\ell+1}$ we see that
\begin{align}
  I_{\ell,\ell + 1}(\alpha ,\beta ) &= \sum\limits_{{h_1}{k_1} \le {M_{\ell}}} {\sum\limits_{{h_2}{k_2} \le {M_{\ell + 1}}} {\sum\limits_{l = 1}^\infty  {\frac{{{\mu_{f,\ell}}({h_1}){\mu_{f,\ell + 1}}({h_2}){\lambda_f^{* \ell - 1}}({k_1})\lambda_f^{* \ell}({k_2}){\sigma _{\alpha , - \beta }}(f,l)}}{{{{({h_1}{h_2}{k_1}{k_2}l)}^{1/2}}}}} } } {e^{ - l/{T^6}}} \nonumber \\
   &\quad \times {P_\ell}[{h_1}{k_1}]{P_{\ell + 1}}[{h_2}{k_2}]\widehat {{w_0}}\left( {\frac{1}{{2\pi }}\log \frac{{{h_2}{k_1}l}}{{{h_1}{k_2}}}} \right), \nonumber 
\end{align}
where $w_0(t):=w(t)(\tfrac{t\sqrt{N}}{2\pi})^{-2\beta}$.
%%%%%%%%%%%%%%%%%%%%%%%%%%%%%%%%%%%%%%%%%%%%%%%%%%%%%%%%%%%%%%%%%%%%%%%%%%%%%%%%%%%%%%%%%%%%%%%%%%%%%%%%
\subsection{Bounding the off-diagonal terms}
Let $C_{\ell,\ell+1}$ denote the contribution to $I_{\ell,\ell+1}$ from the off-diagonal terms, so that
\begin{align}
  C_{\ell,\ell + 1}(\alpha ,\beta ) &= \sum_{\substack{h_1 k_1 \le M_{\ell} \\ h_2 k_2 \le M_{\ell+1} \\ l \ge 1 \\ h_1k_2 \ne h_2k_1 l}} {\frac{{{\mu_{f,\ell}}({h_1}){\mu_{f,\ell + 1}}({h_2}){\lambda_f^{* \ell - 1}}({k_1})\lambda_f^{* \ell}({k_2}){\sigma _{\alpha , - \beta }}(f,l)}}{{{{({h_1}{h_2}{k_1}{k_2}l)}^{1/2}}}}{e^{ - l/{T^6}}}{P_\ell}[{h_1}{k_1}]{P_{\ell + 1}}[{h_2}{k_2}]}  \nonumber \\
   &\quad \times \widehat {{w_0}}\left( {\frac{1}{{2\pi }}\log \frac{{{h_2}{k_1}l}}{{{h_1}{k_2}}}} \right). \nonumber  
\end{align}
Given that $M_{\ell} = T^{\nu_\ell}$ and $M_{\ell+1} = T^{\nu_{\ell+1}}$, we have to estimate the above term. Since we define $w_0(x)= w(x)(\frac{t\sqrt{N}}{2\pi})^{-2\beta}$, we have  $\int_{-\infty}^\infty w_0(x)\,dx  \ll T$. Furthermore, it was shown in \cite{bcy} that
\begin{align}
 w_0 \bigg( \frac{1}{2\pi}\log x \bigg) \ll_B \frac{T}{( 1+  \frac{T}{L}\log x  )^B}
\end{align}
for any $B \ge 0$. Let us split  $C_{\ell ,\ell  + 1}$ into
\begin{align}
C_{\ell ,\ell  + 1} = C_{\ell ,\ell  + 1}' +C_{\ell ,\ell  + 1}'' \quad \textnormal{with} \quad C_{\ell ,\ell  + 1}' =\sum_{1\leq l\leq T^8} \quad \textnormal{and} \quad C_{\ell ,\ell  + 1}'' =\sum_{l\ge T^8}.
\end{align}
For the second term, we get the bound
 \begin{align*}
 C_{\ell ,\ell  + 1}'' 
 &\ll 
\sum_{l\ge T^8}  \sum_{\substack{h_1 k_1 \le M_\ell \\ h_2 k_2 \le M_{\ell+1} \\ h_1k_2 \ne h_2 k_1 l}} 
  \frac{|\mu_{f,\ell} ({h_1}){\mu_{f,\ell  + 1}}({h_2}){\lambda_f^{* \ell - 1}}({k_1})\lambda_f^{* \ell}({k_2}){\sigma _{\alpha , - \beta }}(l)|}{{{{({h_1}{h_2}{k_1}{k_2}l)}^{1/2}}}}{e^{ - l/{T^6}}}
   \int_{-\infty}^\infty w_0(x)\,dx \\
   &\ll_\ell 
   T \sum_{l\ge T^8}  \sum_{\substack{h_1 k_1 \le M_\ell \\ h_2 k_2 \le M_{\ell+1} \\ h_1k_2 \ne h_2 k_1 l}} 
   \frac{(h_1h_2 l)^{\varepsilon}(k_1k_2)^{\varepsilon+\bm{\theta}}}{(h_1h_2k_1k_2 l)^{1/2}} e^{ - l/T^6} \\
      &\ll
   T \bigg(\sum_{l\ge T^8} l^{-1/2+\varepsilon} e^{ - l/T^6} \bigg) 
   \bigg( \sum_{h_1 k_1 \le M_\ell} (h_1k_1)^{\bm{\theta}-1/2+\varepsilon}\bigg) 
 \bigg( \sum_{h_2 k_2 \le M_\ell} (h_2k_2)^{\bm{\theta}-1/2+\varepsilon}\bigg) 
   \\
   &\ll 
   T^8 e^{-T} T^{(\bm{\theta}+1/2)(\nu_\ell+\nu_{\ell+1})+3\varepsilon}
   \ll T^{-2017},
\end{align*}
where we have used \eqref{criticalbounds}. We now come to $C_{\ell ,\ell  + 1}'$. 
We choose $\nu_{\ell}$ and $\nu_{\ell+1}$ so that $\nu_\ell+\nu_{\ell+1}<1$ and thus we have for $\frac{h_2k_1 l}{h_1k_2}\neq 1$ that
\begin{align}
 \left|1 -\frac{h_2k_1 l}{h_1k_2} \right| \ge \frac{1}{h_1k_2} \ge \frac{1}{M_{\ell}M_{\ell+1}} \ge T^{-1+\varepsilon}.
\end{align}
Therefore
\begin{align*}
 w_0\left( \frac{1}{2\pi}\log \frac{h_2k_1 l}{h_1k_2}\right)
&\ll_B 
\frac{T}{( 1+  \frac{T}{L}\log(\frac{h_2k_1 l}{h_1k_2})  )^B}
=
\frac{T}{( 1+  \frac{T}{L}\log (1+(\frac{h_2k_1 l}{h_1k_2}-1))  )^B}\\
&\ll_B 
\frac{T}{( 1+  \frac{T}{L}T^{-1+\varepsilon}  )^B}
\ll
T^{1-\varepsilon B}.
\end{align*}
Using this as well as the bounds from \eqref{criticalbounds} yields
 \begin{align*}
 C_{\ell ,\ell  + 1}'
 &\ll 
\sum_{1\leq l\leq T^8}  \sum_{\substack{h_1 k_1 \le M_\ell \\ h_2 k_2 \le M_{\ell+1} \\ h_1k_2 \ne h_2 k_1 l}} 
  \frac{|\mu_{f,\ell} ({h_1}){\mu_{f,\ell  + 1}}({h_2}){\lambda_f^{* \ell - 1}}({k_1})\lambda_f^{* \ell}({k_2}){\sigma _{\alpha , - \beta }}(l)|}{{{{({h_1}{h_2}{k_1}{k_2}l)}^{1/2}}}}{e^{ - l/{T^6}}}
   \left|\widehat w_0\left( \frac{1}{2\pi}\log \frac{h_2k_1 l}{h_1k_2} \right) \right| \\
   &\ll_\ell 
  T^{1-\varepsilon B} \sum_{l {\le} T^8}  \sum_{\substack{h_1 k_1 \le M_\ell \\ h_2 k_2 \le M_{\ell+1} \\ h_1k_2 \ne h_2 k_1 l}} 
   \frac{(h_1h_2 l)^{\varepsilon}(k_1k_2 )^{\bm{\theta}+\varepsilon}}{(h_1h_2k_1k_2 l)^{1/2}} e^{ - l/T^6} \\
   &\ll 
   T^{1-\varepsilon B} \bigg(\sum_{l {\le} T^8} l^{\varepsilon-1/2}e^{ - l/T^6} \bigg)
      \bigg( \sum_{h_1 k_1 \le M_\ell} (h_1k_1)^{\bm{\theta}-1/2+\varepsilon}\bigg) 
      \bigg( \sum_{h_2 k_2 \le M_\ell} (h_2k_2)^{\bm{\theta}-1/2+\varepsilon}\bigg) 
   \\
   &\ll 
   T^{9-\varepsilon B} T^{(\bm{\theta}+1/2)(\nu_\ell+\nu_{\ell+1})+3\varepsilon}
   \ll T^{-2017},
\end{align*}
by choosing $B$ large enough and using $l^{\varepsilon-1/2}e^{ - l/T^6} \leq 1$. This shows that for $\nu_\ell + \nu_{\ell+1} <1$ the off-diagonal terms get absorbed in the error term and do not contribute to our final results. Note that the condition $\nu_\ell+\nu_{\ell+1}<1$ is needed only for bound of $w_0$.
%%%%%%%%%%%%%%%%%%%%%%%%%%%%%%%%%%%%%%%%%%%%%%%%%%%%%%%%%%%%%%%%%%%%%%%%%%%%%%%%%%%%%%%%%%%%%%%%%%%%%%%%
\subsection{The diagonal terms $h_1k_2 = h_2k_1 l$ and their reduction to a contour integral}
By employing the Mellin identities
\[{P_\ell }[{h_1}{k_1}] = \sum\limits_{i = 0}^{\deg {P_\ell }} {\frac{{{a_{\ell ,i}}}}{{{{\log }^i}{M_\ell }}}{{(\log {M_\ell }/{h_1}{k_1})}^i}}  = \sum\limits_i {\frac{{{a_{\ell ,i}}i!}}{{{{\log }^i}{M_\ell }}}} \frac{1}{{2\pi i}}\int_{(1)} {{{\left( {\frac{{{M_\ell }}}{{{h_1}{k_1}}}} \right)}^s}\frac{{ds}}{{{s^{i + 1}}}}}, \]
and
\[{P_{\ell  + 1}}[{h_2}{k_2}] = \sum\limits_{j = 0}^{\deg {P_{\ell  + 1}}} {\frac{{{a_{\ell  + 1,j}}}}{{{{\log }^j}{M_{\ell  + 1}}}}{{(\log {M_{\ell  + 1}}/{h_2}{k_2})}^j}}  = \sum\limits_j {\frac{{{a_{\ell  + 1,j}}j!}}{{{{\log }^j}{M_{\ell  + 1}}}}} \frac{1}{{2\pi i}}\int_{(1)} {{{\left( {\frac{{{M_{\ell  + 1}}}}{{{h_2}{k_2}}}} \right)}^u}\frac{{du}}{{{u^{j + 1}}}}} ,\]
as well as the Cahen-Mellin integral
\[
e^{ - y} = \frac{1}{2\pi i}\int_{(c)} \Gamma (z)y^{ - z}dz, \quad c>0, \quad \real(y)>0, 
\]
we arrive at
\begin{align}
  I_{\ell,\ell + 1}(\alpha ,\beta ) &= \widehat {{w_0}}(0)\sum\limits_i {\sum\limits_j {\frac{{{a_{\ell ,i}}i!{a_{\ell  + 1,j}}j!}}{{{{\log }^i}{M_\ell }{{\log }^j}{M_{\ell  + 1}}}}} } {\left( {\frac{1}{{2\pi i}}} \right)^3}\int_{(1)} {\int_{(1)} {\int_{(1)} {{T^{3z}}\Gamma (z)M_\ell ^sM_{\ell  + 1}^u} } }  \nonumber \\
   &\quad \times \sum\limits_{{h_2}{k_1}l = {h_1}{k_2}} {\frac{{{\mu _{f,\ell}}({h_1}){\mu_{f,\ell + 1}}({h_2}){\lambda_f^{* \ell - 1}}({k_1}){\lambda_f^{* \ell}}({k_2}){\sigma _{\alpha , - \beta }}(f,l)}}{{h_1^{1/2 + s}h_2^{1/2 + u}k_1^{1/2 + s}k_2^{1/2 + u}{l^{1/2 + z}}}}} dz\frac{{ds}}{{{s^{i + 1}}}}\frac{{du}}{{{u^{j + 1}}}} + O({T^{1 - \varepsilon }}). \nonumber  
\end{align}
We must now evaluate the arithmetic sum $\sum_{{h_2}{k_1}l = {h_1}{k_2}}$ and turn into a ratio of $L$-functions.
\begin{lemma} \label{lemma3.3}
	Let $\Upsilon_{\alpha,\beta}$ be the set of vectors $u,s, z \in \mathbb{C}^3$ satisfying
	\begin{align}
	\real(s) &> -1/4, \nonumber \\
	\real(u) &> -1/4, \nonumber \\
	\real(z) + \real(u) &> -1/2 - \real(\alpha), \nonumber \\
	\real(z) + \real(u) &> -1/2 + \real(\beta), \nonumber \\
	\real(s) + \real(z) &> -1/2 - \real(\alpha), \nonumber \\
	\real(s) + \real(z) &> -1/2 + \real(\beta). \nonumber
	\end{align}
	Then one has
	\begin{align}
	&\sum\limits_{{h_2}{k_1}l = {h_1}{k_2}} {\frac{{{\mu _{f,\ell}}({h_1}){\mu_{f,\ell + 1}}({h_2}){\lambda_f^{* \ell-1}}({k_1}){\lambda_f^{* \ell}}({k_2}){\sigma _{\alpha , - \beta }}(f,l)}}{{h_1^{1/2 + s}h_2^{1/2 + u}k_1^{1/2 + s}k_2^{1/2 + u}{l^{1/2 + z}}}}}  \nonumber \\
	&= \frac{{{L^{2{\ell ^2}}}(f \otimes f,1 + s + u){L^\ell }(f \otimes f,1 + \alpha  + u + z){L^\ell }(f \otimes f,1 - \beta  + u + z)}}{{{L^{\ell (\ell  - 1)}}(f \otimes f, 1 + 2s){L^{\ell (\ell  + 1)}}(f \otimes f, 1 + 2u){L^\ell }(f \otimes f,1 + \alpha  + s + z){L^\ell }(f \otimes f,1 - \beta  + s + z)}} \nonumber \\
	& \quad \times {B_{\alpha ,\beta }}(s,u,z), \nonumber
	\end{align}
	where $B_{\alpha,\beta}(s,u,z)$ is given by an absolutely convergent Euler product on $\Upsilon_{\alpha,\beta}$.
\end{lemma}
\begin{proof}
	Let us set
	\[
	\mathcal{S}_{\ell,\ell+1} = \sum\limits_{{h_2}{k_1}l = {h_1}{k_2}} {\frac{{{\mu _{f,\ell}}({h_1}){\mu_{f,\ell + 1}}({h_2}){\lambda_f^{* \ell-1}}({k_1}){\lambda_f^{* \ell}}({k_2}){\sigma _{\alpha , - \beta }}(f,l)}}{{h_1^{1/2 + s}h_2^{1/2 + u}k_1^{1/2 + s}k_2^{1/2 + u}{l^{1/2 + z}}}}}.
	\]
	The definition of ${\sigma _{\alpha , - \beta }}(f,l) = \sum\nolimits_{ab = l} {{\lambda _f}(a){\lambda _f}(b){a^{ - \alpha }}{b^\beta }}$ allows us to write
	\[
	\mathcal{S}_{\ell,\ell+1} = \sum\limits_{{h_2}{k_1}ab = {h_1}{k_2}} {\frac{{{\mu_{f,\ell} }({h_1}){\mu _{f,\ell  + 1}}({h_2}){\lambda_f^{* \ell-1}}({k_1}){\lambda_f^{* \ell} }({k_2}){\lambda _f}(a){\lambda _f}(b)}}{{h_1^{1/2 + s}h_2^{1/2 + u}k_1^{1/2 + s}k_2^{1/2 + u}{a^{1/2 + \alpha  + z}}{b^{1/2 - \beta  + z}}}}}.
	\]
We now translate this into an Euler product over primes so that
	\[
	\mathcal{S}_{\ell,\ell+1} = \prod\limits_p {\sum\limits_{{\ell _2} + {\ell _3} + {\ell _5} + {\ell _6} = {\ell _1} + {\ell _4}} {\frac{{{\mu_{f,\ell} }({p^{{\ell _1}}}){\mu _{f,\ell  + 1}}({p^{{\ell _2}}}){\lambda_f^{* \ell-1}}({p^{{\ell _3}}}){\lambda_f^{* \ell} }({p^{{\ell _4}}}){\lambda _f}({p^{{\ell _5}}}){\lambda _f}({p^{{\ell _6}}})}}{{{{({p^{{\ell _1}}})}^{1/2 + s}}{{({p^{{\ell _2}}})}^{1/2 + u}}{{({p^{{\ell _3}}})}^{1/2 + s}}{{({p^{{\ell _4}}})}^{1/2 + u}}{{({p^{{\ell _5}}})}^{1/2 + \alpha  + z}}{{({p^{{\ell _6}}})}^{1/2 - \beta  + z}}}}} },
	\]
	where we have used the substitutions ${h_1} = {p^{{\ell _1}}},{h_2} = {p^{{\ell _2}}},{k_1} = {p^{{\ell _3}}},{k_2} = {p^{{\ell _4}}}$ and $a = {p^{{\ell _5}}},b = {p^{{\ell _6}}}$. 
	Note that we can consider only first order terms in $p$ since they are enough to determine the expression in the lemma. Using the facts that $\mu_{f,\ell}(p) = -\ell \lambda_f(p)$ and $\lambda_f^{* \ell}(p) = \ell \lambda_f(p)$ we have
	\begin{align}
	\mathcal{S}_{\ell,\ell+1} &= \prod\limits_p \bigg(1 + \frac{{(\ell  + 1)\ell \ \lambda_f(p)^2 }}{{{p^{1 + s + u}}}} - \frac{{\ell (\ell  - 1) \ \lambda_f(p)^2}}{{{p^{1 + 2s}}}} - \frac{{\ell {\lambda _f}(p)^2}}{{{p^{1 + \alpha  + s + z}}}} - \frac{{\ell {\lambda _f}(p)^2}}{{{p^{1 - \beta  + s + z}}}} - \frac{{(\ell  + 1)\ell \ \lambda_f(p)^2 }}{{{p^{1 + 2u}}}}  \nonumber \\
	&\quad \quad \quad \quad \quad \quad + \frac{{(\ell  - 1)\ell \ \lambda_f(p)^2}}{{{p^{1 + s + u}}}} + \frac{{\ell {\lambda _f}(p)^2}}{{{p^{1 + \alpha  + u + z}}}} + \frac{{\ell {\lambda _f}(p)^2}}{{{p^{1 - \beta  + u + z}}}} + O({p^{ - 2 + \varepsilon(s,u,z,\alpha, \beta) }})\bigg) \nonumber \\
	&= \frac{{{L^{2{\ell ^2}}}(f \otimes f,1 + s + u){L^\ell }(f \otimes f,1 + \alpha  + u + z){L^\ell }(f \otimes f,1 - \beta  + u + z)}}{{{L^{\ell (\ell  - 1)}}(f \otimes f, 1 + 2s){L^{\ell (\ell  + 1)}}(f \otimes f, 1 + 2u){L^\ell }(f \otimes f,1 + \alpha  + s + z){L^\ell }(f \otimes f,1 - \beta  + s + z)}} \nonumber \\
	& \quad \times {B_{\alpha ,\beta }}(s,u,z), \nonumber
	\end{align}
	where $\varepsilon(s,u,z,\alpha, \beta) \in \Upsilon_{\alpha,\beta}$ and
	\[
	{B_{\alpha ,\beta }}(s,u,z) = \prod_p \bigg( 1 + \sum_{\substack{r, l}} \frac{b_{p,l, \ell}(p)}{p^{r+Y_{r,l, \ell}(s,u,z,\alpha,\beta)}} \bigg),
	\]
	with $|b_{p,l, \ell}(p)| \ll \ell^2$ and $Y_{r,l, \ell}(u,s,z,\alpha,\beta)$ are linear forms in $s,u, z, \alpha, \beta$ and the sum over $r,l$ is absolutely convergent in $\Upsilon_{\alpha,\beta}$.
\end{proof}
\noindent This means that we are left with
\begin{align}
  {I_{\ell ,\ell  + 1}}(\alpha ,\beta ) &= \widehat {{w_0}}(0)\sum\limits_i {\sum\limits_j {\frac{{{a_{\ell ,i}}i!{a_{\ell  + 1,j}}j!}}{{{{\log }^i}{M_\ell }{{\log }^j}{M_{\ell  + 1}}}}} } {\left( {\frac{1}{{2\pi i}}} \right)^3}\int_{(1)} {\int_{(1)} {\int_{(1)} {{T^{3z}}\Gamma (z)M_\ell ^sM_{\ell  + 1}^u} } }  \nonumber \\
	   &\quad \times \frac{1}{{{L^{\ell (\ell  + 1)}}(f \otimes f,1 + 2u){L^\ell }{{(f \otimes f,1 + \alpha  + s + z)}^\ell }(f \otimes f,1 - \beta  + s + z)}} \nonumber \\
   &\quad \times \frac{{{L^\ell }{{(f \otimes f,1 + \alpha  + u + z)}^\ell }(f \otimes f,1 - \beta  + u + z)}}{{{L^{\ell (\ell  + 1)}}(f \otimes f,1 + 2u)}} \nonumber \\
   &\quad \times {L^{2{\ell ^2}}}(f \otimes f,1 + s + u){B_{\alpha ,\beta }}(s,u,z)dz\frac{{ds}}{{{s^{i + 1}}}}\frac{{du}}{{{u^{j + 1}}}} + O({T^{1 - \varepsilon }}). \nonumber  
\end{align}
The next step is to move the $s$- and $u$- contours of integration to $\real(s)=\real(u)=\delta$, and then move the $z$-contour to $-2\delta/3$, where $\delta > 0$ is some fixed constant such that the arithmetical factor converges absolutely. This has the effect of crossing a simple pole at $z=0$. Also, on the new paths the integral can be bounded in a straightforward way by using absolute values. Thus, the contribution to $I_{\ell,\ell+1}$ is
\begin{align}
  &\int_{ - \infty }^\infty  {w(t)} \sum\limits_i {\sum\limits_j {\frac{{{a_{\ell ,i}}i!{a_{\ell  + 1,j}}j!}}{{{{\log }^i}{M_\ell }{{\log }^j}{M_{\ell  + 1}}}}} } {\left( {\frac{1}{{2\pi i}}} \right)^3}\int_{\operatorname{Re} (s) = \delta } {\int_{\operatorname{Re} (u) = \delta } {\int_{\operatorname{Re} (z) =  - 2\delta /3} {{T^{3z}}\Gamma (z)M_\ell ^sM_{\ell  + 1}^u} } }  \nonumber \\
   &\quad \times \frac{{{L^{2{\ell ^2}}}(f \otimes f,1 + s + u){L^\ell }(f \otimes f,1 + \alpha  + u + z){L^\ell }(f \otimes f,1 - \beta  + u + z)}}{{{L^{\ell (\ell  - 1)}}(f \otimes f, 1 + 2s){L^{\ell (\ell  + 1)}}(f \otimes f, 1 + 2u){L^\ell }(f \otimes f,1 + \alpha  + s + z){L^\ell }(f \otimes f,1 - \beta  + s + z)}} \nonumber \\
   &\quad \times {B_{\alpha ,\beta }}(s,u,z)dz\frac{{ds}}{{{s^{i + 1}}}}\frac{{du}}{{{u^{j + 1}}}}dt \ll \int_{ - \infty }^\infty  {|w(t)|dt} {\left( {\frac{{{M_\ell }{M_{\ell  + 1}}}}{{{T^2}}}} \right)^\delta } \ll {T^{1 - \varepsilon }}, \nonumber  
\end{align}
since $\nu_{\ell}+\nu_{\ell+1} < 2$. This implies that
\begin{align}
  {I_{\ell ,\ell  + 1}}(\alpha ,\beta ) &= \widehat {{w_0}}(0)\sum\limits_i {\sum\limits_j {\frac{{{a_{\ell ,i}}i!{a_{\ell  + 1,j}}j!}}{{{{\log }^i}{M_\ell }{{\log }^j}{M_{\ell  + 1}}}}} } {\left( {\frac{1}{{2\pi i}}} \right)^2}\int_{(\delta )} \int_{(\delta )} \mathop {\operatorname{Res} }\limits_{z = 0}  T^{3z}\Gamma (z)M_\ell ^sM_{\ell  + 1}^u \nonumber \\
   &\quad \times \frac{1}{{{L^{\ell (\ell  + 1)}}(f \otimes f,1 + 2u){L^\ell }{{(f \otimes f,1 + \alpha  + s + z)}^\ell }(f \otimes f,1 - \beta  + s + z)}} \nonumber \\
   &\quad \times \frac{{{L^\ell }{{(f \otimes f,1 + \alpha  + u + z)}^\ell }(f \otimes f,1 - \beta  + u + z)}}{{{L^{\ell (\ell  + 1)}}(f \otimes f,1 + 2u)}} \nonumber \\
   &\quad \times {L^{2{\ell ^2}}}(f \otimes f,1 + s + u){B_{\alpha ,\beta }}(s,u,z)\frac{{ds}}{{{s^{i + 1}}}}\frac{{du}}{{{u^{j + 1}}}} + O({T^{1 - \varepsilon }}) \nonumber \\
   & = \widehat {{w_0}}(0)\sum\limits_i {\sum\limits_j {\frac{{{a_{\ell ,i}}i!{a_{\ell  + 1,j}}j!}}{{{{\log }^i}{M_\ell }{{\log }^j}{M_{\ell  + 1}}}}} } {K_{\ell ,\ell  + 1}(\alpha,\beta,i,j)} + O({T^{1 - \varepsilon }}), \nonumber  
\end{align}
where
\begin{align}
 {K_{\ell ,\ell  + 1}}(\alpha ,\beta ,i,j) &= {\left( {\frac{1}{{2\pi i}}} \right)^2}\int_{(\delta )} {\int_{(\delta )} {M_\ell ^sM_{\ell  + 1}^u{L^{2{\ell ^2}}}(f \otimes f,1 + s + u){B_{\alpha ,\beta }}(s,u,0)} }  \nonumber \\
   &\quad \times \frac{1}{{{L^{\ell (\ell  + 1)}}(f \otimes f,1 + 2u){L^\ell }{{(f \otimes f,1 + \alpha  + s)}^\ell }(f \otimes f,1 - \beta  + s)}} \nonumber \\
   &\quad \times \frac{{{L^\ell }{{(f \otimes f,1 + \alpha  + u)}^\ell }(f \otimes f,1 - \beta  + u)}}{{{L^{\ell (\ell  + 1)}}(f \otimes f,1 + 2u)}}\frac{{ds}}{{{s^{i + 1}}}}\frac{{du}}{{{u^{j + 1}}}}. \nonumber 
\end{align}
Before we proceed we note that
\begin{align}
  B_{\alpha ,\beta}(s,s,s) &= \sum\limits_{{h_2}{k_1}l = {h_1}{k_2}} {\frac{{{\mu_{f,\ell} }({h_1}){\mu_{f,\ell  + 1}}({h_2}){\lambda_f^{* \ell - 1}}({k_1})\lambda_f^{* \ell}({k_2}){\sigma _{\alpha , - \beta }}(f,l)}}{{{{({h_1}{h_2}{k_1}{k_2}l)}^{1/2 + s}}}}}  \nonumber \\
   &= \sum\limits_{{h_2}{k_1}l = {h_1}{k_2}} {\frac{{{\mu_{f,\ell} }({h_1}){\mu_{f,\ell  + 1}}({h_2}){\lambda_f^{* \ell - 1}}({k_1})\lambda_f^{* \ell}({k_2}){\sigma _{\alpha , - \beta }}(f,l)}}{{{{({h_1}{k_2})}^{1 + 2s}}}}}  \nonumber \\
   &= \sum\limits_{j = 1}^\infty  \bigg(\sum\limits_{{h_1}{k_2} = j} {\frac{{{\mu_{f,\ell} }({h_1})\lambda_f^{* \ell}({k_2})}}{{{{({h_1}{k_2})}^{1 + 2s}}}}} \bigg)\bigg(\sum\limits_{{h_2}{k_1}l=j}{\sigma _{\alpha , - \beta }}(f,l) {{\mu_{f,\ell  + 1}}({h_2}){\lambda_f^{* \ell - 1}}({k_1})}\bigg)   = 1, \nonumber 
\end{align} 
because the first bracket is $0$ if $j\neq 1$ and $1$ if $j=1$ and since
\[
\sum_{{h_2}{k_1}l=j} {\sigma _{\alpha , - \beta }}(f,l) {{\mu_{f,\ell  + 1}}({h_2}){\lambda_f^{* \ell - 1}}({k_1})}  = 1,
\]
when $j=1$ by the definition of ${\sigma _{\alpha , - \beta }}(f,l)$. This means that $B_{\alpha,\beta}(s,s,s)=1$ for all values of $s$. Let us recall that the Rankin-Selberg convolution $L$-function is given 
\[L(f \otimes g,s) = L(\chi ,2s)\sum\limits_{n = 1}^\infty  {\frac{{{\lambda _f}(n){\lambda _g}(n)}}{{{n^s}}}}, \]
from which we obtain (since $\nu_{\ell+1} \leq \nu_{\ell}$) that 
\begin{align}
  {K_{\ell ,\ell  + 1}}(\alpha ,\beta ,i,j) &= \sum\limits_{n \le {M_{\ell  + 1}}} {\frac{{{{(\lambda _f^2(n))}^{ * 2{\ell ^2}}}}}{n}} {\left( {\frac{1}{{2\pi i}}} \right)^2}\int_{(\delta )} {\int_{(\delta )} {{{\{ {\zeta ^{(N)}}(2(1 + s + u))\} }^{2{\ell ^2}}}} }  \nonumber \\
   &\quad \times \frac{1}{{{L^{\ell (\ell  + 1)}}(f \otimes f,1 + 2u){L^\ell }{{(f \otimes f,1 + \alpha  + s)}^\ell }(f \otimes f,1 - \beta  + s)}} \nonumber \\
   &\quad \times \frac{{{L^\ell }{{(f \otimes f,1 + \alpha  + u)}^\ell }(f \otimes f,1 - \beta  + u)}}{{{L^{\ell (\ell  + 1)}}(f \otimes f,1 + 2u)}} \nonumber \\
   &\quad \times {\left( {\frac{{{M_\ell }}}{n}} \right)^s}{\left( {\frac{{{M_{\ell  + 1}}}}{n}} \right)^u}{B_{\alpha ,\beta }}(s,u,0)\frac{{ds}}{{{s^{i + 1}}}}\frac{{du}}{{{u^{j + 1}}}}. \nonumber 
\end{align}
The double integral $K_{\ell,\ell+1}$ can be computed by similar methods to those employed in the calculation of $K_{\ell,\ell}$. We define the integrand to be
\begin{align}
  &{r_{\ell ,\ell  + 1}}(\alpha ,\beta ,i,j,s,u) = \frac{{{{\{ {\zeta ^{(N)}}(2(1 + s + u))\} }^{2{\ell ^2}}}}}{{{L^{\ell (\ell  + 1)}}(f \otimes f,1 + 2u){L^\ell }{{(f \otimes f,1 + \alpha  + s)}^\ell }(f \otimes f,1 - \beta  + s)}} \nonumber \\
  &\quad \times \frac{{{L^\ell }{{(f \otimes f,1 + \alpha  + u)}^\ell }(f \otimes f,1 - \beta  + u)}}{{{L^{\ell (\ell  + 1)}}(f \otimes f,1 + 2u)}}{\left( {\frac{{{M_\ell }}}{n}} \right)^s}{\left( {\frac{{{M_{\ell  + 1}}}}{n}} \right)^u}\frac{1}{{{s^{i + 1}}}}\frac{1}{{{u^{j + 1}}}}{B_{\alpha ,\beta }}(s,u,0). \nonumber 
\end{align}
Let us follow the strategy of $I_{\ell,\ell}$ and that of Lemma 5.7 of \cite{bcy} by using the zero-free region of $L(f \otimes f,s)$, see \cite{iwanieckowalski}. Since $L(f \otimes f,s)$ does not vanish, we replace the double integrals of $\real(u) = \real(v) = \delta$ by the contour of integration $\gamma$ on page~\pageref{eq:curve_gamma}. We get by the Cauchy residue theorem 
\begin{align}
  &{\left( {\frac{1}{{2\pi i}}} \right)^2}\int_{(\delta )} {\int_{(\delta )} {{r_{\ell ,\ell  + 1}}(\alpha ,\beta ,i,j,s,u)dsdu} }  \nonumber \\
   &= \mathop {\operatorname{Res} }\limits_{s = 0} \frac{1}{{2\pi i}}\int_{\operatorname{Re} (u) = \delta } {{r_{\ell ,\ell  + 1}}(\alpha ,\beta ,i,j,s,u)du}  + {\left( {\frac{1}{{2\pi i}}} \right)^2}\int_{s \in \gamma } {\int_{\operatorname{Re} (u) = \delta } {{r_{\ell ,\ell  + 1}}(\alpha ,\beta ,i,j,s,u)dsdu} }  \nonumber \\
   &= \mathop {\operatorname{Res} }\limits_{s = u = 0} {r_{\ell ,\ell  + 1}}(\alpha ,\beta ,i,j,s,u) + \mathop {\operatorname{Res} }\limits_{s = 0} \frac{1}{{2\pi i}}\int_{u \in \gamma } {{r_{\ell ,\ell  + 1}}(\alpha ,\beta ,i,j,s,u)du}  \nonumber \\
   &\quad + \mathop {\operatorname{Res} }\limits_{u = 0} \frac{1}{{2\pi i}}\int_{s \in \gamma } {{r_{\ell ,\ell  + 1}}(\alpha ,\beta ,i,j,s,u)ds}  + {\left( {\frac{1}{{2\pi i}}} \right)^2}\int_{s \in \gamma } \int_{u \in \gamma } {r_{\ell ,\ell  + 1}}(\alpha ,\beta ,i,j,s,u)dsdu .  \nonumber  
\end{align}
Again, the first estimation will be that of ${\operatorname{Res} _{s = 0}}\frac{1}{{2\pi i}}\int_{s \in \gamma } {{r_{\ell ,\ell  + 1}}(\alpha ,\beta ,i,j,s,u)ds}$. We start by writing the residue as a contour integral over a small circle of radius $1/L$ centered at $0$, i.e.
\begin{align}
  &\mathop {\operatorname{Res} }\limits_{s = 0} \frac{1}{{2\pi i}}\int_{u \in \gamma } {r_{\ell ,\ell  + 1}}(\alpha ,\beta ,i,j,s,u)du  \nonumber \\
  & = {\left( {\frac{1}{{2\pi i}}} \right)^2}\int_{u \in \gamma } {{{\left( {\frac{{{M_{\ell  + 1}}}}{n}} \right)}^u}\frac{{{L^\ell }(f \otimes f,1 + \alpha  + u){L^\ell }(f \otimes f,1 - \beta  + u)}}{{{L^{\ell (\ell  + 1)}}(f \otimes f, 1 + 2u)}}}  \nonumber \\
  &\quad \times \oint_{D(0,{L^{ - 1}})} {{{\left( {\frac{{{M_\ell }}}{n}} \right)}^s}\frac{{{{\{ {\zeta ^{(N)}}(2(1 + s + u))\} }^{2{\ell ^2}}}{B_{\alpha ,\beta }}(s,u,0)}}{{{L^{\ell (\ell  - 1)}}(f \otimes f, 1 + 2s){L^\ell }(f \otimes f,1 + \alpha  + s){L^\ell }(f \otimes f,1 - \beta  + s)}}\frac{{ds}}{{{s^{i + 1}}}}} \frac{{du}}{{{u^{j + 1}}}}. \nonumber  
\end{align}
Next we use the fact that $\zeta^{(N)}(2(1+s+u))B_{\alpha,\beta}(s,u,0) \ll 1$ in this contour of integration and
\begin{align}
  \frac{1}{{{L^{\ell (\ell  - 1)}}(f \otimes f, 1 + 2s){L^\ell }(f \otimes f,1 + \alpha  + s){L^\ell }(f \otimes f,1 - \beta  + s)}}  &\ll {(2s)^{\ell (\ell  - 1)}}{(\alpha  + s)^\ell }{( - \beta  + s)^\ell } \nonumber \\
	&\ll L^{ - \ell (\ell  + 1)}, \nonumber 
\end{align}
since $s \asymp 1/L$, to write
\begin{align}
  &\mathop {\operatorname{Res} }\limits_{s = 0} \frac{1}{{2\pi i}}\int_{u \in \gamma } {{r_{\ell ,\ell  + 1}}(\alpha ,\beta ,i,j,s,u)du}  \nonumber \\
  &\quad \ll {L^{i - \ell (\ell  + 1)}}\int_{u \in \gamma } {{{\left( {\frac{{{M_{\ell  + 1}}}}{n}} \right)}^{\operatorname{Re} (u)}}\left| {\frac{{{L^\ell }(f \otimes f,1 + \alpha  + u){L^\ell }(f \otimes f,1 - \beta  + u)}}{{{L^{\ell (\ell  + 1)}}(f \otimes f, 1 + 2u)}}} \right|\frac{{du}}{{|u{|^{j + 1}}}}}.  \nonumber 
\end{align}
The novelty is that in addition to the bound \eqref{boundsLfunctions1}, we shall also use \cite[Chapter 5]{iwanieckowalski}
\begin{align} \label{boundsLfunctions2}
L(f \otimes f,\sigma  + i\tau ) \ll _{N,\varepsilon }|\tau {|^{4(\alpha'  + \varepsilon )}},\quad \alpha'  = \max \{ \tfrac{1}{2}(1 - \sigma ),0\},
\end{align}
for all $\varepsilon >0$ and where $N$ is the level of the $L$-function. %and
%\[\frac{1}{{L(f \otimes f,\sigma  + i\tau )}} \ll \log |\tau |,\]
This enables us to obtain
\begin{align}
  \int_{u \in \gamma } &{{\left( {\frac{{{M_{\ell  + 1}}}}{n}} \right)}^{\operatorname{Re} (u)}} \left| {\frac{{{L^\ell }(f \otimes f,1 + \alpha  + u){L^\ell }(f \otimes f,1 - \beta  + u)}}{{{L^{\ell (\ell  + 1)}}(f \otimes f,1 + 2u)}}} \right|\frac{{du}}{{|u{|^{j + 1}}}}  \nonumber \\
   &\ll 
   \int_{|\tau | \ge Y} {\frac{{{{\log }^{\ell (\ell  + 1)}}|\tau |}}{{|\tau {|^{j + 1-\varepsilon}}}}d\tau }  
   + 
   {(\log Y)^{\ell (\ell  + 1)}}\int_{ - c/\log Y}^0 {\frac{{d\sigma }}{{|\sigma  + iY{|^{j + 1 - 4c\ell /\log Y-\varepsilon}}}}}  \nonumber \\
   &\quad + {\left( {\frac{{{M_{\ell  + 1}}}}{n}} \right)^{ - c/\log Y}}{(\log Y)^{\ell (\ell  + 1)}} \nonumber \\
	 &\quad \times \bigg( \int_{ c/\log Y \le |\tau | \le Y}   
   + \int_{0 \le |\tau | \le c/\log Y}   \bigg)\frac{{d\tau }}{{|\tau  - ic/\log Y{|^{j + 1 - 4c\ell /\log Y-\varepsilon}}}} \nonumber \\
   &\ll \frac{{{{(\log Y)}^{\ell (\ell  + 1)}}}}{{{Y^j}}} + {(\log Y)^{j  +\ell(\ell+1)+\varepsilon }}{\left( {\frac{{{M_{\ell  + 1}}}}{n}} \right)^{ - c/\log Y}}. \nonumber 
\end{align}
This means that
\[\mathop {\operatorname{Res} }\limits_{s = 0} \frac{1}{{2\pi i}}\int_{u \in \gamma } {{r_{\ell ,\ell  + 1}}(\alpha ,\beta ,i,j,s,u)du}  
\ll 
{L^{i - \ell (\ell  + 1)}}{(\log Y)^{\ell (\ell  + 1)+\varepsilon}}\bigg( {\frac{1}{{{Y^j}}} + {{(\log Y)}^j}{{\left( {\frac{{{M_{\ell  + 1}}}}{n}} \right)}^{ - c/\log Y}}} \bigg).\]
By using a similar technique, one gets
\[
\mathop {\operatorname{Res} }\limits_{u = 0} \frac{1}{{2\pi i}}\int_{s \in \gamma } {{r_{\ell ,\ell  + 1}}(\alpha ,\beta ,i,j,s,u)ds}  \ll {L^{j - \ell (\ell  - 1) }}{(\log Y)^{\ell (\ell  + 1)+\varepsilon}}\bigg( {\frac{1}{{{Y^i}}} + {{(\log Y)}^i}{{\left( {\frac{{{M_{\ell }}}}{n}} \right)}^{ - c/\log Y}}} \bigg).
\]
Now we bound the double integral over $s \in \gamma$ and $u \in \gamma$, i.e.
\begin{align}
  {\left( {\frac{1}{{2\pi i}}} \right)^2}&\int_{s \in \gamma } {\int_{u \in \gamma } {{r_{\ell ,\ell  + 1}}(\alpha ,\beta ,i,j,s,u)dsdu} }  \nonumber \\
  &\ll 
  \int_{s \in \gamma } {{{\left( {\frac{{{M_\ell }}}{n}} \right)}^{\operatorname{Re} (s)}}\frac{1}{{{L^{\ell (\ell  - 1)}}(f \otimes f, 1 + 2s){L^\ell }(f \otimes f,1 + \alpha  + s){L^\ell }(f \otimes f,1 - \beta  + s)}}\frac{{ds}}{{{s^{i + 1}}}}}  \nonumber \\
  &\quad \times \int_{u \in \gamma } {{{\left( {\frac{{{M_{\ell  + 1}}}}{n}} \right)}^{\operatorname{Re}(u)}}\frac{{{L^\ell }(f \otimes f,1 + \alpha  + u){L^\ell }(f \otimes f,1 - \beta  + u)}}{{{L^{\ell (\ell  + 1)}}(f \otimes f, 1 + 2u)}}\frac{{du}}{{{u^{j + 1}}}}}  \nonumber \\
  & \ll 
  \bigg( {\frac{{{{\log }^{\ell (\ell  + 1)}}Y}}{{{Y^j}}} + {{(\log Y)}^{j+\ell(\ell+1)}}{{\left( {\frac{{{M_{\ell  + 1}}}}{n}} \right)}^{ - c/\log Y}}} \bigg) \nonumber \\
	&\quad \times \bigg( {\frac{{{{\log }^{\ell (\ell  + 1)}}Y}}{{{Y^i}}} + {{(\log Y)}^{i+\ell(\ell+1)}}{{\left( {\frac{{{M_\ell }}}{n}} \right)}^{ - c/\log Y}}} \bigg) \nonumber \\
  & \ll 
  {(\log Y)^{2\ell (\ell  + 1)}} \bigg( {\frac{1}{{{Y^{i + j}}}} + {{(\log Y)}^{i + j}}{{\left( {\frac{{{M_\ell }}}{n}} \right)}^{ - c/\log Y}}} \bigg), \nonumber 
\end{align}
since $\max(M_{\ell},M_{\ell+1})=M_{\ell}$. Let us recall that
\[\Omega (q) = \sum\limits_{n \le {M_{\ell+1}}} {\frac{{{{(\lambda _f^2(n))}^{ * 2{\ell ^2}}}}}{n} \bigg( {\frac{1}{{{Y^q}}} + {{(\log Y)}^q}{{\left( {\frac{{{M_{\ell  }}}}{n}} \right)}^{ - c/\log Y}}} \bigg)} .\]
We use Lemma~\ref{eulermaclaurinlemma} regarding the convolution $\lambda_f^2(n)^{*k}$, to write
\begin{align}
  \Omega (q) &= \sum\limits_{n \le {M_{\ell  + 1}}} {\frac{{{{(\lambda _f^2(n))}^{ * 2{\ell ^2}}}}}{n} \bigg( {\frac{1}{{{Y^q}}} + {{(\log Y)}^q}{{\left( {\frac{{{M_\ell }}}{n}} \right)}^{ - c/\log Y}}} \bigg)}  \nonumber \\
   &\ll 
   \frac{1}{Y^q} \sum\limits_{n \le M_{\ell  + 1}} \frac{(\lambda _f^2(n))^{ * 2\ell ^2}}{n}     
   + 
   (\log Y)^q \sum\limits_{n \le M_{\ell  + 1}} \frac{(\lambda _f^2(n))^{ * 2\ell ^2}}{n}
   \left(\frac{M_\ell}{n}  \right)^{ - c/\log Y}     
   \nonumber \\
   &\ll 
   \frac{\log^{2{\ell ^2}}(M_{\ell  + 1})}{Y^q}  
   +  
   (\log Y)^q 
   (\log M_{\ell  + 1})^{ 2\ell ^2}.\nonumber  
\end{align}
The choice of $Y$ has to be such that $Y=o(T)$, specifically we take $Y=\log T =L$. 
Thus we get $\Omega (q)\ll L^{2\ell^2+\epsilon}$.
Using that $i\ge (\ell+1)^2-(\ell+1)+1$ and $j\ge \ell^2-\ell+1$ and putting all pieces together, we obtain
\begin{align}
  K_{\ell ,\ell  + 1}(\alpha ,\beta ,i,j) &= \sum\limits_{n \le {M_{\ell  + 1}}} {\frac{{{{(\lambda _f^2(n))}^{ * 2{\ell ^2}}}}}{n}\mathop {\operatorname{Res} }\limits_{s = u = 0} {r_{\ell ,\ell  + 1}}(\alpha ,\beta ,i,j,s,u)}  \nonumber \\
   &\quad + O({L^{i - \ell (\ell  + 1)}}\Omega (j)\log Y + {L^{j -\ell (\ell  - 1)}}\Omega (i)\log Y + \Omega (i + j){\log ^2}Y) \nonumber \\
   &= \sum\limits_{n \le {M_{\ell  + 1}}} {\frac{{{{(\lambda _f^2(n))}^{ * 2{\ell ^2}}}}}{n}\mathop {\operatorname{Res} }\limits_{s = u = 0} {r_{\ell ,\ell  + 1}}(\alpha ,\beta ,i,j,s,u)}  + O( L^{i+j-1} )  . \nonumber 
\end{align}
In order to get the main term of the lemma we need to compute the residue at $s=u=0$ of $r_{\ell,\ell+1}$. 
This is accomplished by expressing the residue as two contour integrals over small circles of radii $1/L$ centered at $0$. In other words,
\begin{align}
  \mathop {\operatorname{Res} }\limits_{s = u = 0} {r_{\ell ,\ell  + 1}}(\alpha ,\beta ,i,j,s,u) &= {\left( {\frac{1}{{2\pi i}}} \right)^2}\oint_{D(0,{L^{ - 1}})} {\oint_{D(0,{L^{ - 1}})} {{{\left( {\frac{{{M_\ell }}}{n}} \right)}^s}{{\left( {\frac{{{M_{\ell  + 1}}}}{n}} \right)}^u}} }  \nonumber \\
   &\quad \times \frac{{{{\{ {\zeta ^{(N)}}(2(1 + s + u))\} }^{2{\ell ^2}}}}}{{{L^{\ell (\ell  + 1)}}(f \otimes f,1 + 2u){L^\ell }{{(f \otimes f,1 + \alpha  + s)}^\ell }(f \otimes f,1 - \beta  + s)}} \nonumber \\
   &\quad \times \frac{{{L^\ell }{{(f \otimes f,1 + \alpha  + u)}^\ell }(f \otimes f,1 - \beta  + u)}}{{{L^{\ell (\ell  + 1)}}(f \otimes f,1 + 2u)}}{B_{\alpha ,\beta }}(s,u,0)\frac{{ds}}{{{s^{i + 1}}}}\frac{{du}}{{{u^{j + 1}}}}. \nonumber 
\end{align}
Now we must separate the complex variables $s$ and $u$ to decouple these two integrals. To do this, we recall that $s \asymp u \asymp 1/L$ and hence
\begin{align}
  {\zeta ^{(N)}}{(2(1 + s + u))^2} &= {\zeta ^{(N)}}{(2)^2} + O(1/L) \nonumber \\
  {B_{\alpha ,\beta }}(s,u,0) &= {B_{0,0}}(0,0,0) + O(1/L) \nonumber \\
  \frac{1}{{L(f \otimes f,1 + \alpha  + s)}} &= \frac{{\alpha  + s}}{{{{\operatorname{Res} }_{s = 1}}L(f \otimes f,s)}}(1 + O(1/L)) \nonumber \\
  L(f \otimes f,1 + \alpha  + u) &= \frac{{{{\operatorname{Res} }_{u = 1}}L(f \otimes f,u)}}{{\alpha  + u}}(1 + O(1/L)), \nonumber  
\end{align}
and we recall that we had shown that $B_{0,0}(0,0,0)=1$. Thus
\begin{align}
  &\frac{{{{\{ {\zeta ^{(N)}}(2(1 + s + u))\} }^{2{\ell ^2}}}{L^\ell }(f \otimes f,1 + \alpha  + u){L^\ell }(f \otimes f,1 - \beta  + u){B_{\alpha ,\beta }}(s,u,0)}}{{{L^{\ell (\ell  - 1)}}(f \otimes f,1 + 2s){L^{\ell (\ell  + 1)}}(f \otimes f,1 + 2u){L^\ell }(f \otimes f,1 + \alpha  + s){L^\ell }(f \otimes f,1 - \beta  + s)}} \nonumber \\
  & \quad = \frac{{{{\{ {\zeta ^{(N)}}(2)\} }^{2{\ell ^2}}}}}{{({{\operatorname{Res} }_{s = 1}}{{L(f \otimes f,s))}^{2{\ell ^2}}}}}\frac{{{{(\alpha  + s)}^\ell }{{( - \beta  + s)}^\ell }}}{{{{(\alpha  + u)}^\ell }{{( - \beta  + u)}^\ell }}}{(2s)^{\ell (\ell  - 1)}}{(2u)^{\ell (\ell  + 1)}} + O(L^{-2\ell^2-1}). \nonumber  
\end{align}
Indeed, we now get the product of two cleanly separated integrals
\begin{align}
  \mathop {\operatorname{Res} }\limits_{s = u = 0} {r_{\ell ,\ell  + 1}}(\alpha ,\beta ,i,j,s,u) &= \frac{{{{\{ {\zeta ^{(N)}}(2)\} }^{2{\ell ^2}}}}}{{({{\operatorname{Res} }_{s = 1}}{{L(f \otimes f,s))}^{2{\ell ^2}}}}} \nonumber \\
   &\quad \times \bigg( {\frac{{{2^{\ell (\ell  - 1)}}}}{{2\pi i}}\oint_{D(0,{L^{ - 1}})} {{{\left( {\frac{{{M_\ell }}}{n}} \right)}^s}{{(\alpha  + s)}^\ell }{{( - \beta  + s)}^\ell }\frac{{ds}}{{{s^{i + 1 - \ell (\ell  - 1)}}}}} } \bigg) \nonumber \\
   &\quad \times \bigg( {\frac{{{2^{\ell (\ell  + 1)}}}}{{2\pi i}}\oint_{D(0,{L^{ - 1}})} {{{\left( {\frac{{{M_{\ell  + 1}}}}{n}} \right)}^u}\frac{1}{{{{(\alpha  + u)}^\ell }{{( - \beta  + u)}^\ell }}}\frac{{du}}{{{u^{j + 1 - \ell (\ell  + 1)}}}}} } \bigg) \nonumber \\
	 &\quad + O(L^{i+j-1}). \nonumber  
\end{align}
We shall compute these integrals by the use of Cauchy's integral theorem. We will proceed in more generality than strictly needed. First, we have
\begin{align}
  \frac{1}{{2\pi i}}\oint_{D(0,{L^{ - 1}})} {{{\left( {\frac{{{M_\ell }}}{n}} \right)}^s}{{(\alpha  + s)}^p}{{( - \beta  + s)}^q}\frac{{ds}}{{{s^{k + 1}}}}} &= \frac{{{d^{p + q}}}}{{d{x^p}d{y^q}}}\bigg( {\frac{{{e^{\alpha x - \beta y}}}}{{2\pi i}}\oint_{D(0,{L^{ - 1}})} {{{\left( {\frac{{{M_\ell }}}{n}{e^{x + y}}} \right)}^s}\frac{{ds}}{{{s^{k + 1}}}}} } \bigg)\bigg|_{x = y = 0} \nonumber \\
   &= \frac{1}{{k!}}\frac{{{d^{p + q}}}}{{d{x^p}d{y^q}}}\bigg( {{e^{\alpha x - \beta y}}{{\left( {x + y + \log \frac{{{M_\ell }}}{n}} \right)}^k}} \bigg)\bigg|_{x = y = 0}. \nonumber  
\end{align}
Our case of interest naturally follows by taking $p=q=\ell$ and $k=i-\ell(\ell-1)$. For the $u$-integral we proceed in a slightly different way. We will use the equality
\begin{align} \label{trick1}
\int_{1/q}^1 {{r^{\alpha  + u - 1}}{{\log }^\tau }rdr}  = \frac{{{{( - 1)}^\tau }\tau !}}{{{{(\alpha  + u)}^{\tau  + 1}}}} - \frac{{{q^{ - \alpha  - u}}}}{{{{(\alpha  + u)}^{\tau  + 1}}}}P(u,\alpha ,\log q),
\end{align}
which is valid for all complex numbers $\alpha, u$, positive $q$ and $\tau = 0,1,2,\cdots$. Here $P$ is a polynomial in $\log q$ of degree $\tau-1$. We temporarily set $q= M_{\ell+1}/n$ so that
\begin{align}
  \frac{1}{{2\pi i}}\oint_{D(0,{L^{ - 1}})} &{q^u}\frac{1}{{{{(\alpha  + u)}^m}{{( - \beta  + u)}^n}}}\frac{{du}}{{{u^{k + 1}}}}  \nonumber \\
   &= \frac{{{{( - 1)}^{m - 1}}}}{{(m - 1)!}}\frac{1}{{2\pi i}}\oint_{D(0,{L^{ - 1}})} {{q^u}\frac{1}{{{{( - \beta  + u)}^n}}}\int_{1/q}^1 {{r^{\alpha  + u - 1}}{{\log }^{m - 1}}rdr} \frac{{du}}{{{u^{k + 1}}}}}  + E(q), \nonumber  
\end{align}
where $E(q)$ is the term arising from the second part of \eqref{trick1}, i.e. 
\begin{align} \label{vanishingintegral}
E(q) =  - \frac{{{{( - 1)}^{m - 1}}}}{{(m - 1)!}}{q^{ - \alpha }}\frac{1}{{2\pi i}}\oint {\frac{{P(u,\alpha ,\log q)}}{{{{(\alpha  + u)}^m}{{( - \beta  + u)}^n}}}} \frac{{du}}{{{u^{k+1}}}}.
\end{align}
It can be seen, by taking the contour to be arbitrarily large, that this term vanishes. Reversing the order of integration in the main term yields
\begin{align}
  \frac{1}{{2\pi i}}\oint_{D(0,{L^{ - 1}})} &{q^u}\frac{1}{{{{(\alpha  + u)}^m}{{( - \beta  + u)}^n}}}\frac{{du}}{{{u^{k + 1}}}}  \nonumber \\
   &= \frac{{{{( - 1)}^{m - 1}}}}{{(m - 1)!}}\int_{1/q}^1 {r^{\alpha  - 1}}{{\log }^{m - 1}}r\frac{1}{{2\pi i}}\oint_{D(0,{L^{ - 1}})} {{{(qr)}^u}\frac{1}{{{{( - \beta  + u)}^n}}}\frac{{du}}{{{u^{k + 1}}}}} dr . \nonumber 
\end{align}
Applying again \eqref{trick1} but with the lower boundary of integration at $1/(qr)$ and seeing that the second term of \eqref{trick1} will also vanish by the same reason as the previous second term \eqref{vanishingintegral}, we then obtain
\begin{align}
  \frac{{{{( - 1)}^{m - 1}}}}{{(m - 1)!}} &\int_{1/q}^1 {{r^{\alpha  - 1}}\frac{1}{{2\pi i}}\oint_{D(0,{L^{ - 1}})} {{{(qr)}^u}\frac{1}{{{{( - \beta  + u)}^m}}}\frac{{du}}{{{u^{j - k + 1}}}}} {{\log }^{m - 1}}rdr}  \nonumber \\
   &= \frac{{{{( - 1)}^{m + n}}}}{{(m - 1)!(n - 1)!}}\int_{1/q}^1 {\int_{1/qr}^1 {{r^{\alpha  - 1}}{t^{-\beta  - 1}}{{\log }^{m - 1}}r{{\log }^{n - 1}}t} \frac{1}{{2\pi i}}\oint {{{(qrt)}^u}\frac{{du}}{{{u^{k + 1}}}}} dtdr}  \nonumber \\
   &= \frac{{{{( - 1)}^{m + n}}}}{{k!(m - 1)!(n - 1)!}}\int_{1/q}^1 \int_{1/qr}^1 {{r^{\alpha  - 1}}{t^{-\beta  - 1}}{{\log }^{m - 1}}r{{\log }^{n - 1}}t} {{\left( {\log rt\frac{{{M_2}}}{n}} \right)}^{k}}dtdr.  \nonumber  
\end{align}
The last step is to make the change of variables $r=q^{-a}$ and $t=q^{-b}$ so that the above becomes
\[\frac{{{{\log }^{k + n + m}}q}}{{k!(m - 1)!(n - 1)!}}\iint_{\substack{a + b \le 1 \\ a,b \ge 0}} {{(1 - a - b)}^{k}}{{\left( {\frac{{{M_2}}}{n}} \right)}^{ - a\alpha  + b\beta }}{a^{m - 1}}{b^{n - 1}}dadb.\]
Our case of interest in this setting follows by taking $m=n=\ell$ and $k=j-\ell(\ell+1)$. The end result of this reasoning is that
\begin{align}
  \mathop {\operatorname{Res} }\limits_{s = u = 0} {r_{\ell ,\ell  + 1}}(\alpha ,\beta ,i,j,s,u) &= \frac{{{{\{ {\zeta ^{(N)}}(2)\} }^{2{\ell ^2}}}}}{{({{\operatorname{Res} }_{s = 1}}{{L(f \otimes f,s))}^{2{\ell ^2}}}}} \nonumber \\
   &\quad \times \frac{{{2^{\ell (\ell  - 1)}}}}{{(i-\ell(\ell-1))!}}\frac{{{d^{2\ell }}}}{{d{x^\ell }d{y^\ell }}} \bigg( {{e^{\alpha x - \beta y}}{{\left( {x + y + \log \frac{{{M_\ell }}}{n}} \right)}^{i - \ell (\ell  - 1)}}} \bigg)\bigg|_{x = y = 0} \nonumber \\
   &\quad \times \frac{{{2^{\ell (\ell  + 1)}}{{\log }^{j - \ell (\ell  + 1) + 2\ell }(M_{\ell+1}/n)}}}{{(j - \ell (\ell  + 1))!}} \nonumber \\
	 &\quad \times \iint_{\substack{a + b \le 1 \\ a, b \ge 0}} {{(1 - a - b)}^{^{j - \ell (\ell  + 1)}}}{{\left( {\frac{{{M_{\ell  + 1}}}}{n}} \right)}^{ - a\alpha  + b\beta }}{{(ab)}^{\ell  - 1}}dadb  +  O(L^{i+j-2\ell^2-1}). \nonumber
\end{align}
When we go back to $K_{\ell,\ell+1}$ with this new information we get
\begin{align}
  K_{\ell ,\ell  + 1}(\alpha ,\beta ,i,j,s,u) &= {\left( {\frac{1}{{2\pi i}}} \right)^2}\int_{(\delta )} \int_{(\delta )} {r_{\ell ,\ell  + 1}}(\alpha ,\beta ,i,j,s,u)dsdu \nonumber \\
	 &= \frac{{{{\{ {\zeta ^{(N)}}(2)\} }^{2{\ell ^2}}}}}{{({{\operatorname{Res} }_{s = 1}}{{L(f \otimes f,s))}^{2{\ell ^2}}}}}  \frac{{{2^{\ell (\ell  - 1)}}}}{{(i-\ell(\ell-1))!}}\frac{{{d^{2\ell }}}}{{d{x^\ell }d{y^\ell }}}\frac{{{2^{\ell (\ell  + 1)}}{{\log }^{j - \ell (\ell  + 1) + 2\ell }(M_{\ell+1}/n)}}}{{(j - \ell (\ell  + 1))!}} \nonumber \\
   &\quad \times \sum\limits_{n \le {M_{\ell  + 1}}} {\frac{{{{(\lambda _f^2(n))}^{ * 2{\ell ^2}}}}}{n}} \bigg( {{e^{\alpha x - \beta y}}{{\left( {x + y + \log \frac{{{M_\ell }}}{n}} \right)}^{i - \ell (\ell  - 1)}}} \bigg)\bigg|_{x = y = 0} \nonumber \\
   &\quad \times \iint_{\substack{a + b \le 1 \\ a, b \ge 0}} {{{(1 - a - b)}^{^{j - \ell (\ell  + 1)}}}{{\left( {\frac{{{M_{\ell  + 1}}}}{n}} \right)}^{ - a\alpha  + b\beta }}{{(ab)}^{\ell  - 1}}dadb}  + O(L^{i+j-2\ell^2-1}). \nonumber  
\end{align}
Recall that in the last expression for $I_{\ell,\ell+1}$, we have sums over $i$ and over $j$. The sum over $i$ is
\begin{align}
  &\sum\limits_i {\frac{{{a_{\ell ,i}}}}{{{{\log }^i}{M_\ell }}}\frac{{i!}}{{(i - \ell (\ell  - 1))!}}{{\left( {x + y + \log \frac{{{M_\ell }}}{n}} \right)}^{i - \ell (\ell  - 1)}}}  \nonumber \\
   &= \frac{1}{{{{\log }^{\ell (\ell  - 1)}}{M_\ell }}}\sum\limits_i {{a_{\ell ,i}}i(i - 1) \cdots (i - \ell (\ell  - 1) + 1){{\left( {\frac{{x + y}}{{\log {M_\ell }}} + \frac{{\log ({M_\ell }/n)}}{{\log {M_\ell }}}} \right)}^{i - \ell (\ell  - 1)}}}  \nonumber \\
   &= \frac{1}{{{{\log }^{\ell (\ell  - 1)}}{M_\ell }}}P_\ell ^{(\ell (\ell  - 1))}\left( {\frac{{x + y}}{{\log {M_\ell }}} + \frac{{\log ({M_\ell }/n)}}{{\log {M_\ell }}}} \right), \nonumber 
\end{align}
whereas the sum over $j$ is
\begin{align}
  &\sum\limits_j {\frac{{{a_{\ell  + 1,j}}j!}}{{{{\log }^j}{M_{\ell  + 1}}}}\frac{{{{\log }^{j - \ell (\ell  + 1) + 2\ell }}({M_{\ell  + 1}}/n)}}{{(j - \ell (\ell  - 1))!}}{{(1 - a - b)}^{j - \ell (\ell  + 1)}}}  \nonumber \\
  & = \frac{{{{\log }^{2\ell }}({M_{\ell  + 1}}/n)}}{{{{\log }^{\ell (\ell  + 1)}}{M_{\ell  + 1}}}}\sum\limits_j {{a_{\ell  + 1,j}}j(j - 1) \cdots (j - \ell (\ell  - 1) + 1){{\left( {(1 - a - b)\frac{{\log ({M_{\ell  + 1}}/n)}}{{\log {M_{\ell  + 1}}}}} \right)}^{j - \ell (\ell  + 1)}}}  \nonumber \\
  & = \frac{{{{\log }^{2\ell }}({M_{\ell  + 1}}/n)}}{{{{\log }^{\ell (\ell  + 1)}}{M_{\ell  + 1}}}}P_{\ell  + 1}^{(\ell (\ell  + 1))}\left( {(1 - a - b)\frac{{\log ({M_{\ell  + 1}}/n)}}{{\log {M_{\ell  + 1}}}}} \right). \nonumber  
\end{align}
Plugging these results into $I_{\ell,\ell+1}$ we see that
\begin{align}
  I_{\ell ,\ell  + 1}(\alpha ,\beta ) &= \frac{{{2^{2{\ell ^2}}}\widehat {{w_0}}(0)}}{{{{\log }^{\ell (\ell  - 1)}}{M_\ell }{{\log }^{\ell (\ell  + 1)}}{M_{\ell  + 1}}}}{\bigg( {\frac{{\{ {\zeta ^{(N)}}(2)\} }}{{{{\operatorname{Res} }_{s = 1}}L(f \otimes f,s)}}} \bigg)^{2{\ell ^2}}}\frac{{{d^{2\ell }}}}{{d{x^\ell }d{y^\ell }}} \nonumber \\
   &\quad \times {e^{\alpha x - \beta y}}\sum\limits_{n \le {M_{\ell  + 1}}} {\frac{{{{(\lambda _f^2(n))}^{ * 2{\ell ^2}}}}}{n}{{\log }^{2\ell }}({M_{\ell  + 1}}/n)P_\ell ^{(\ell (\ell  - 1))}\left( {\frac{{x + y}}{{\log {M_\ell }}} + \frac{{\log ({M_\ell }/n)}}{{\log {M_\ell }}}} \right){|_{x = y = 0}}}  \nonumber \\
   &\quad \times \iint_{\substack{a + b \le 1 \\ a, b \ge 0}} {P_{\ell  + 1}^{(\ell (\ell  + 1))}\left( {(1 - a - b)\frac{{\log ({M_{\ell  + 1}}/n)}}{{\log {M_{\ell  + 1}}}}} \right){{\left( {\frac{{{M_{\ell  + 1}}}}{n}} \right)}^{ - a\alpha  + b\beta }}{{(ab)}^{\ell  - 1}}dadb}  + O(T/L). \nonumber
\end{align}
Making the changes $x \to x / \log M_{\ell}$ and $y \to y / \log M_{\ell}$ puts this in the more comfortable form
\begin{align}
  I_{\ell ,\ell  + 1}(\alpha ,\beta ) &= \frac{{{2^{2{\ell ^2}}}{T^{ -2 \beta }}\widehat w(0)}}{{{{\log }^{\ell (\ell  - 1) + 2\ell }}{M_\ell }{{\log }^{\ell (\ell  + 1) - 2\ell }}{M_{\ell  + 1}}}}{\bigg( {\frac{{\{ {\zeta ^{(N)}}(2)\} }}{{{{\operatorname{Res} }_{s = 1}}L(f \otimes f,s)}}} \bigg)^{2{\ell ^2}}}\frac{{{d^{2\ell }}}}{{d{x^\ell }d{y^\ell }}} \nonumber \\
   &\quad \times M_\ell ^{\alpha x - \beta y}\sum\limits_{n \le {M_{\ell  + 1}}} {\frac{{{{(\lambda _f^2(n))}^{ * 2{\ell ^2}}}}}{n}\frac{{{{\log }^{2\ell }}({M_{\ell  + 1}}/n)}}{{{{\log }^{2\ell }}{M_{\ell  + 1}}}}P_\ell ^{(\ell (\ell  - 1))}\left( {x + y + \frac{{\log ({M_\ell }/n)}}{{\log {M_\ell }}}} \right){|_{x = y = 0}}}  \nonumber \\
   &\quad \times \iint_{\substack{a + b \le 1 \\ a, b \ge 0}} {P_{\ell  + 1}^{(\ell (\ell  + 1))}\left( {(1 - a - b)\frac{{\log ({M_{\ell  + 1}}/n)}}{{\log {M_{\ell  + 1}}}}} \right){{\left( {\frac{{{M_{\ell  + 1}}}}{n}} \right)}^{ - a\alpha  + b\beta }}{{(ab)}^{\ell  - 1}}dadb}  + O(T/L). \nonumber  
\end{align}
The sum over $n$ is computed by the use of Lemma \ref{eulermaclaurinlemma} with $k=2\ell^2$, $s=-\alpha a + b \beta$, $x=M_{\ell}$, $z=M_{\ell+1}$, $F(u)=P_{\ell}^{(\ell(\ell-1))}(x+y+u)$, $H(u)=u^{2\ell} P_{\ell+1}^{(\ell(\ell+1))} ((1-a-b)u)$. The result is
\begin{align}
  &\sum\limits_{n \le {M_{\ell  + 1}}} {\frac{{{{(\lambda _f^2(n))}^{ * 2\ell^2}}}}{{{n^{1 - a\alpha  + b\beta }}}}{{\left( {\frac{{\log ({M_{\ell  + 1}}/n)}}{{\log {M_{\ell  + 1}}}}} \right)}^{2\ell }}}  \nonumber \\
   &\quad \times P_\ell ^{(\ell (\ell  - 1))}\left( {x + y + \frac{{\log ({M_\ell }/n)}}{{\log {M_\ell }}}} \right)P_{\ell  + 1}^{(\ell (\ell  + 1))}\left( {(1 - a - b)\frac{{\log ({M_{\ell  + 1}}/n)}}{{\log {M_{\ell  + 1}}}}} \right) \nonumber \\
   & = {\left( {\frac{{{{\operatorname{Res} }_{s = 1}}L(f \otimes f,s)}}{{\{ {\zeta ^{(N)}}(2)\} }}} \right)^{2{\ell ^2}}}\frac{{{{\log }^{2{\ell ^2}}}{M_{\ell  + 1}}}}{{(2{\ell ^2} - 1)!M_{\ell  + 1}^{ - a\alpha  + b\beta }}} \nonumber \\
   &\quad \times \int_0^1 {{{(1 - u)}^{2{\ell ^2} - 1}}P_\ell ^{(\ell (\ell  - 1))}\left( {x + y + 1 - (1 - u)\frac{{\log {M_{\ell  + 1}}}}{{\log {M_\ell }}}} \right)}  \nonumber \\
   &\quad \times {u^{2\ell }}P_{\ell  + 1}^{(\ell (\ell  + 1))}((1 - a - b)u)M_{\ell  + 1}^{u( - a\alpha  + b\beta )}du + O(L^{2\ell^2-1}). \nonumber  
\end{align}
This means that we are left with
\begin{align}
  I_{\ell ,\ell  + 1}(\alpha ,\beta ) &= \frac{{{2^{2{\ell ^2}}}{T^{ - 2\beta }}\widehat w(0)}}{{(2{\ell ^2} - 1)!}}{\left( {\frac{{\log {M_{\ell  + 1}}}}{{\log {M_\ell }}}} \right)^{\ell (\ell  + 1)}}\frac{{{d^{2\ell }}}}{{d{x^\ell }d{y^\ell }}} \nonumber \\
   &\quad \times {{M_\ell ^{\alpha x - \beta y}}}\iint_{\substack{a + b \le 1 \\ a, b \ge 0}} {\int_0^1 {{{(1 - u)}^{2{\ell ^2} - 1}}P_\ell ^{(\ell (\ell  - 1))}\left( {x + y + 1 - (1 - u)\frac{{\log {M_{\ell  + 1}}}}{{\log {M_\ell }}}} \right)} } \nonumber \\
  &\quad \times {(ab)^{\ell  - 1}}P_{\ell  + 1}^{(\ell (\ell  + 1))}((1 - a - b)u)M_{\ell  + 1}^{u( - a\alpha  + b\beta )}{u^{2\ell }}dudadb \bigg|_{x = y = 0} + O(T/L). \nonumber  
\end{align}
Setting $M_{\ell} = T^{\nu_\ell}$ and $M_{\ell+1} = T^{\nu_{\ell+1}}$ we obtain Lemma \ref{cellell1lemma}, i.e.
\begin{align}
  c_{\ell ,\ell  + 1}(\alpha ,\beta ) &= \frac{{{2^{2{\ell ^2}}}}}{{(2{\ell ^2} - 1)!}}{\left( {\frac{{{\nu _{\ell  + 1}}}}{{{\nu _\ell }}}} \right)^{\ell (\ell  + 1)}}\frac{{{d^{2\ell }}}}{{d{x^\ell }d{y^\ell }}} \bigg[ \nonumber \\
   &\quad \times \iint_{\substack{a + b \le 1 \\ a, b \ge 0}} {\int_0^1 {{u^{2\ell }}{{(1 - u)}^{2{\ell ^2} - 1}}{{(M_\ell ^{ - x}M_{\ell  + 1}^{au})}^{ - \alpha }}{{(M_\ell ^yM_{\ell  + 1}^{ - bu}T^2)}^{ - \beta }}} } \nonumber \\
   &\quad \times P_\ell ^{(\ell (\ell  - 1))}\left( {x + y + 1 - (1 - u)\frac{{{\nu _{\ell  + 1}}}}{{{\nu _\ell }}}} \right)P_{\ell  + 1}^{(\ell (\ell  + 1))}((1 - a - b)u){(ab)^{\ell  - 1}}dudadb \bigg]_{x = y = 0}. \nonumber  
\end{align}
Therefore, the main term of Theorem \ref{cellell1theorem} is given by
\begin{align}
  c_{\ell ,\ell  + 1} &= Q\left( {\frac{{ - 1}}{{2\log T}}\frac{d}{{d\alpha }}} \right)Q\left( {\frac{{ - 1}}{{2\log T}}\frac{d}{{d\beta }}} \right){c_{\ell ,\ell  + 1}}(\alpha ,\beta ){|_{\alpha  = \beta  =  - R/L}} \nonumber \\
   &= \frac{{{2^{2{\ell ^2}}}}}{{(2{\ell ^2} - 1)!}}{\left( {\frac{{{\nu _{\ell  + 1}}}}{{{\nu _\ell }}}} \right)^{\ell (\ell  + 1)}}{e^R}\frac{{{d^{2\ell }}}}{{d{x^\ell }d{y^\ell }}} \bigg[  \iint_{\substack{a + b \le 1 \\ a, b \ge 0}} \int_0^1 {{u^{2\ell }}{{(1 - u)}^{2{\ell ^2} - 1}}{e^{R[\frac{\nu _\ell }{2}(y - x) + u\frac{\nu _{\ell  + 1}}{2}(a - b)]}}}  \nonumber \\
   &\quad \times Q\bigg( \frac{- x{\nu _\ell } + au{\nu _{\ell  + 1}}}{2}\bigg)Q\bigg(1 +\frac{ y{\nu _\ell } - bu{\nu _{\ell  + 1}}}{2}\bigg) \nonumber \\
   &\quad \times P_\ell ^{(\ell (\ell  - 1))}\left( {x + y + 1 - (1 - u)\frac{{{\nu _{\ell  + 1}}}}{{{\nu _\ell }}}} \right)P_{\ell  + 1}^{(\ell (\ell  + 1))}((1 - a - b)u){(ab)^{\ell  - 1}}dudadb \bigg]_{x = y = 0}. \nonumber 
\end{align}
This ends the computation of $I_{\ell,\ell+1}$.
\subsection{The mean value integral $I_{\ell,\ell+j}(\alpha, \beta)$ for $j \in \N \backslash \{1\}$}
\noindent We must now examine the case $I_{\ell,\ell+j}(\alpha,\beta)$ where $j = 2,3,4,\cdots$. 
We will show that $I_{\ell,\ell+j}(\alpha,\beta) \ll TL^{-1+\varepsilon}$ and therefore the mean value integral contribution of these terms to $\kappa_f$ is zero. 
As before, we start by inserting
\[
{\psi _\ell }(s) = \chi _f^{\ell-1} (s + \tfrac{1}{2} - {\sigma _0})\sum\limits_{{h_1}{k_1} \le {M_\ell }} {\frac{{{\mu_{f,\ell} }({h_1}){\lambda_f^{* \ell - 1}}({k_1})h_1^{{\sigma _0} - 1/2}k_1^{1/2 - {\sigma _0}}}}{{h_1^sk_1^{1 - s}}}{P_\ell }[{h_1}{k_1}]} ,
\]
and
\[
{\psi _{\ell  + j}}(s) = \chi _f^{\ell-1  + j}(s + \tfrac{1}{2} - {\sigma _0})\sum\limits_{{h_2}{k_2} \le {M_{\ell  + j}}} {\frac{{{\mu _{\ell  + j}}({h_2}){\lambda_f^{* \ell-1+j}}({k_2})h_2^{{\sigma _0} - 1/2}k_2^{1/2 - {\sigma _0}}}}{{h_2^sk_2^{1 - s}}}{P_{\ell  + j}}[{h_2}{k_2}]} ,
\]
into the mean value integral $I$, i.e.
\begin{align}
  I_{\ell ,\ell  + j}(\alpha ,\beta ) &= \int_{ - \infty }^\infty  {w(t)L(f,\tfrac{1}{2} + \alpha  + it)L(f,\tfrac{1}{2} + \beta  - it)\overline {{\psi _\ell }} {\psi _{\ell  + j}}({\sigma _0} + it)dt}  \nonumber \\
   &= \sum\limits_{{h_1}{k_1} \le {M_\ell }} {\sum\limits_{{h_2}{k_2} \le {M_{\ell  + j}}} {\frac{{{\mu_{f,\ell} }({h_1}){\lambda_f^{* \ell - 1}}({k_1})}}{{h_1^{1/2}k_1^{1/2}}}{P_\ell }[{h_1}{k_1}]\frac{{{\mu _{\ell  + j}}({h_2}){\lambda_f^{* \ell-1+j}}({k_2})}}{{h_2^{1/2}k_2^{1/2}}}{P_{\ell  + j}}[{h_2}{k_2}]} } {J_{3,f}}, \nonumber 
\end{align}
where
\begin{align}
  {J_{3,f}} &= \int_{ - \infty }^\infty  {w(t)L(f,\tfrac{1}{2} + \alpha  + it)L(f,\tfrac{1}{2} + \beta  - it){{\left( {\frac{{{k_1}{h_2}}}{{{h_1}{k_2}}}} \right)}^{ - it}}\chi _f^j(\tfrac{1}{2} + it)dt}  \nonumber \\
   &= \int_{ - \infty }^\infty  w(t){{\bigg( {\frac{t\sqrt{N}}{{2\pi }}} \bigg)}^{ - 2\beta }}L(f,\tfrac{1}{2} + \alpha  + it)L(f,\tfrac{1}{2} - \beta  + it){{\left( {\frac{{{k_1}{h_2}}}{{{h_1}{k_2}}}} \right)}^{ - it}}\chi _f^{j-1}(\tfrac{1}{2} + it)dt  + O({T^\varepsilon }), \nonumber 
\end{align}
by the use of the functional equation of $L(f,s)$. 
We remark that so far this is the same procedure as in the $I_{\ell,\ell+1}$ case except for the presence of $\chi_f^{j-1}$ in the integrand, and in this case $j$ is an integer strictly greater than one. 
Inserting the duplication formula for the $\Gamma$ function we obtain
\[
\chi_f(s) = \frac{N^{\frac{s-1}{2}} (2\pi)^{s-1} \Gamma\left( 1-s + \frac{k-1}{2} \right)}{N^{-\frac{s}{2}} (2\pi)^{-s} \Gamma\left( s + \frac{k-1}{2} \right)}.
\]
One can show using Stirling's approximation formula that
\[
 (\chi_f(\tfrac{1}{2}+it))^{j-1} = F^{j-1}\bigg( 1+ \sum_{n=1}^{j} b_n t^{-n} +O(t^{-j-1})\bigg)
 \quad \textnormal{with} \quad F(t) = \bigg( \frac{\sqrt{N} t}{2 \pi e} \bigg)^{-2it},
\]
where $b_n$ are complex numbers depending only on $j$ and $k$, where $k$ was the weight of the cusp form, see $\mathsection$1.1. %~\ref{sec:cusp_form}
%and it is therefore a constant.
Let us handle the error term first. We have $E(t):= ( \frac{\sqrt{N} t}{2 \pi e} )^{-2it(j-1)}O(t^{-j-1}) = O(t^{-2})$.
By power moment estimates (see e.g. \cite[Corollary 2]{bernard} or \cite{zhang}) we have
\begin{align}
  \int_{ - \infty }^\infty  w(t)&{{\bigg( {\frac{t\sqrt{N}}{{2\pi }}} \bigg)}^{ - 2\beta }}L(f,\tfrac{1}{2} + \alpha  + it)L(f,\tfrac{1}{2} - \beta  + it){{\left( {\frac{{{k_1}{h_2}}}{{{h_1}{k_2}}}} \right)}^{ - it}}{E(t)}dt  \nonumber \\
	& \ll \frac{1}{T^{2}} \int_{T/4}^{2T} |L(f,\tfrac{1}{2} + \alpha  + it)| |L(f,\tfrac{1}{2} - \beta  + it)| dt\nonumber \\
	& \le \frac{1}{T^{2}} \bigg(\int_{T/4}^{2T} |L(f,\tfrac{1}{2} + \alpha  + it)|^2 dt \bigg) ^{1/2} \bigg(\int_{T/4}^{2T} |L(f,\tfrac{1}{2} - \beta  + it)|^2 dt \bigg)^{1/2} \nonumber \\
  & \ll \frac{1}{T^{2}} (T \log T)^{1/2} (T \log T)^{1/2} = \frac{\log T}{T}, \nonumber  
\end{align}
by the Cauchy-Schwarz inequality. We next consider the term 
\[
\int_{ - \infty }^\infty t^{-n} w(t){{\bigg( {\frac{t\sqrt{N}}{{2\pi }}} \bigg)}^{ - 2\beta }}L(f,\tfrac{1}{2} + \alpha  + it)L(f,\tfrac{1}{2} - \beta  + it){{\left( {\frac{{{k_1}{h_2}}}{{{h_1}{k_2}}}} \right)}^{ - it}}{F^{j - 1}}(\tfrac{1}{2} + it)dt .
\]
We now use Lemma~\ref{lemmasigma} together with the definition of $\sigma_{\alpha,-\beta}(f,l)$ to further rewrite this as
\begin{align*} 
&\sum\limits_{l = 1}^\infty  {\frac{{{\sigma _{\alpha , - \beta }}(f,l)}}{{{l^{1/2}}}}{e^{ - l/{T^6}}}
\int_{ - \infty }^\infty  {t^{-n}  w(t){{\bigg( {\frac{t \sqrt{N}}{{2\pi }}} \bigg)}^{ - 2\beta }}{{\left( {\frac{{{k_1}{h_2}l}}{{{h_1}{k_2}}}} \right)}^{ - it}}{F^{j - 1}}(\tfrac{1}{2} + it)dt} } \\
=&\,
\sum\limits_{l = 1}^\infty  \frac{\sigma _{\alpha , - \beta }(f,l)}{l^{1/2}} e^{ - l/T^6}
\int_{ - \infty }^\infty  t^{-n}  w(t){{\bigg( {\frac{t \sqrt{N}}{{2\pi }}} \bigg)}^{ - 2\beta}} {\bigg( {\frac{t \sqrt{N}}{{2\pi e }}} \bigg)}^{ -2(j-1)it } {{\left( {\frac{{{k_1}{h_2}l}}{{{h_1}{k_2}}}} \right)}^{ - it}}{dt} . 
\end{align*}
The key observation comes from noticing that for all $1 \le h_1, k_1 \le M_{\ell}$ and $1 \le h_2, k_2 \le M_{\ell+j}$ as well as for any $l \ge 1$, one has
\begin{align} \label{keyorthogonal}
\bigg( {\frac{t\sqrt{N}}{{2\pi e}}} \bigg)^{2(j - 1)}\frac{{{k_1}{h_2}l}}{{{h_1}{k_2}}} 
\ge 
\frac{{{T^{2(j - 1)}}}}{2^{2(j-1)}}\frac{1}{{ (2\pi e)^{2(j-1)}{M_\ell }{M_{\ell  + j}}}} 
= 
\frac{{{T^{2(j - 1)}}}}{{(4\pi e)^{2(j-1)}{T^{{\nu _\ell } + {\nu _{\ell  + j}}}}}} \ge {T^{{\varepsilon _0}}},
\end{align}
provided that
\begin{align} \label{keyorthogonal2}
\nu_\ell + \nu_{\ell + j} < 2(j - 1) - \varepsilon .
\end{align}
From the conditions of the test function $w$ we have $w^{(r)}(t) \ll (L/T)^r$. Using this, it is straight forward to show that we have for each $r\geq 1$ and $n\geq 0$
\begin{align} \label{keyorthogonal2.1}
 \frac{d^r}{dt^r} \bigg(t^{-n}w(t) \bigg(\frac{t \sqrt{N}}{2\pi } \bigg)^{- 2\beta }\bigg) 
 \ll_{r,j} 1
 \quad \textnormal{for} \quad T/2\leq t\leq 2T.
\end{align}
Hence, picking up from \eqref{keyorthogonal} and \eqref{keyorthogonal2.1} and using integration by parts, we arrive at
\begin{align} \label{keyorthogonal3}
\int_{ - \infty }^\infty \bigg( t^{-n}w(t){{\bigg( {\frac{t\sqrt{N}}{{2\pi }}} \bigg)}^{ - 2\beta }} \bigg) 
 {\bigg( {{{\bigg( {\frac{t\sqrt{N}}{{2\pi e}}} \bigg)}^{2(j - 1)}}\frac{{{k_1}{h_2}l}}{{{h_1}{k_2}}}} \bigg)}^{ - it} dt  \ll _{r,{\varepsilon _0}} \frac{1}{{{T^r}}},
\end{align}
for any fixed integer $r$ and uniformly in $l$. Thus, using \eqref{keyorthogonal3} we can further bound $J_{3,f}$ as
\[
J_{3,f} \ll_{r,{\varepsilon _0}} \frac{1}{{{T^r}}}\sum\limits_{l = 1}^\infty  {\frac{{{\sigma _{\alpha , - \beta }}(f,l)}}{{{l^{1/2}}}}{e^{ - l/{T^6}}}}  
+ O_{\varepsilon ,f}({T^{{-1}+\varepsilon} }) \ll_{\varepsilon ,{\varepsilon _0}}{T^{{-1}+\varepsilon} }.
\]
The last step is to plug this back into $I_{\ell,\ell+j}(\alpha,\beta)$ to see that
\begin{align}
  I_{\ell ,\ell  + j}(\alpha ,\beta ) &\ll_{\varepsilon ,{\varepsilon _0}}
  {T^{{-1}+\varepsilon} }
  \sum\limits_{{h_1}{k_1} \le {M_\ell }} {\sum\limits_{{h_2}{k_2} \le {M_{\ell  + j}}} {\frac{{|{\mu_{f,\ell} }({h_1}){\mu _{\ell  + j}}({h_2}){\lambda_f^{* \ell - 1}}({k_1}){\lambda_f^{* \ell-1+j}}({k_2})|}}{{{{({h_1}{h_2}{k_1}{k_2})}^{1/2}}}}|{P_\ell }[{h_1}{k_1}]{P_{\ell  + j}}[{h_2}{k_2}]|} }  \nonumber \\
  &\ll _{\ell,\varepsilon ,{\varepsilon _0}} {T^{{-1}+2\varepsilon }}
    \sum\limits_{{h_1}{k_1} \le {M_\ell }} \sum\limits_{{h_2}{k_2} \le {M_{\ell  + j}}} (h_1h_2 k_1k_2)^{\bm{\theta}-1/2 + \varepsilon} ,  \nonumber   
\end{align}
by the use of \eqref{criticalbounds}. Recognizing that the sums can be consolidated by employing the divisor function we obtain
\begin{align}
  I_{\ell ,\ell  + j}(\alpha ,\beta ) &\ll _{\ell,\varepsilon ,{\varepsilon _0}} {T^{{-1}+2\varepsilon }}
  \sum\limits_{n \le {M_\ell }} \tau(n) n^{\bm{\theta}-1/2+\varepsilon} \sum\limits_{m \le {M_{\ell  + j}}} \tau(m) m^{\bm{\theta}-1/2+\varepsilon}  
   \nonumber \\
   &\ll _{\ell,\varepsilon ,{\varepsilon _0}} {T^{{-1}+4\varepsilon }}M_\ell ^{\bm{\theta}+1/2 + \varepsilon}M_{\ell  + j}^{\bm{\theta}+1/2 + \varepsilon} = T^{{-1}+6\varepsilon  + ({\nu _\ell } + {\nu _{\ell  + j}})(\bm{\theta}+1/2)}. 
\end{align}
% Since $\nu _\ell  + \nu_{\ell  + j}< 2$, 
By our assumptions on $\nu _\ell$ and $\nu_{\ell  + j}$, we can choose $\varepsilon$ so small that $I_{\ell ,\ell  + j}(\alpha ,\beta ) \ll T/L$.
This ends the proof of Lemma \ref{cellelljlemma} and Theorem \ref{cellelljtheorem}.
%%%%%%%%%%%%%%%%%%%%%%%%%%%%%%%%%%%%%%%%%%%%%%%%%%%%%%%%%%%%%%%%%%%%%%%%%%%%%%%%%%%%%%%%%%%%%%%%%%%%%%%
\section{Numerical evidence and situation of simple zeros}
In this section we supply the numerical procedure to obtain Theorem \ref{theoremofkappa} and explain the situation regarding the simple zeros and the $\nu_1, \nu_2 \to 1$ conjecture of Farmer.\\

As shown in \cite[p. 230]{bernard} and \cite[p. 216]{farmer}, 
\[
\frac{1}{T} \int_1^T |V\psi(\sigma_0+it)|^2 dt
\]
produces the constants $c_{1,1}, c_{1,2}$ and $c_{2,2}$ with $2R$ and $\nu_{1}/2, \nu_2/2$. We use \texttt{Mathematica} to numerically evaluate $c_{1,1}(P_{\mathcal{L}},Q,2R,\nu_1/2)$ and $c_{1,2}(P_{\mathcal{L}},Q,2R,\nu_1/2,\nu_2/2)$ with the following parameters:
\begin{align}
%\mathcal{L} & = 2, \nonumber \\
\bm{\theta}  = 0, \quad \textnormal{so that} \quad \nu_1 = \nu_2 = \tfrac{1}{4}, \quad \textnormal{as well as} \quad R = 2.82505, \nonumber
\end{align}
as well as polynomials
\begin{align}
Q(x) &= .498939 + 1.53685(1-2x) -2.7925(1-2x)^3 + 2.77524 (1-2x)^5 -1.01853(1-2x)^7, \nonumber \\
P_1(x) &= .921756 x + .150879x^2 - .371912 x^3 + .488862x^4 - .189585 x^5, \nonumber \\
P_2(x) &= -.0000537029 x^3 + .0000752763 x^4  -.000142568 x^5. \nonumber
\end{align}
This leads to $\kappa_f \ge .0693872$.
Note that we have used for this result that the Ramanujan's hypothesis is proven for primitive cusp forms (see \eqref{kimsarnak} and \eqref{ramanujanhypothesis}).
We therefore could use $\bm{\theta}=0$ in the above computation. If we work instead with the weaker $\bm{\theta} =7/64$, obtained by Kim and Sarnak, then we use the following choices of parameters: 
\begin{align}
%\mathcal{L} &= 2, \nonumber \\
\bm{\theta}  = \tfrac{7}{64}, \quad \textnormal{so that} \quad \nu_1 = \tfrac{5}{27} \quad \textnormal{and} \quad \nu_2 = \tfrac{25}{149}, \quad \textnormal{as well as} \quad R = 3.21, \nonumber
\end{align}
and for the polynomials we take 
\begin{align}
Q(x) &= .499386 + 1.58992(1-2x) -2.99061(1-2x)^3 + 3.01825(1-2x)^5 -1.11694(1-2x)^7, \nonumber \\
P_1(x) &= .93271 x + .147723 x^2 - .35572 x^3 + .444208 x^4 - .168921 x^5, \nonumber \\
P_2(x) &= -.0000665503 x^3 - .00016405 x^4 + .0000736009 x^5. \nonumber
\end{align}
This gives $\kappa_f \ge .0297607$. We therefore see that the Ramanujan's hypothesis has a significant influence to the result obtained by the methods of this paper.\\

Finally, as described\color[rgb]{0,0,0} in the introduction, we explain what happens if we are in the context of the Riemann zeta-function and adapt our results mutatis mutandis with $\nu_1, \nu_2 \to 1$ following Farmer's conjecture \cite{farmer}. In that case, if we take $R=0.75$ and
\begin{align}
Q(x) &= .521417 + .488276(1-2x) - .0155446(1-2x)^3 + .00683032(1-2x)^5 - .0320679(1-2x)^7, \nonumber \\
P_1(x) &= .702374 x + .00612233 x^2 + .281569 x^3 + .296314 x^4 - .286379 x^5, \nonumber \\
P_2(x) &= .0690439 x^3 - .0187972 x^4 + .0319485 x^5. \nonumber
\end{align}
then we get $\kappa \ge .60563$, where $\kappa$ is the lim inf of $N_0(T)/N(T)$ as $T \to \infty$.\\

As mentioned earlier, from numerical experiments, one would need a size of about $\nu=2/5$ to get a proportion of simple critical zeros (recalling that we use $\nu/2$ in the constants $c$). The following plots will illustrate this phenomenon. If we take only $\mathcal{L}=1$, and for the sake of simplicity $P(x)=x$ and $Q(x)=1-x$ (in fact, having $Q(x)$ be a polynomial of degree one is necessary to obtain simple critical zeros, see \cite{anderson} and \cite{heathbrown} as well as \cite[p. 37]{bcy} and \cite[p. 513]{feng}), then
\begin{align} \label{kappaexample}
  c_{1,1} &= 1 + \frac{1}{\nu _1}\int_0^1 \int_0^1 e^{2Rv} \left( {\frac{d}{{dx}}{e^{R\nu _1 x}}P(x + u)Q(v + \nu _1 x){|_{x = 0}}} \right)^2dudv   \nonumber \\
   &= 1 - \frac{{3 + 6R - 2{R^3}(\nu _1  - 3)\nu_1  + 2{R^4}{\nu _1 ^2} + {R^2}(6 + {\nu _1 ^2}) - {e^{2R}}(3 + {R^2}{\nu _1 ^2})}}{{12{R^3}\nu _1 }}, 
\end{align}
and 
\[
\kappa  \ge 1 - \frac{1}{R}\log c_{1,1} + o(1).
\]
Let us now graph $\widehat{\kappa} = \widehat{\kappa}(R,\nu_1) =$ RHS of \eqref{kappaexample} to see the proportion of simple zeros as a function of $R$ and of $\nu_1$.
\begin{figure}[H]
	\centering
	\includegraphics[scale=0.89]{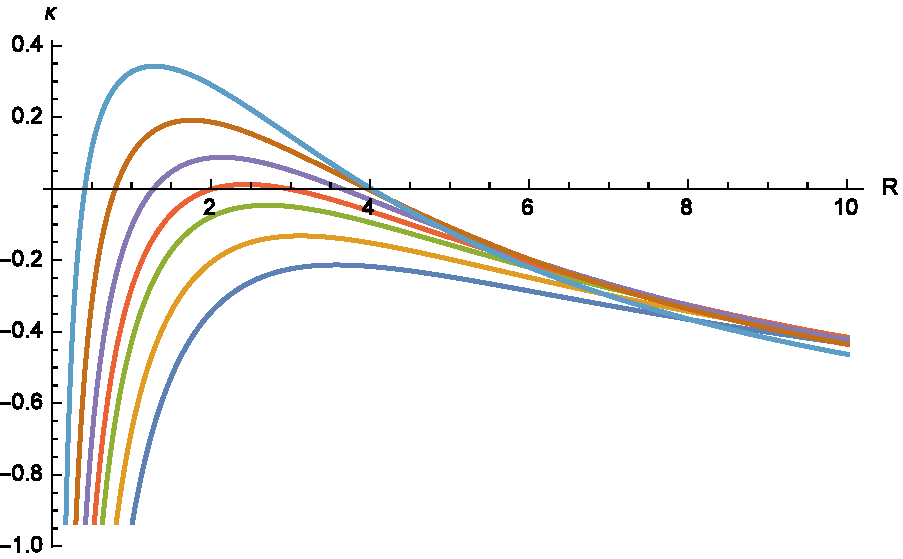}
	\hspace{5mm}
	\includegraphics[scale=0.80]{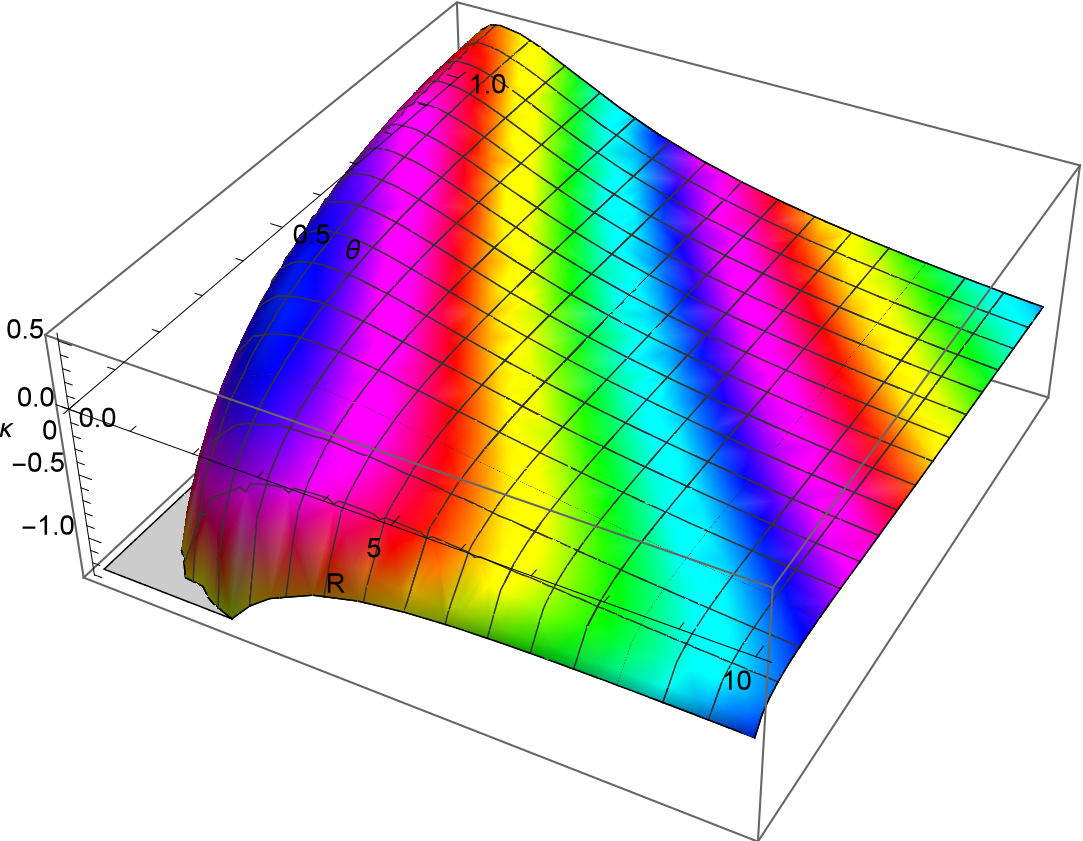}
	\caption{Left-hand side: $\widehat{\kappa}(R,\nu_1)$ for $\nu_1=1/2$ (light blue), $\nu_1=1/3$ (brown), $\nu_1=1/4$ (purple), $\nu_1=1/5$ (red), $\nu_1=1/6$ (green), $\nu_1=1/8$ (orange), $\nu_1=5/54$ (dark blue). Right-hand side: surface plot of $\widehat{\kappa}(R,\nu_1)$ with $\tfrac{1}{10} \le R \le 10$ and $\tfrac{1}{100} \le \theta \le 1$.}
\end{figure}
We note that when $\nu _1 = \frac{1}{2}$, this was considered in \cite{youngsimple} and the optimal value is $R \approx 1.3$. 
%%%%%%%%%%%%%%%%%%%%%%%%%%%%%%%%%%%%%%%%%%%%%%%%%%%%%%%%%%%%%%%%%%%%%%%%%%%%%%%%%%%%%%%%%%%%%%%%%%%%%%%
\section{Acknowledgments}
The first author wishes to acknowledge partial support from SNF grant PP00P2 138906.\\
The second author wishes to thank Keiju Sono for a cordial correspondence while working on similar results. Sono's results in \cite{sono} for the Riemann zeta-function overlap with our computations and these were produced independently of ours.\\
The authors are extremely grateful to the anonymous referees for their comments and suggestions. Their corrections have removed inaccuracies and greatly increased the clarity of the manuscript.
%%%%%%%%%%%%%%%%%%%%%%%%%%%%%%%%%%%%%%%%%%%%%%%%%%%%%%%%%%%%%%%%%%%%%%%%%%%%%%%%%%%%%%%%%%%%%%%%%%%%%%%

\end{document}